\documentclass[a4paper,oneside,reqno]{amsart}

\usepackage[T1]{fontenc}
\usepackage[utf8]{inputenc}
\usepackage[english]{babel}
\usepackage[osf,sc]{mathpazo}
\usepackage{eucal}
\normalfont

\usepackage{fullpage}
\usepackage{tabularx,booktabs,multirow}
\usepackage{caption,subcaption}
\usepackage{wrapfig}
\usepackage[dvipsnames]{xcolor}
\usepackage{hyperref}
\usepackage{enumitem}
\usepackage{booktabs}

\hypersetup{
	colorlinks=true,
	linktocpage=true,
	pdfstartpage=1,
	pdfstartview=FitV,
	breaklinks=true,
	pdfpagemode=UseNone,
	pageanchor=true,
	pdfpagemode=UseOutlines,
	plainpages=false,
	bookmarksnumbered,
	bookmarksopen=true,
	bookmarksopenlevel=1,
	hypertexnames=false,
	pdfhighlight=/O,
	urlcolor=BrickRed,
	linkcolor=RoyalBlue,
	citecolor=ForestGreen
}

\usepackage{graphicx}
\usepackage{tikz,tikz-cd}
\usetikzlibrary{calc,decorations.markings}

\usepackage[textsize=footnotesize]{todonotes}
\presetkeys{todonotes}{color=Apricot}{}
\usepackage{amsmath,amssymb,amsthm}
\usepackage{mathtools,thmtools} 
\usepackage{bm,braket}
\numberwithin{equation}{section}

\usepackage[noabbrev]{cleveref}
\AddToHook{cmd/appendix/before}{\crefalias{section}{appendix}}
\crefname{lemma}{lemma}{lemmata}
\Crefname{lemma}{Lemma}{Lemmata}
\crefname{subsection}{subsection}{subsections}
\Crefname{subsection}{Subsection}{Subsections}
\crefname{conjecture}{conjecture}{conjectures}
\Crefname{conjecture}{Conjecture}{Conjectures}
\crefname{properties}{properties}{Properties}
\Crefname{Properties}{Properties}{Properties}

\newcommand{\mb}[1]{\mathbb{#1}}
\newcommand{\mc}[1]{\mathcal{#1}}
\newcommand{\mf}[1]{\mathfrak{#1}}
\newcommand{\ms}[1]{\mathsf{#1}}

\newcommand{\MG}{\mathsf{M\ddot{o}G}}
\newcommand{\Mod}{\mathcal{N}}
\newcommand{\ZMod}{\mathcal{N}^{\Z}}
\newcommand{\Aut}{\mathrm{Aut}}
\newcommand{\Att}{\mathrm{AT}}

\newcommand{\Z}{\mathbb{Z}}
\newcommand{\Q}{\mathbb{Q}}
\newcommand{\R}{\mathbb{R}}
\newcommand{\C}{\mathbb{C}}
\renewcommand{\P}{\mathbb{P}}

\newcommand{\codim}{\mathrm{codim}}
\newcommand{\Rec}{\mathrm{Rec}}
\newcommand{\tr}{\mathrm{tr}}
\DeclareMathOperator*{\Res}{Res}

\theoremstyle{plain}
\newtheorem{theorem}{Theorem}[section]
\newtheorem{proposition}[theorem]{Proposition}

\newtheorem{lemma}[theorem]{Lemma}

\theoremstyle{definition}
\newtheorem{definition}[theorem]{Definition}
\newtheorem{remark}[theorem]{Remark}
\newtheorem{properties}[theorem]{Properties}
\newtheorem{example}[theorem]{Example}

\theoremstyle{plain}
\newtheorem{introthm}{Theorem}

\usepackage[
	style=alphabetic,
	sorting=nyt,
	sortcites=true,
	sortcites=ynt,
	maxbibnames=99,
	maxalphanames=4,
	giveninits=true,
	backend=bibtex
]{biblatex}
\renewbibmacro{in:}{}
\usepackage{csquotes}
\begin{document}

\title{Refined lattice point counting on the moduli space of Klein surfaces}

\author[N. K. Chidambaram]{N. K. Chidambaram}
\address[N. K. Chidambaram]{Departamento de Matemáticas Fundamentales, UNED, Madrid, Spain}
\email{nitin.chidambaram@mat.uned.es}

\author[E. Garcia-Failde]{E. Garcia-Failde}
\address[E. Garcia-Failde]{Departament de Matemàtiqes, Universitat Politècnica de Catalunya, Barcelona, Spain}
\email{elba.garcia@upc.edu}

\author[A. Giacchetto]{A. Giacchetto}
\address[A. Giacchetto]{Department Mathematik, ETH Z\"urich, Z\"urich, Switzerland}
\email{alessandro.giacchetto@math.ethz.ch}

\author[K. Osuga]{K. Osuga}
\address[K. Osuga]{Kobayashi--Masakawa Institute for the Origin of Particles and the Universe \& Graduate School of Mathematics, Nagoya University, Nagoya, Japan}
\email{osuga@math.nagoya-u.ac.jp}

\begin{abstract}
	We introduce the moduli space of metric M\"obius graphs, which extend ribbon graphs to the non-orientable world. This space contains both the moduli space of Riemann surfaces and the moduli space of non-orientable Klein surfaces. Each metric M\"obius graph is equipped with a measure of non-orientability. We count lattice points in this moduli space, weighted by the measure of non-orientability, and prove a refined version of Norbury's recursion for this count. Taking the limit as the mesh becomes finer, we deduce a recursion for the Euclidean volumes, yielding a refined version of the Witten--Kontsevich recursion. As an application, we give a geometric definition of the refined Euler characteristic of the moduli space and compute it explicitly, thereby answering a question of Goulden, Harer, and Jackson.
\end{abstract}

\maketitle
\tableofcontents

\newpage
\section{Introduction}
\label{sec:intro}

\subsection{Moduli, measure of non-orientability, and refined lattice point count}
Fix \smash{$g \in \frac{1}{2}\Z_{\ge 0}$} and \smash{$n \in \Z_{>0}$} satisfying $2g-2+n > 0$. Let $\Mod_{g,n}(L)$ be the moduli space of metric, possibly non-orientable ribbon graphs (which we call \emph{M\"obius graphs}) of genus $g$ with $n$ labelled boundaries of lengths $L = (L_1,\ldots,L_n) \in \R_{>0}^n$. By definition, it admits a cell decomposition
\begin{equation}
	\Mod_{g,n}(L)
	\coloneqq
	\Biggl(
		\bigsqcup_{G \in \MG_{g,n}} \frac{P_G(L)}{\Aut(G)}
	\Biggr)\Bigg/\!\!\sim\, ,
\end{equation}
where $\MG_{g,n}$ is the finite set of M\"obius graphs of genus $g$ with $n$ faces, the identification is performed via edge degeneration, and $P_G(L)$ is the space of metrics on $G$ with boundaries $L$, a polytope in \smash{$\R_{>0}^{E(G)}$} defined by the adjacency matrix. This endows $\Mod_{g,n}(L)$ with the structure of a real orbifold of dimension $6g-6+2n$.

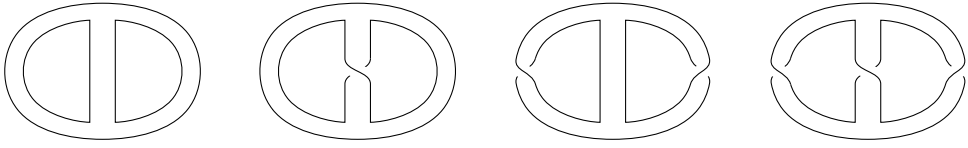
\begin{figure}[hb]
	\begin{tikzpicture}[x=1pt,y=1pt,scale=.6]
		\draw (-5.3333, 712) .. controls (5.3333, 698.6667) and (26.6667, 693.3333) .. (48, 693.3333) .. controls (69.3333, 693.3333) and (90.6667, 698.6667) .. (101.3333, 712) .. controls (112, 725.3333) and (112, 746.6667) .. (101.3333, 760) .. controls (90.6667, 773.3333) and (69.3333, 778.6667) .. (48, 778.6667) .. controls (26.6667, 778.6667) and (5.3333, 773.3333) .. (-5.3333, 760) .. controls (-16, 746.6667) and (-16, 725.3333) .. cycle;
		\draw (40, 768) .. controls (30.765, 767.814) and (13.4566, 763.838) .. (4.8023, 753.8982) .. controls (-3.8519, 743.9583) and (-3.8519, 728.0547) .. (4.8023, 718.1148) .. controls (13.4566, 708.175) and (30.765, 704.199) .. (40, 704) -- cycle;
		\draw (56, 704) .. controls (65.383, 704.199) and (82.6915, 708.175) .. (91.3458, 718.1148) .. controls (100, 728.0547) and (100, 743.9583) .. (91.3458, 753.8982) .. controls (82.6915, 763.838) and (65.383, 767.814) .. (56, 768) -- cycle;
		\draw (200, 728) .. controls (200, 736) and (216, 736) .. (216, 744);
		\draw [white, line width=2mm](216, 728) .. controls (216, 736) and (200, 736) .. (200, 744);
		\draw (154.6667, 712) .. controls (165.3333, 698.6667) and (186.6667, 693.3333) .. (208, 693.3333) .. controls (229.3333, 693.3333) and (250.6667, 698.6667) .. (261.3333, 712) .. controls (272, 725.3333) and (272, 746.6667) .. (261.3333, 760) .. controls (250.6667, 773.3333) and (229.3333, 778.6667) .. (208, 778.6667) .. controls (186.6667, 778.6667) and (165.3333, 773.3333) .. (154.6667, 760) .. controls (144, 746.6667) and (144, 725.3333) .. cycle;
		\draw (200, 728) -- (200, 704) .. controls (190.765, 704.199) and (173.4566, 708.175) .. (164.8023, 718.1148) .. controls (156.1481, 728.0547) and (156.1481, 743.9583) .. (164.8023, 753.8982) .. controls (173.4566, 763.838) and (190.765, 767.814) .. (200, 768) -- (200, 744) .. controls (200, 736) and (216, 736) .. (216, 728) -- (216, 704) .. controls (225.383, 704.199) and (242.6915, 708.175) .. (251.3458, 718.1148) .. controls (260, 728.0547) and (260, 743.9583) .. (251.3458, 753.8982) .. controls (242.6915, 763.838) and (225.383, 767.814) .. (216, 768) -- (216, 744);
		\draw (416.641, 744.208) .. controls (419.603, 736.035) and (429.841, 736.01) .. (428.66, 728.504) .. controls (427.561, 722.472) and (425.119, 716.732) .. (421.333, 712) .. controls (410.667, 698.667) and (389.333, 693.333) .. (368, 693.333) .. controls (346.667, 693.333) and (325.333, 698.667) .. (314.667, 712) .. controls (310.881, 716.732) and (308.439, 722.472) .. (307.34, 728.504) .. controls (306.123, 736.01) and (316.999, 736.038) .. (319.494, 744.163);
		\draw [white, line width=2mm](376, 704) .. controls (385.383, 704.199) and (402.692, 708.175) .. (411.346, 718.115) .. controls (413.833, 720.971) and (415.605, 724.32) .. (416.663, 727.878) .. controls (419.149, 735.928) and (429.779, 736.036) .. (428.64, 743.602) .. controls (427.533, 749.595) and (425.097, 755.296) .. (421.333, 760) .. controls (410.667, 773.333) and (389.333, 778.667) .. (368, 778.667) .. controls (346.667, 778.667) and (325.333, 773.333) .. (314.667, 760) .. controls (310.903, 755.296) and (308.467, 749.595) .. (307.36, 743.602);
		\draw [white, line width=2mm](319.478, 727.902) .. controls (316.999, 735.995) and (305.929, 735.984) .. (307.36, 743.602);
		\draw (416.641, 744.208) .. controls (415.581, 747.739) and (413.816, 751.061) .. (411.346, 753.898) .. controls (402.692, 763.838) and (385.383, 767.814) .. (376, 768) -- (376, 704) .. controls (385.383, 704.199) and (402.692, 708.175) .. (411.346, 718.115) .. controls (413.833, 720.971) and (415.605, 724.32) .. (416.663, 727.878) .. controls (419.149, 735.928) and (429.779, 736.036) .. (428.64, 743.602) .. controls (427.533, 749.595) and (425.097, 755.296) .. (421.333, 760) .. controls (410.667, 773.333) and (389.333, 778.667) .. (368, 778.667) .. controls (346.667, 778.667) and (325.333, 773.333) .. (314.667, 760) .. controls (310.903, 755.296) and (308.467, 749.595) .. (307.36, 743.602) .. controls (305.929, 735.984) and (316.999, 735.995) .. (319.478, 727.902) .. controls (320.535, 724.335) and (322.31, 720.978) .. (324.802, 718.115) .. controls (333.457, 708.175) and (350.765, 704.199) .. (360, 704) -- (360, 768) .. controls (350.765, 767.814) and (333.457, 763.838) .. (324.802, 753.898) .. controls (322.322, 751.049) and (320.553, 747.711) .. (319.494, 744.163);
		\draw (479.478, 727.902) .. controls (480.535, 724.335) and (482.31, 720.978) .. (484.802, 718.115) .. controls (493.457, 708.175) and (510.765, 704.199) .. (520, 704) -- (520, 728) .. controls (520, 736) and (536, 736) .. (536, 744) -- (536, 768) .. controls (545.383, 767.814) and (562.692, 763.838) .. (571.346, 753.898) .. controls (573.816, 751.061) and (575.581, 747.739) .. (576.641, 744.208) .. controls (579.603, 736.035) and (589.841, 736.01) .. (588.66, 728.504) .. controls (587.561, 722.472) and (585.119, 716.732) .. (581.333, 712) .. controls (570.667, 698.667) and (549.333, 693.333) .. (528, 693.333) .. controls (506.667, 693.333) and (485.333, 698.667) .. (474.667, 712) .. controls (470.881, 716.732) and (468.439, 722.472) .. (467.34, 728.504) .. controls (466.123, 736.01) and (476.999, 736.038) .. (479.494, 744.163);
		\draw [white, line width=2mm](536, 704) .. controls (545.383, 704.199) and (562.692, 708.175) .. (571.346, 718.115) .. controls (573.833, 720.971) and (575.605, 724.32) .. (576.663, 727.878) .. controls (579.149, 735.928) and (589.779, 736.036) .. (588.64, 743.602) .. controls (587.533, 749.595) and (585.097, 755.296) .. (581.333, 760) .. controls (570.667, 773.333) and (549.333, 778.667) .. (528, 778.667) .. controls (506.667, 778.667) and (485.333, 773.333) .. (474.667, 760) .. controls (470.903, 755.296) and (468.467, 749.595) .. (467.36, 743.602);
		\draw [white, line width=2mm](479.478, 727.902) .. controls (476.999, 735.995) and (465.929, 735.984) .. (467.36, 743.602);
		\draw [white, line width=2mm](536, 728) .. controls (536, 736) and (520, 736) .. (520, 744);
		\draw (479.478, 727.902) .. controls (476.999, 735.995) and (465.929, 735.984) .. (467.36, 743.602) .. controls (468.467, 749.595) and (470.903, 755.296) .. (474.667, 760) .. controls (485.333, 773.333) and (506.667, 778.667) .. (528, 778.667) .. controls (549.333, 778.667) and (570.667, 773.333) .. (581.333, 760) .. controls (585.097, 755.296) and (587.533, 749.595) .. (588.64, 743.602) .. controls (589.779, 736.036) and (579.149, 735.928) .. (576.663, 727.878) .. controls (575.605, 724.32) and (573.833, 720.971) .. (571.346, 718.115) .. controls (562.692, 708.175) and (545.383, 704.199) .. (536, 704) -- (536, 728) .. controls (536, 736) and (520, 736) .. (520, 744) -- (520, 768) .. controls (510.765, 767.814) and (493.457, 763.838) .. (484.802, 753.898) .. controls (482.322, 751.049) and (480.553, 747.711) .. (479.494, 744.163);
	\end{tikzpicture}
	\caption{Four M\"obius graphs. From left to right, they have the topology $(g,n)$ of: a pair of pants $(0,3)$, a crossed-capped pair of pants $(\frac{1}{2},2)$, a one-holed Klein bottle $(1,1)$, and a one-holed torus $(1,1)$.}
	\label{fig:Mobius}
\end{figure}

For integer genus $g$, the moduli space $\Mod_{g,n}(L)$ contains a connected component given by the moduli space of orientable metric ribbon graphs, which is isomorphic to the moduli space of curves/Riemann surfaces\footnote{
	Up to a global quotient by $\Z_2$, accounting for orientation-reversing morphisms.
}. The moduli space of metric ribbon graphs plays a crucial role in the work of Harer--Zagier \cite{HZ86} and Kontsevich \cite{Kon92} in the computation of the Euler characteristic of the moduli space of curves and $\psi$-class intersection numbers, respectively. On the other hand, $\Mod_{g,n}(L)$ also contains a connected component corresponding to the moduli space of real curves/non-orientable Klein surfaces\footnote{
	Again up to a global quotient by $\Z_2^n$, accounting for the choice of local orientations of faces.
}: namely, the component in which the fixed locus of the anti-holomorphic involution is empty, or equivalently in which the Klein surfaces have no boundary \cite{GHJ01}.

Inspired by Chapuy--Do\l\k{e}ga \cite{CD22} and previous results in \cite{LaC09,DFS13,Dol17}, motivated in turn by the Goulden--Jackson $b$-conjecture for non-orientable branched coverings \cite{GJ96}, we define on each cell $P_G(L)$ a \emph{measure of non-orientability} $\rho_G(\ell;b)$. This is a polynomial in $b$ that quantifies how non-orientable the M\"obius graph $G$ and its metric $\ell$ are. Key properties of $\rho$ are:
\begin{itemize}
	\item
	For $b=0$, $\rho_G$ detects orientability: it is identically $1$ if $G$ is orientable, and $0$ otherwise.
	
	\item
	For $b=1$, $\rho_G$ is identically $1$.
\end{itemize}
Thus, for general $b$, $\rho_G$ provides a statistic measuring the degree of non-orientability of a metric on $G$. One may view $b$ as an interpolation parameter between the orientable and non-orientable sectors of the combinatorial moduli space; see \cite{BCD23,Ruz23,CDO,CDO26} for further results on the enumeration of graphs on non-orientable surfaces weighted by measures of non-orientability. 

When the boundary lengths $L_i$ are positive integers, the polytope $P_G(L)$ is integral, so it is natural to consider its lattice points $P_G^{\Z}(L) \coloneqq P_G(L) \cap \Z^{E(G)}$. The count of lattice points of the moduli space $\mc{N}_{g,n}(L)$, weighted by the measure of non-orientability, is the central quantity studied in this paper:
\begin{equation}
	N_{g,n}(L;b)
	\coloneqq
	\sum_{G \in \MG_{g,n}}
		\frac{1}{|\Aut(G)|}
		\sum_{\ell \in P_G^{\Z}(L)}
			\rho_G(\ell;b) .
\end{equation}

We call $N_{g,n}$ the \emph{refined lattice point count}. The case $b = 0$ recovers Norbury's lattice point count on the moduli space of metric ribbon graphs \cite{Nor10}, divided by two. The factor of half is explained by the discrepancy on the orbifold structure due to the presence of orientation-reversing morphisms. The case $b = 1$ gives the unweighted lattice point count on the moduli space of metric Möbius graphs. Thus, the refined count interpolates between these two extremes.

Our first result shows that the refined lattice point count is a piecewise quasipolynomial, as one would expect for lattice point counts in parametric polytopes. This conclusion is not automatic in our setting as the measure of non-orientability is a rational function of the edge lengths, and weighted lattice point counts with rational weights need not be piecewise quasipolynomial in general.

\begin{introthm}[Polynomial properties]\label{thm:polynomiality}
	The refined lattice point count $N_{g,n}(L;b)$ is a symmetric, rational, continuous, piecewise quasipolynomial function of period $2$ and degree $6g-6+2n$ in the boundary lengths $L=(L_1,\ldots,L_n)$. The walls are given by the equations
	\begin{equation}
		\sum_{i=1}^n \epsilon_i L_i = 0,
		\qquad
		\epsilon_i \in \set{+1,-1,0}.
	\end{equation}
	Moreover, $N_{g,n}(L;b)$ is a polynomial in $b$ of degree at most $2g$.
\end{introthm}

\begin{table}[b]
	\centering
	\begin{tabularx}{\textwidth}{c c c X}
		\toprule
		$g$ & $n$ & $k$ & $N^{[k]}_{g,n}(L_1,\ldots,L_n;b)$ \\
		\midrule
		$0$ & $3$
			& $0,2$ & $\frac{1}{2}$ \\
		\midrule
		$\tfrac{1}{2}$ & $2$
			& $0,2$ & $\frac{b}{4}\bigl(\max(L_1,L_2)-1\bigr)$ \\
		\midrule
		$1$ & $1$ & $0$ & $\frac{1}{96} \bigl( (1+b)(L_1^2-4) + b^2(5L_1^2-12L_1+4) \bigr)$ \\
		\midrule
		\multirow{2}{*}{$0$} & \multirow{2}{*}{$4$}
			& $0,4$ & $\frac{1}{8} \bigl(\sum L_i^2 - 4 \bigr)$ \\
		&	& $2$ & $\frac{1}{8} \bigl(\sum L_i^2 - 2 \bigr)$ \\
		\midrule
		\multirow[c]{2}{*}{$\tfrac{1}{2}$} & \multirow[c]{2}{*}{$3$}
			& $0$ & $\frac{b}{96} \bigl( \Delta^3 - 4\Delta + 3(\sum L_i-2) (\sum L_i^2-4) -6\prod L_i \bigr)$ \\
		&	& $2$ & $\frac{b}{96} \bigl( \Delta^3 - 4\Delta + 3(\sum L_i-2) (\sum L_i^2-4) -6\prod L_i + 6(L_1+L_2-2) \bigr)$ \\
		\midrule
		\multirow{4}{*}{$1$} & \multirow{4}{*}{$2$}
			& $0$ & $ \frac{b^{2}}{1536}\big( \Delta^4 + 4\Delta^3(\sum L_i-2) + 2\Delta^2(\sum L_i^2+2\prod L_i-6\sum L_i+6) - 16\Delta(\sum L_i-2)$ \\
		& & & $\qquad + (\sum L_i-4)(\sum L_i-2)(3\sum L_i^2+6\prod L_i+6\sum L_i-8) \big) + \frac{b+1}{768}(\sum L_i^2 - 8)(\sum L_i^2 - 4)
		$ \\
		& & $2$ & $\frac{b^{2}}{1536}\big( \Delta^4 + 4\Delta^3(\sum L_i-2) + 2\Delta^2(\sum L_i^2+2\prod L_i-6\sum L_i) - 16\Delta(\sum L_i-2) + 12(\sum L_i - 2)^2$ \\
		& & & $\qquad + (\sum L_i-4)(\sum L_i-2)(3\sum L_i^2+6\prod L_i+6\sum L_i-8) \big) + \frac{b+1}{768}(\sum L_i^2 - 10)(\sum L_i^2 - 2)$ \\
		\midrule
		$\tfrac{3}{2}$ & $1$
			& $0$ & $\frac{b}{46080} (L_1^2-4) (L_1-4) \bigl( (1+b)(17 L_1^2 + 38 L_1-60) + 30 b^2 (L_1^2-L_1) \bigr)$ \\
		\bottomrule
	\end{tabularx}
	\caption{
		The piecewise polynomials $N^{[k]}_{g,n}$ for $2g-2+n \le 2$. Here $\Delta$ is the piecewise linear function $\max(2L_1 - \sum L_i,\ldots,2L_n - \sum L_i,0)$.
	}
	\label{tab:Ngnk}
\end{table}

Being a continuous, piecewise quasipolynomial of period $2$ means that, once the parities of the $L_i$ are fixed in $\Z_2$, the function $N_{g,n}(L;b)$ is given by a polynomial in each chamber. Continuity means that the polynomials attached to adjacent chambers agree along the walls. Moreover, since $N_{g,n}(L;b)$ is symmetric in the boundary lengths $L_i$, it suffices to know the piecewise polynomial for a fixed number of odd $L_i$. We use \smash{$N^{[k]}_{g,n}$} to denote the piecewise polynomial where the first $k$ lengths $L_i$ are odd, for $0 \le k \le n$. The count $N_{g,n}$ vanishes unless $\sum_{i=1}^n L_i$ is even; in particular, \smash{$N^{[k]}_{g,n}$} can be non-zero only when $k$ is even. We list the simplest \smash{$N^{[k]}_{g,n}$} in \cref{tab:Ngnk}.

\subsection{Lattice point recursion}
Our second result is a recursive formula for the lattice point count.

\begin{introthm}[Refined lattice point recursion]\label{thm:lattice:rec}
	For $2g-2+n > 1$, the refined lattice point count satisfies the recursion relation
	\begin{equation}\label{eq:lattice:rec}
	\begin{split}
		N_{g,n}(L_1,\ldots,L_n;b)
		&=
		\sum_{m=2}^n \sum_{p > 0}
			p \, \mc{R}(L_1,L_m,p)
			N_{g,n-1}(p,L_2,\ldots,\widehat{L_m},\ldots,L_n;b) \\
		&\qquad
		+
		b \sum_{p > 0}
			p(L_1 - 1) \, \mc{E}(L_1,p)
			N_{g-\frac{1}{2},n}(p,L_2,\ldots,L_n;b) \\
		&\qquad\qquad
		+
		\sum_{p,q > 0}
			p q \, \mc{D}(L_1,p,q)
			\Bigg(
				\frac{1+b}{2} N_{g-1,n+1}(p,q,L_2,\ldots,L_n;b) \\
		&\qquad\qquad\qquad 
				+
				\sum_{\substack{g_1+g_2 = g \\ I_1 \sqcup I_2 = \{2,\ldots,n\}}}
					N_{g_1,1+|I_1|}(p,L_{I_1};b)
					N_{g_2,1+|I_2|}(q,L_{I_2};b)
			\Bigg) ,
	\end{split}
	\end{equation}
	where $\mc{R}$, $\mc{E}$, and $\mc{D}$, corresponding to the geometric operations of reduction of a boundary, excision of a two-holed cross-cap, and degeneration into simpler pieces (cf. \cref{fig:RTR}), are given explicitly by
	\begin{equation}\label{eq:kernels}
	\begin{aligned}
		\mc{R}(L_1,L_m,p) &= \frac{1}{2L_1} \Bigl(
			[L_1 + L_m - p]_+
			-
			[-L_1 + L_m - p]_+
			+
			[L_1 - L_m - p]_+ 
		\Bigr), \\
		\mc{E}(L_1,p) &= \frac{1}{2L_1} [L_1 - p]_+ , \\
		\mc{D}(L_1,p,q) &= \frac{1}{L_1} [L_1 - p - q]_+ .
	\end{aligned}
	\end{equation}
	Here $[x]_+ \coloneqq \max(x,0)$ is the ramp function. Together with the initial conditions
	\begin{equation}\label{eq:base:top:lattice}
	\begin{aligned}
		N_{0,3}(L_1,L_2,L_3;b)
		&=
		\frac{1 + (-1)^{L_1+L_2+L_3}}{2}
		\frac{1}{2}, \\
		N_{\frac{1}{2},2}(L_1,L_2;b)
		&=
		\frac{1 + (-1)^{L_1+L_2}}{2} \,
		b \frac{\max(L_1,L_2) - 1}{4}, \\
		N_{1,1}(L_1;b)
		&=
		\frac{1 + (-1)^{L_1}}{2} \, \frac{
			(1+b)(L_1^2 - 4)
			+
			b^2 (5L_1^2 - 12L_1 + 4)
		}{96} ,
	\end{aligned}
	\end{equation}
	the recursion uniquely determines the refined lattice point count.
\end{introthm}

\begin{figure}[b]
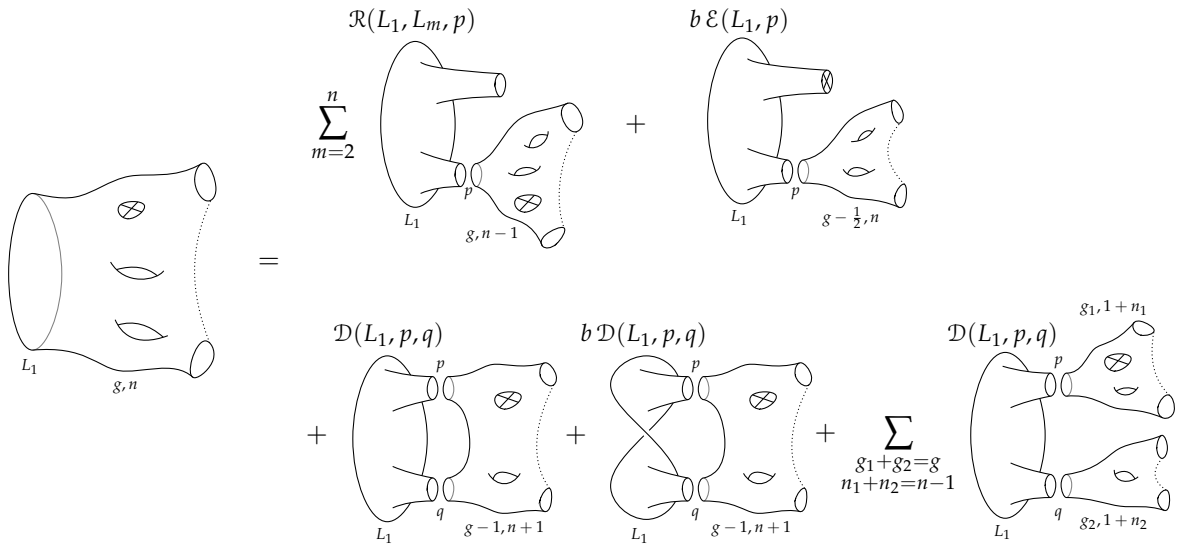


	\caption{A pictorial representation of the recursion formula.}
	\label{fig:RTR}
\end{figure}

Upon setting $b=0$, the recursion reduces to Norbury's recursion for the number of lattice points on the moduli space of curves \cite{Nor10}. Its structure also parallels Mirzakhani's recursion for Weil--Petersson volumes \cite{Mir07a}, with the $\mc{R}$ and $\mc{D}$ functions serving as combinatorial analogues of Mirzakhani's kernel functions as shown in \cite{ABCGLW26}. In the non-orientable hyperbolic setting, analogous recursion formulas were obtained in \cite{GGO} for $b=1$, using Norbury’s extension of the Mirzakhani–McShane identities to the non-orientable case \cite{Nor08}. In our context, the non-orientable contributions are encoded by an $\mc{E}$-term, corresponding to gluing a two-holed cross-cap, together with an additional $\mc{D}$-term accounting for the gluing of a pair of pants in an orientation-reversing fashion. Both contributions are therefore weighted by the refinement parameter~$b$.

We prove the refined lattice point recursion via a Tutte-like argument, analysing how a metric M\"obius graph changes when we remove an edge. From this perspective, the three contributions in the recursion correspond to the probabilities that edge removal decreases the number of faces by one (the $\mc{R}$-term), leaves it unchanged (the $\mc{E}$-term), or increases it by one (the $\mc{D}$-terms). The main additional difficulty in the refined setting is the presence of the measure of non-orientability; its definition is designed to be compatible with this edge-removal decomposition.

\subsection{Volume recursion}
Another general feature of weighted lattice point counts is that their leading term agrees with the corresponding weighted Euclidean \emph{volume}. In our case, when $\sum_{i=1}^n L_i$ is even,
\begin{equation}\label{eq:NV}
	N_{g,n}(L;b) = \frac{2}{2^{2g-2+n}}\, V_{g,n}(L;b) + \cdots ,
\end{equation}
where the dots denote terms of lower degree in $L$, and
\begin{equation}
	V_{g,n}(L;b)
	\coloneqq
	2^{2g-2+n}
	\sum_{\substack{G \in \MG_{g,n} \\ \textup{trivalent}}}
		\frac{1}{|\Aut(G)|}
		\int_{P_G(L)}
			\rho_G(\ell;b) \, d\mu_G(\ell)
\end{equation}
with $d\mu_G$ the Euclidean measure on $P_G(L)$. The factor $2^{-(2g-2+n)}$ in \cref{eq:NV} is purely conventional and is chosen to match the unrefined case. The factor of $2$, on the other hand, reflects the parity constraint: lattice points contribute only half the time, namely only when $\sum_{i=1}^n L_i$ is even.

The refined lattice point recursion from \cref{thm:lattice:rec} then implies, via a Riemann-sum-to-Riemann-integral analysis, the following recursion for the refined volumes.

\begin{introthm}[Refined volume recursion]\label{thm:volume:rec}
	For $2g-2+n > 1$, the refined volumes satisfy the recursion relation
	\begin{equation}\label{eq:volume:rec}
	\begin{split}
		V_{g,n}(L_1,\ldots,L_n;b)
		&=
		\sum_{m=2}^n
			\int_{0}^{+\infty}
				p \, \mc{R}(L_1,L_m,p)
				V_{g,n-1}(p,L_2,\ldots,\widehat{L_m},\ldots,L_n;b) \, dp \\
		&\qquad
		+
		b \int_{0}^{+\infty}
			p L_1 \, \mc{E}(L_1,p)
			V_{g-\frac{1}{2},n}(p,L_2,\ldots,L_n;b) \, dp \\
		&\qquad\qquad
		+
		\int_{0}^{+\infty} \int_{0}^{+\infty}
			p q \, \mc{D}(L_1,p,q)
			\Bigg(
				\frac{1+b}{2} V_{g-1,n+1}(p,q,L_2,\ldots,L_n;b) \\
		&\qquad\qquad\qquad
				+
				\sum_{\substack{g_1+g_2 = g \\ I_1 \sqcup I_2 = \{2,\ldots,n\}}}
					V_{g_1,1+|I_1|}(p,L_{I_1};b)
					V_{g_2,1+|I_2|}(q,L_{I_2};b) 
			\Bigg) dp dq , \\
	\end{split}
	\end{equation}
	where $\mc{R}$, $\mc{E}$ and $\mc{D}$ are as in \cref{eq:kernels}. Together with the initial conditions
	\begin{equation}\label{eq:base:top:volume}
		V_{0,3}(L_1,L_2,L_3;b) = \frac{1}{2},
		\quad
		V_{\frac{1}{2},2}(L_1,L_2;b) = b \frac{\max(L_1,L_2)}{4},
		\quad
		V_{1,1}(L_1;b) = \frac{L_1^2}{96} ( 1 + b + 5b^2 ),
	\end{equation}
	the recursion uniquely determines the refined volumes.
\end{introthm}

The refined volumes at $b=0$ are precisely half the volumes of the moduli space of metric ribbon graphs, introduced by Kontsevich \cite{Kon92} in his proof of Witten's conjecture \cite{Wit91} for $\psi$-class intersection numbers on the moduli space of stable curves (see \cite{DGY} for an intersection-theoretic expression of the lower-order coefficients of $N_{g,n}$ at $b=0$):
\begin{equation}
	V_{g,n}(L;b)\big|_{b=0}
	=
	\frac{1}{2}
	\int_{\overline{\mathcal{M}}_{g,n}} \exp\left( \frac{1}{2} \sum_{i=1}^n \psi_i L_i^2 \right) .
\end{equation}
Moreover, the recursion at $b=0$ is equivalent to the Virasoro constraints for these intersection numbers. In this sense, the refined volumes and the recursion at generic $b$ provide refinements of the generating series of $\psi$-class intersection numbers and the associated Virasoro constraints respectively. We leave the intriguing question of finding an intersection-theoretic interpretation of the refined volumes on an appropriate moduli space for future work.

\subsection{Euler characteristic}
Our final result concerns a \emph{refined Euler characteristic} for the moduli space of metric M\"obius graphs. Goulden, Harer, and Jackson computed, using $\beta$-matrix model techniques, the orbifold Euler characteristic of the moduli space of Klein surfaces \cite{GHJ01}. In fact, their method naturally produces a one-parameter refinement of this Euler characteristic; however, beyond the two extreme cases $b=0$ and $b=1$, this refinement did not come with a direct geometric interpretation.

We provide such an interpretation by weighting the cell decomposition of the moduli space of metric M\"obius graphs by the measure of non-orientability. Concretely, we define
\begin{equation}
	\chi_{g,n}(b)
	\coloneqq
	\sum_{G \in \MG_{g,n}}
		(-1)^{\dim{P_L(G)}}
		\frac{\braket{\rho_G(b)}}{|\Aut(G)|} ,
\end{equation}
where $\braket{\rho_G(b)} \coloneqq \rho_G(1,\ldots,1;b)$ denotes the measure of non-orientability evaluated at the uniform metric on $G$, obtained by assigning unit length to all edges. In this sense, $\braket{\rho_G(b)}$ measures the average non-orientability of the cell associated with~$G$.

By the basic properties of the measure of non-orientability, the specialisations $b=0$ and $b=1$ encode the Euler characteristics of the moduli of Riemann and Klein surfaces, respectively. We relate $\chi_{g,n}(b)$ to the refined lattice point polynomial and obtain an explicit closed formula, thereby recovering the Harer--Zagier formula \cite{HZ86} and the Goulden--Harer--Jackson formula \cite{GHJ01}, and providing a geometric interpretation of the one-parameter refinement.

\begin{introthm}[Refined Euler characteristic]\label{thm:EC}
	The refined Euler characteristic equals the refined lattice point count evaluated at zero boundary lengths: $\chi_{g,n}(b) = N_{g,n}^{[0]}(0;b)$.
	Moreover, its specialisations encode the orbifold Euler characteristics of the moduli of Riemann and non-orientable Klein surfaces:
	\begin{equation}
		\chi(\mc{M}_{g,n}) = 2\,\chi_{g,n}(b)\big|_{b=0},
		\qquad
		\chi(\mc{K}_{g,n})
		=
		2^{n} \Bigl( \chi_{g,n}(b)\big|_{b=1} - \chi_{g,n}(b)\big|_{b=0} \Bigr).
	\end{equation}
	Finally, it is explicitly given by
	\begin{equation}\label{eq:introchi}
		\chi_{g,n}(b)
		=
		(-1)^{n}\,
		\Gamma(2g-2+n)\,
		\frac{
			B_{2,2g}(0 \,|\, \beta^{1/2}, -\beta^{-1/2})
		}{2 \beta^g (2g)!},
	\end{equation}
	where $\beta = \frac{1}{1+b}$ and $B_{2,2g}$ is the $(2g)$-th double Bernoulli polynomial. 
\end{introthm}

The authors of \cite{MP12} use the ribbon graph description to derive a recursive formula for the Poincar\'e polynomial of the moduli space of Riemann surfaces. The techniques developed in the present paper should yield an analogous result for the Poincar\'e polynomial of the moduli space of Klein surfaces; we leave this direction for future work.

\begin{table}[!htbp]
	\centering
	\begin{tabular}{r|ccccc
	}
	\toprule
		& $n=0$
			& $1$
			& $2$
			& $3$
			& $4$
		\\
		\midrule
		$g=0$
			&
			&
			&
			& $\frac{1}{2}$
			& $-\frac{1}{2}$
		\\
		$\frac{1}{2}$
			&
			&
			& $\frac{b}{4}$
			& $-\frac{b}{4}$ & $\frac{b}{2}$
		\\
		$1$
			&
			& $-\frac{1+b-b^2}{24}$
			& $\frac{1+b-b^2}{24}$
			& $-\frac{1+b-b^2}{12}$
			& $\frac{1+b-b^2}{4}$
		\\
		$\frac{3}{2}$
			& $-\frac{b(1+b)}{48}$
			& $\frac{b(1+b)}{48}$
			& $-\frac{b(1+b)}{24}$
			& $\frac{b(1+b)}{8}$
			& $-\frac{b(1+b)}{2}$
		\\
		$2$
			& $-\frac{3+6b-b^2-4b^3-b^4}{1440}$
			& $\frac{3+6b-b^2-4b^3-b^4}{720}$
			& $-\frac{3+6b-b^2-4b^3-b^4}{240}$
			& $\frac{3+6b-b^2-4b^3-b^4}{60}$
			& $-\frac{3+6b-b^2-4b^3-b^4}{12}$
		\\
		$\frac{5}{2}$
			& $\frac{b(1+b)(3+3b+b^2)}{1440}$
			& $-\frac{b(1+b)(3+3b+b^2)}{480}$
			& $\frac{b(1+b)(3+3b+b^2)}{120}$
			& $-\frac{b(1+b)(3+3b+b^2)}{24}$
			& $\frac{b(1+b)(3+3b+b^2)}{4}$
		\\
		\bottomrule
	\end{tabular}
	\caption{The refined Euler characteristic for $g < 3$ and $n \le 4$.}
\end{table}

\subsection{Refined topological recursion, \texorpdfstring{G$\boldsymbol\beta$E}{GβE}, and physics}
To prove some of the results above, we use the \emph{refined topological recursion} formalism recently introduced in \cite{KO23,Osu24a} (see also \cite{CE06} for an earlier attempt). More precisely, we show that the refined topological recursion correlators \smash{$\omega^{\textup{Web}}_{g,n}$} on the Weber spectral curve encode the refined lattice point count via a discrete Laplace transform, under the identification of refinement parameters \smash{$\mf{b} = - \frac{b}{\sqrt{1+b}}$}:
\begin{equation}
	\omega^{\textup{Web}}_{g,n}(z_1,\ldots,z_n; \mf{b})
	=
	(-1)^{n}
	\frac{2}{(1+b)^{g}}
	\sum_{L_1,\ldots,L_n > 0} N_{g,n}(L_1,\ldots,L_n;b) \prod_{i=1}^n L_i \, z_i^{L_i - 1} \, dz_i .
\end{equation}
This refines another result of Norbury \cite{Nor13}. Consequently, we obtain an explicit formula for the refined Euler characteristic using the refined topological recursion free energies on the Weber curve computed through the variational formula in \cite{KO25}. This relationship to refined topological recursion places the refined lattice point counts in the context of the Gaussian $\beta$-ensemble (G$\beta$E). Under $\beta = \frac{1}{1+b}$, the count $N_{g,n}$ coincides with the \emph{pruned genus-$g$ G$\beta$E correlators}, up to an overall combinatorial normalisation (see \cref{app:GbetaE} for the definition of the pruned correlators):
\begin{equation}
	N_{g,n}(L_1,\ldots,L_n;b)
	=
	\frac{1}{2\beta^g}\,
	\Braket{\frac{t_{L_1}}{L_1},\ldots,\frac{t_{L_n}}{L_n}}^{\mathrm{G}\beta\mathrm{E}}_g .
\end{equation}
We also prove that the Laplace transform of the refined volumes matches the refined topological recursion correlators \smash{$\omega^{\textup{Airy}}_{g,n}$} on the Airy spectral curve:
\begin{equation}
	\omega^{\textup{Airy}}_{g,n}(z_1,\ldots,z_n; \mf{b})
	=
	\frac{2}{(1+b)^{g}}
	\int_{0}^{+\infty} \cdots \int_{0}^{+\infty}
		V_{g,n}(L_1,\ldots,L_n;b)
		\prod_{i=1}^n L_i \, e^{-z_i L_i} \, dL_i \, dz_i.
\end{equation}

It is worth pointing out that in \cite{CEM09} the authors study a different combinatorial model for non-orientable ribbon graphs and show that their generating functions satisfy the so-called non-commutative topological recursion. Refined topological recursion and non-commutative topological recursion provide two distinct extensions of the original Eynard--Orantin formalism \cite{EO07} to the $b$-deformed setting; see \cite{CEM10,BE19,BBCC24} for further details on the latter. 

Finally, the refined lattice point counts can be viewed in a broader physics context. Indeed, for certain protected sectors of large-$N$ gauge theories (notably the half-BPS sector of $\mathrm{SU}(N)$ $\mathcal N=4$ super Yang--Mills), one can reorganise the gauge-theory Feynman diagram expansion into a sum over Riemann surfaces with explicit moduli, so that individual diagrams correspond to discrete lattice points on the moduli space of Riemann surfaces \cite{GKKMS}. It is then tempting to speculate that an analogous picture should exist for orthogonal and symplectic gauge groups, where non-orientable worldsheets contribute. In this perspective, integral metric M\"obius graphs should provide the appropriate combinatorial gadget. In particular, since the Gaussian ensembles associated with $\mathrm{SO}(N)$ and $\mathrm{Sp}(N)$ correspond to the special Dyson indices $\beta=1/2$ and $\beta=2$, respectively, one may expect these values to pick out the orthogonal and symplectic cases within our one-parameter family.

A parallel motivation comes from two-dimensional gravity. Volumes of ribbon-graph moduli spaces control the high-energy (Airy) regime of JT gravity and its variants. In particular, time-reversal-invariant theories naturally involve non-orientable geometries together with a crosscap-counting parameter \cite{SW20}. From this perspective, the refined volumes $V_{g,n}$ form a unified framework interpolating between orientable and non-orientable sectors: the specialisation $b=1$ (assigning equal weight to orientable and non-orientable contributions) coincides with the ``Airy volumes'' studied recently in the time-reversal-invariant setting \cite{SSYY24,DEHR25}, while \smash{$b=-\frac{1}{2}$} is naturally expected to correspond to the time-reversal-invariant theory with weight $(-1)^{n_{\mathrm{c}}}$, where $n_{\mathrm{c}}$ denotes the number of crosscaps.

\subsection*{Acknowledgments}
The authors would like to thank M.~Do\l\k{e}ga and D.~Lewa\'nski for valuable discussions, and G.~Borot, P.~Georgieva, E.~A.~Mazenc, M.~Mulase, P.~Norbury and Y.~Schuler for comments on an early draft. We also thank ETH Zürich, the Universitat Politècnica de Catalunya, the University of Melbourne, and Nagoya University for their hospitality.

N.K.C. is supported by the Ramón y Cajal fellowship RYC2023-042878-I, funded by MCIN/AEI/\allowbreak10.13039/\allowbreak501100011033 and by the European Social Fund Plus (FSE+). 
E.G.-F. is supported by the Ramón y Cajal fellowship RYC2023-045188-I, funded by MCIN/AEI/10.13039/501100011033 and by the FSE+. She also acknowledges support from a Tremplin grant (Sorbonne Université), a PEPS grant (CNRS), the ERC-SyG ReNewQuantum, the ANR CarteEtPlus ANR-23-CE48-0018, and the project PID2024-155686NB-I00 of the Spanish Ministry of Science and Innovation. 
A.G. is supported by a Hermann--Weyl Instructorship from the Forschungsinstitut für Mathematik at ETH~Zürich. He also acknowledges support from an ETH~Fellowship (22-2~FEL-003). 
K.O. acknowledges support from JSPS KAKENHI Grant-in-Aid for JSPS Fellows (22KJ0715) and for Early-Career Scientists (23K12968,\allowbreak26K16980), and also in part for Scientific Research B (24K00525). K.O.\ also acknowledges support from the Kobayashi--Maskawa Institute for the Origin of Particles and the Universe at Nagoya University.

\section{M\"obius graphs and the measure of non-orientability}
\label{sec:defs}

In this section we introduce M\"obius graphs, their moduli space, and the measure of non-orientability. The term ``M\"obius graphs'' was coined in \cite{MW03} in the study of Feynman diagram expansions of orthogonal matrix models (see also \cite{BIPZ78,MY05}). M\"obius graphs are the non-orientable analogue of ribbon graphs, which arise, for instance, in the context of the Gaussian unitary ensemble. Although for fixed genus and number of faces these graphs form a discrete set, one can introduce a moduli space by endowing each edge with a length, i.e. a metric. In the orientable case, the resulting moduli space of metric ribbon graphs is isomorphic to the moduli space of curves and has played a crucial role in understanding several of its fundamental properties. Here, we introduce the corresponding non-orientable picture.

Following Chapuy and Do\l\k{e}ga \cite{CD22}, we define a measure of non-orientability on this moduli space. Their notion is motivated by the deformation of Schur symmetric functions into Jack symmetric functions with parameter $1+b$, which appears in connection with the Gaussian $\beta$-ensemble, where \smash{$-\tfrac{b}{\sqrt{1+b}} = \beta^{1/2} - \beta^{-1/2}$}. In our setting, the measure of non-orientability of a metric M\"obius graph is a function that, loosely speaking, records ``how non-orientable'' a point of the moduli space is via the refinement parameter~$b$.

\subsection{M\"obius graphs}
\label{sub:moebius_graphs}
A ribbon graph is a graph $G$ equipped with a cyclic order on the half-edges incident to each vertex. A bicoloured ribbon graph is a ribbon graph together with a $\Z_2$-assignment on its edges. Given a bicoloured ribbon graph $G$, we define its \emph{flip} at a vertex $v$ to be the bicoloured ribbon graph $G'$ obtained by reversing the cyclic order at $v$ and, simultaneously, reversing the $\Z_2$-colouring on all edges adjacent to $v$. An example of a flip move is pictured below, with the $0$-coloured edges shown in black and the $1$-coloured edges shown in orange.
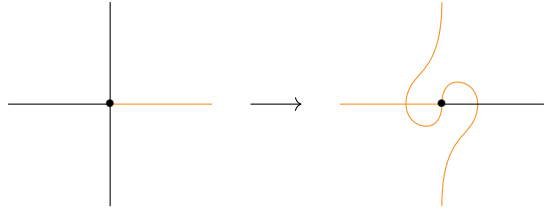
\begin{center}
\begin{tikzpicture}[x=1pt,y=1pt,scale=.6]
	\draw (72, 584) -- (136, 584) -- (136, 648);
	\draw (136, 520) -- (136, 584);
	\draw [BurntOrange] (136, 584) -- (200, 584);
	\node at (136, 584) {\scriptsize$\bullet$};

	\draw [->] (224, 584) -- (256, 584);

	\draw [BurntOrange] (280, 584) -- (344, 584);
	\draw [BurntOrange] (344, 584) .. controls (344, 568) and (332, 568) .. (326, 573.3333) .. controls (320, 578.6667) and (320, 589.3333) .. (326, 597.6667) .. controls (332, 606) and (344, 612) .. (344, 648);
	\draw [BurntOrange] (344, 584) .. controls (344, 600) and (356, 600) .. (362, 594.6667) .. controls (368, 589.3333) and (368, 578.6667) .. (362, 570.3333) .. controls (356, 562) and (344, 556) .. (344, 520);
	\draw (344, 584) -- (408, 584);
	\node at (344, 584) {\scriptsize$\bullet$};
\end{tikzpicture}
\end{center}
Two bicoloured ribbon graphs are called equivalent if they are related by a sequence of vertex flips.

\begin{definition}
	A \emph{M\"obius graph} is an equivalence class $[G]$ of bicoloured ribbon graphs under vertex flips. By abuse of notation, we denote such an equivalence class simply by $G$.
\end{definition}

Given a M\"obius graph $G$, its topological realisation is the homeomorphism class of a (possibly non-orientable) surface with boundary~$\Sigma_G$, obtained by replacing each $0$- or $1$-coloured edge with an untwisted or twisted ribbon, respectively, and gluing these ribbons at the vertices according to the prescribed cyclic orders. Throughout the paper, we freely pass between the description of M\"obius graphs as bicoloured ribbon graphs and as their topological realisations, depending on convenience. In terms of the topological realisation, a flip at a vertex can be visualised as
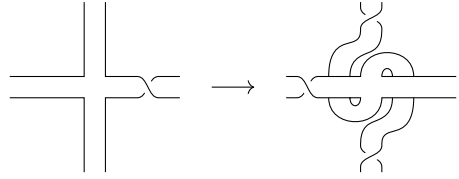
\begin{center}
	\begin{tikzpicture}[x=1pt,y=1pt,scale=.5]
		\draw (200, 768)  -- (184, 768)  .. controls (176, 768) and (176, 752) .. (168, 752)  -- (144, 752)  -- (144, 696);
		\draw (128, 824)  -- (128, 768)  -- (72, 768);
		\draw (72, 752)  -- (128, 752)  -- (128, 696);
		\draw[white, line width=1.5mm] (168, 768)  .. controls (176, 768) and (176, 752) .. (184, 752);
		\draw (144, 824)  -- (144, 768)  -- (168, 768)  .. controls (176, 768) and (176, 752) .. (184, 752)  -- (200, 752);
		\draw[->] (224, 760)  -- (256, 760);
		\draw (352, 696)  -- (352, 700)  .. controls (352, 708) and (336, 708) .. (336, 716)  .. controls (336, 720) and (336, 728) .. (342, 732)  .. controls (348, 736) and (360, 736) .. (360, 752);
		\draw (336, 824)  -- (336, 820)  .. controls (336, 812) and (352, 812) .. (352, 804)  .. controls (352, 800) and (352, 792) .. (346, 788)  .. controls (340, 784) and (328, 784) .. (328, 768);
		\draw (280, 752)  -- (288, 752)  .. controls (296, 752) and (296, 768) .. (304, 768)  -- (336, 768)  .. controls (336, 792) and (376, 792) .. (376, 768);
		\draw[white, line width=1.5mm] (288, 768)  .. controls (296, 768) and (296, 752) .. (304, 752);
		\draw[white, line width=1.5mm] (336, 804)  .. controls (336, 812) and (352, 812) .. (352, 820);
		\draw[white, line width=1.5mm] (336, 700)  .. controls (336, 708) and (352, 708) .. (352, 716);
		\draw (280, 768)  -- (288, 768)  .. controls (296, 768) and (296, 752) .. (304, 752)  -- (336, 752)  .. controls (336, 744) and (328, 744) .. (328, 752);
		\draw (312, 768)  .. controls (312, 792) and (324, 792) .. (330, 794)  .. controls (336, 796) and (336, 800) .. (336, 804)  .. controls (336, 812) and (352, 812) .. (352, 820)  -- (352, 824);
		\draw (360, 768)  .. controls (360, 776) and (352, 776) .. (352, 768)  -- (408, 768);
		\draw (312, 752)  .. controls (312, 728) and (352, 728) .. (352, 752)  -- (408, 752);
		\draw (336, 696)  -- (336, 700)  .. controls (336, 708) and (352, 708) .. (352, 716)  .. controls (352, 720) and (352, 724) .. (358, 726)  .. controls (364, 728) and (376, 728) .. (376, 752);
	\end{tikzpicture}
\end{center}
which provides a geometric motivation for the definition above.

In the topological realisation of a M\"obius graph $G$, each boundary component (or face) of $\Sigma_G$ is a circle which, in general, does not carry a consistent orientation. From now on, we assume that the $n$ boundary components are labelled by $1,\ldots,n$. We define the \emph{type} of $G$ to be the pair $(g,n)$, where $g$ is the genus and $n$ is the number of boundary components of $\Sigma_G$. The genus is defined by the relation $\chi(\Sigma_G) = 2-2g-n$, where $\chi(\Sigma_G)$ denotes the Euler characteristic of $\Sigma_G$. If $\Sigma_G$ is orientable, then $g$ agrees with the usual genus. If $\Sigma_G$ is non-orientable, then $2g$ is the maximal number of cross-caps of $\Sigma_G$. In particular, \smash{$g\in \tfrac12\Z_{\ge 0}$} is a non-negative half-integer, while $n\in \Z_{>0}$.

The Euler relation can be written as
\begin{equation}
	2g - 2 + n = |E(G)| - |V(G)|,
\end{equation}
where $E(G)$ and $V(G)$ denote the sets of edges and vertices of $G$, respectively. We impose the stability condition $2g-2+n>0$, so that $|E(G)|>|V(G)|$. If all vertices have valency at least~$3$, then there are only finitely many M\"obius graphs of a fixed type $(g,n)$. In particular, if all vertices are trivalent, then $|E(G)| = 6g-6+3n$. We denote by $\MG_{g,n}$ the set of (isomorphism classes of) M\"obius graphs of type $(g,n)$. Unless stated otherwise, all M\"obius graphs are assumed to be connected, face-labelled, and with all vertices of valency at least~$3$.

For a given M\"obius graph $G$, we define $\Aut(G)$ to be the group of automorphisms of the underlying graph that preserve the cyclic orderings, the $\Z_2$-colouring, and the face labelling, up to vertex flips. Note that, for an orientable ribbon graph $G$, the automorphism group viewed as a M\"obius graph is twice as large as the automorphism group viewed as an oriented ribbon graph. From a topological perspective, this reflects the distinction between oriented and merely orientable surfaces, and accounts for an additional factor of~$2$ coming from orientation-reversing morphisms.

\begin{example}\label{ex:chi1}
	Below is a list of all M\"obius graphs, together with their topological realisations, for $2g-2+n=1$, as well as the orders of their automorphism groups. The cyclic ordering at each vertex is given by the orientation of the plane. The $0$-coloured edges are shown in black, the $1$-coloured edges in orange. The face labelling is indicated by a number in $\set{1,\dots,n}$ placed inside each face.\\[1ex]
	$\bullet\;(g,n) = (0,3)$: a pair of pants.
	
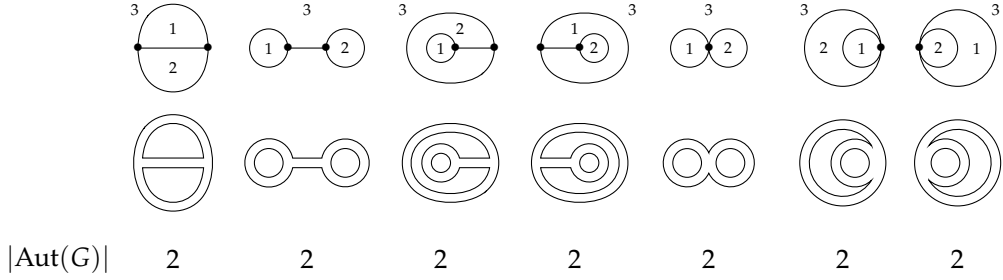
\begin{center}
		\begin{tikzpicture}[x=1pt,y=1pt,scale=.45]
			\draw (42.6667, 794.6667) .. controls (53.3333, 808) and (74.6667, 808) .. (85.3333, 794.6667) .. controls (96, 781.3333) and (96, 754.6667) .. (85.3333, 741.3333) .. controls (74.6667, 728) and (53.3333, 728) .. (42.6667, 741.3333) .. controls (32, 754.6667) and (32, 781.3333) .. cycle;
			\draw (34.6667, 767.985) -- (93.3333, 768.001);
			\draw (144, 768) circle[radius=16];
			\draw (208, 768) circle[radius=16];
			\draw (160, 768) -- (192, 768);
			\draw (322.6667, 746.6667) .. controls (336, 757.3333) and (336, 778.6667) .. (322.6667, 789.3333) .. controls (309.3333, 800) and (282.6667, 800) .. (269.3333, 789.3333) .. controls (256, 778.6667) and (256, 757.3333) .. (269.3333, 746.6667) .. controls (282.6667, 736) and (309.3333, 736) .. cycle;
			\draw (288, 768) circle[radius=12];
			\draw (300, 768) -- (332.667, 768);
			\draw  [shift={(381.333, 789.333)}, scale=-1](0, 0) .. controls (13.3333, 10.6667) and (13.3333, 32) .. (0, 42.6667) .. controls (-13.3333, 53.3333) and (-40, 53.3333) .. (-53.3333, 42.6667) .. controls (-66.6667, 32) and (-66.6667, 10.6667) .. (-53.3333, 0) .. controls (-40, -10.6667) and (-13.3333, -10.6667) .. cycle;
			\draw (416, 768) circle[radius=-12];
			\draw  [shift={(404, 768)}, scale=-1](0, 0) -- (32.667, 0);
			\draw (496, 768) circle[radius=16];
			\draw (528, 768) circle[radius=16];
			\draw (624, 768) circle[radius=32];
			\draw (640, 768) circle[radius=16];
			\draw (720, 768) circle[radius=32];
			\draw (704, 768) circle[radius=16];
			\node at (64, 784) {\tiny$1$};
			\node at (64, 752) {\tiny$2$};
			\node at (32, 800) {\tiny$3$};
			\node at (144, 768) {\tiny$1$};
			\node at (208, 768) {\tiny$2$};
			\node at (176, 800) {\tiny$3$};
			\node at (288, 768) {\tiny$1$};
			\node at (304, 784) {\tiny$2$};
			\node at (256, 800) {\tiny$3$};
			\node at (448, 800) {\tiny$3$};
			\node at (416, 768) {\tiny$2$};
			\node at (400, 784) {\tiny$1$};
			\node at (496, 768) {\tiny$1$};
			\node at (528, 768) {\tiny$2$};
			\node at (512, 800) {\tiny$3$};
			\node at (640, 768) {\tiny$1$};
			\node at (608, 768) {\tiny$2$};
			\node at (592, 800) {\tiny$3$};
			\node at (736, 768) {\tiny$1$};
			\node at (704, 768) {\tiny$2$};
			\node at (752, 800) {\tiny$3$};
			\node at (34.6667, 767.985) {\scriptsize$\bullet$};
			\node at (93.3333, 768.001) {\scriptsize$\bullet$};
			\node at (160, 768) {\scriptsize$\bullet$};
			\node at (192, 768) {\scriptsize$\bullet$};
			\node at (300, 768) {\scriptsize$\bullet$};
			\node at (332.667, 768) {\scriptsize$\bullet$};
			\node at (371.333, 768) {\scriptsize$\bullet$};
			\node at (404, 768) {\scriptsize$\bullet$};
			\node at (512, 768) {\scriptsize$\bullet$};
			\node at (656, 768) {\scriptsize$\bullet$};
			\node at (688, 768) {\scriptsize$\bullet$};
			\draw(40, 701.3333) .. controls (52, 716) and (76, 716) .. (88, 701.3333) .. controls (100, 686.6667) and (100, 657.3333) .. (88, 642.6667) .. controls (76, 628) and (52, 628) .. (40, 642.6667) .. controls (28, 657.3333) and (28, 686.6667) .. cycle;
			\draw(38.493, 667.932) -- (89.5073, 667.937) .. controls (88.9095, 660.353) and (86.6293, 653.095) .. (82.6667, 648) .. controls (73.3333, 636) and (54.6667, 636) .. (45.3333, 648) .. controls (41.3715, 653.094) and (39.0914, 660.35) .. (38.493, 667.932) -- cycle;
			\draw (45.3333, 696) .. controls (54.6667, 708) and (73.3333, 708) .. (82.6667, 696) .. controls (86.631, 690.903) and (88.9115, 683.641) .. (89.5081, 676.053) -- (38.4916, 676.05) .. controls (39.0878, 683.639) and (41.3684, 690.902) .. (45.3333, 696) -- cycle;
			\draw(144, 672) circle[radius=12];
			\draw(208, 672) circle[radius=12];
			\draw(188.4041, 676) arc[start angle=-168.4631, end angle=168.4631, x radius=20, y radius=-20] -- (164, 668) arc[start angle=11.4212, end angle=348.5788, x radius=20.2, y radius=-20.2] -- cycle;
			\draw(288, 672) circle[radius=8];
			\draw(325.3333, 648) .. controls (340, 660) and (340, 684) .. (325.3333, 696) .. controls (310.6667, 708) and (281.3333, 708) .. (266.6667, 696) .. controls (252, 684) and (252, 660) .. (266.6667, 648) .. controls (281.3333, 636) and (310.6667, 636) .. cycle;
			\draw(328.671, 676) .. controls (327.75, 681.579) and (324.859, 686.888) .. (320, 690.667) .. controls (308, 700) and (284, 700) .. (272, 690.667) .. controls (260, 681.333) and (260, 662.667) .. (272, 653.333) .. controls (284, 644) and (308, 644) .. (320, 653.333) .. controls (324.854, 657.109) and (327.745, 662.412) .. (328.671, 668) -- (303.492, 668) arc[start angle=14.4775, end angle=345.5225, x radius=16, y radius=-16] -- cycle;
			\draw(412, 672) circle[radius=-8];
			\draw [shift={(374.667, 696)}, scale=-1](0, 0) .. controls (14.6667, 12) and (14.6667, 36) .. (0, 48) .. controls (-14.6667, 60) and (-44, 60) .. (-58.6667, 48) .. controls (-73.3333, 36) and (-73.3333, 12) .. (-58.6667, 0) .. controls (-44, -12) and (-14.6667, -12) .. cycle;
			\draw(396.508, 676) arc[start angle=14.4775, end angle=345.5225, x radius=-16, y radius=16] -- (371.329, 668) .. controls (372.25, 662.421) and (375.141, 657.112) .. (380, 653.333) .. controls (392, 644) and (416, 644) .. (428, 653.333) .. controls (440, 662.667) and (440, 681.333) .. (428, 690.667) .. controls (416, 700) and (392, 700) .. (380, 690.667) .. controls (375.146, 686.891) and (372.255, 681.588) .. (371.329, 676) -- cycle;
			\draw(493.8443, 671.9384) circle[radius=12];
			\draw(530.0135, 671.9764) circle[radius=12];
			\draw(511.9197, 680.4978) arc[start angle=25.3402, end angle=334.7803, radius=20] arc[start angle=-154.6586, end angle=154.7791, radius=20] -- cycle;
			\draw(634.147, 671.917) circle[radius=12];
			\draw(624, 672) circle[radius=36];
			\draw(648.112, 686.234) arc[start angle=45.7131, end angle=313.3529, radius=20] arc[start angle=31.4898, end angle=329.4454, x radius=28, y radius=-28] -- cycle;
			\draw(709.85, 672.08) circle[radius=-12];
			\draw(720, 672) circle[radius=-36];
			\draw(695.888, 657.766) arc[start angle=45.7131, end angle=313.3529, radius=-20] arc[start angle=31.4898, end angle=329.4454, x radius=-28, y radius=28] -- cycle;
			\node at (64, 592) {$2$};
			\node at (176, 592) {$2$};
			\node at (288, 592) {$2$};
			\node at (400, 592) {$2$};
			\node at (512, 592) {$2$};
			\node at (624, 592) {$2$};
			\node at (720, 592) {$2$};
			\node at (-32, 592) {$|\Aut(G)|$};
		\end{tikzpicture}
	\end{center}

	$\bullet\;(g,n) = (\frac{1}{2},2)$: a two-holed cross-cap.
	
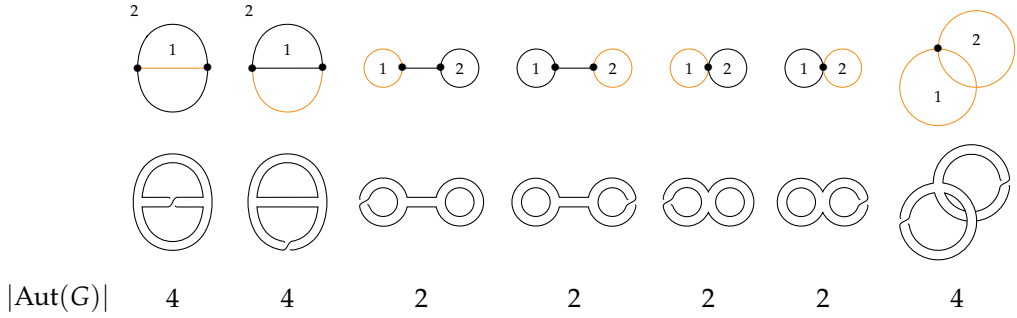
\begin{center}
		\begin{tikzpicture}[x=1pt,y=1pt,scale=.45]
			\draw (42.6667, 794.6667) .. controls (53.3333, 808) and (74.6667, 808) .. (85.3333, 794.6667) .. controls (96, 781.3333) and (96, 754.6667) .. (85.3333, 741.3333) .. controls (74.6667, 728) and (53.3333, 728) .. (42.6667, 741.3333) .. controls (32, 754.6667) and (32, 781.3333) .. cycle;
			\draw [BurntOrange] (34.6667, 767.985) -- (93.3333, 768.001);
			\node at (64, 784) {\tiny$1$};
			\node at (32, 816) {\tiny$2$};
			\node at (34.6667, 767.985) {\scriptsize$\bullet$};
			\node at (93.3333, 768.001) {\scriptsize$\bullet$};
			\draw (130.6667, 767.985) -- (189.3333, 768.001);
			\draw (138.667, 794.667) .. controls (149.333, 808) and (170.667, 808) .. (181.333, 794.667) .. controls (186.667, 788) and (189.333, 778) .. (189.333, 768);
			\draw [BurntOrange] (189.333, 768) .. controls (189.333, 758) and (186.667, 748) .. (181.333, 741.333) .. controls (170.667, 728) and (149.333, 728) .. (138.667, 741.333) .. controls (133.333, 748) and (130.667, 758) .. (130.667, 768);
			\draw (130.667, 768) .. controls (130.667, 778) and (133.333, 788) .. (138.667, 794.667);
			\node at (160, 784) {\tiny$1$};
			\node at (128, 816) {\tiny$2$};
			\node at (130.6667, 767.985) {\scriptsize$\bullet$};
			\node at (189.3333, 768.001) {\scriptsize$\bullet$};
			\draw [BurntOrange] (240, 768) circle[radius=16];
			\draw (304, 768) circle[radius=16];
			\draw (256, 768) -- (288, 768);
			\node at (240, 768) {\tiny$1$};
			\node at (304, 768) {\tiny$2$};
			\node at (256, 768) {\scriptsize$\bullet$};
			\node at (288, 768) {\scriptsize$\bullet$};
			\draw (368, 768) circle[radius=16];
			\draw [BurntOrange] (432, 768) circle[radius=16];
			\draw (384, 768) -- (416, 768);
			\node at (368, 768) {\tiny$1$};
			\node at (432, 768) {\tiny$2$};
			\node at (384, 768) {\scriptsize$\bullet$};
			\node at (416, 768) {\scriptsize$\bullet$};
			\draw [BurntOrange] (704, 752) circle[radius=32];
			\draw [BurntOrange] (736, 784) circle[radius=32];
			\node at (704, 784) {\scriptsize$\bullet$};
			\node at (704, 744) {\tiny$1$};
			\node at (736, 792) {\tiny$2$};
			\draw [BurntOrange] (496, 768) circle[radius=16];
			\draw (528, 768) circle[radius=16];
			\node at (496, 768) {\tiny$1$};
			\node at (528, 768) {\tiny$2$};
			\node at (512, 768) {\scriptsize$\bullet$};
			\draw (592, 768) circle[radius=16];
			\draw [BurntOrange] (624, 768) circle[radius=16];
			\node at (592, 768) {\tiny$1$};
			\node at (624, 768) {\tiny$2$};
			\node at (608, 768) {\scriptsize$\bullet$};
			\draw (67.994, 651.924) .. controls (64, 652) and (64, 660) .. (60, 660);
			\draw [white, line width=1mm](60, 652) .. controls (64, 652) and (64, 660) .. (67.9948, 660.04);
			\draw (40, 685.3333) .. controls (52, 700) and (76, 700) .. (88, 685.3333) .. controls (100, 670.6667) and (100, 641.3333) .. (88, 626.6667) .. controls (76, 612) and (52, 612) .. (40, 626.6667) .. controls (28, 641.3333) and (28, 670.6667) .. cycle;
			\draw (67.994, 651.924) -- (89.5073, 651.937) .. controls (88.9095, 644.353) and (86.6293, 637.095) .. (82.6667, 632) .. controls (73.3333, 620) and (54.6667, 620) .. (45.3333, 632) .. controls (41.3715, 637.094) and (39.0914, 644.35) .. (38.4889, 651.985) -- (60, 652) .. controls (64, 652) and (64, 660) .. (67.9948, 660.04) -- (89.5081, 660.053) .. controls (88.9115, 667.641) and (86.631, 674.903) .. (82.6667, 680) .. controls (73.3333, 692) and (54.6667, 692) .. (45.3333, 680) .. controls (41.3684, 674.902) and (39.0878, 667.639) .. (38.4867, 659.987) -- (60, 660);
			\draw (155.878, 623.35) .. controls (160.005, 622.588) and (160.02, 615.529) .. (164.019, 615.912) .. controls (171.62, 616.843) and (178.895, 620.428) .. (184, 626.667) .. controls (190, 634) and (193, 645) .. (193, 656) .. controls (193, 667) and (190, 678) .. (184, 685.333) .. controls (172, 700) and (148, 700) .. (136, 685.333) .. controls (130, 678) and (127, 667) .. (127, 656);
			\draw [white, line width=1mm](164.1363, 623.352) .. controls (160.0103, 622.799) and (160.0021, 615.565) .. (155.9997, 615.909);
			\draw (185.508, 660.053) -- (134.487, 659.987) .. controls (135.088, 667.639) and (137.368, 674.902) .. (141.333, 680) .. controls (150.667, 692) and (169.333, 692) .. (178.667, 680) .. controls (182.631, 674.903) and (184.911, 667.641) .. (185.508, 660.053) -- cycle;
			\draw (155.878, 623.35) .. controls (150.331, 624.299) and (145.08, 627.183) .. (141.333, 632) .. controls (137.371, 637.094) and (135.091, 644.35) .. (134.489, 651.985) -- (185.507, 651.937) .. controls (184.909, 644.353) and (182.629, 637.095) .. (178.667, 632) .. controls (174.923, 627.187) and (169.678, 624.304) .. (164.136, 623.352) .. controls (160.01, 622.799) and (160.002, 615.565) .. (156, 615.909) .. controls (148.393, 616.837) and (141.109, 620.422) .. (136, 626.667) .. controls (130, 634) and (127, 645) .. (127, 656);
			\draw (220.397, 659.985) .. controls (219.69, 656.007) and (226.963, 655.974) .. (228.688, 651.996);
			\draw [white, line width=1mm](220.399, 652.006) .. controls (219.671, 655.949) and (227.027, 656.028) .. (228.682, 659.989);
			\draw (304, 656) circle[radius=12];
			\draw (228.6877, 651.9959) arc[start angle=-160.5081, end angle=160.5851, radius=12] .. controls (226.998, 655.987) and (219.689, 656.009) .. (220.399, 652.006) arc[start angle=-168.5961, end angle=-11.4212, radius=20.2] -- (284.404, 652) arc[start angle=-168.4631, end angle=168.4631, radius=20] -- (260, 660) arc[start angle=11.4212, end angle=168.6222, radius=20.2];
			\draw [shift={(451.274, 652.015)}, scale=-1](0, 0) .. controls (-0.707, -3.978) and (6.566, -4.011) .. (8.291, -7.989);
			\draw [shift={(451.272, 659.994)}, scale=-1, white, line width=1mm](0, 0) .. controls (-0.728, 3.943) and (6.628, 4.022) .. (8.283, 7.983);
			\draw (367.671, 656) circle[radius=-12];
			\draw [shift={(442.983, 660.004)}, scale=-1](0, 0) arc[start angle=-160.5081, end angle=160.5851, radius=12] .. controls (-1.6897, 3.9911) and (-8.9987, 4.0131) .. (-8.2887, 0.0101) arc[start angle=-168.5961, end angle=-11.4212, radius=20.2] -- (55.7163, 0.0041) arc[start angle=-168.4631, end angle=168.4631, radius=20] -- (31.3123, 8.0041) arc[start angle=11.4212, end angle=168.6222, radius=20.2];
			\draw (700.001, 671.7491) arc[start angle=97.179, end angle=172.8752, radius=32] .. controls (671.868, 639.988) and (679.344, 640.003) .. (680, 636) arc[start angle=-170.5381, end angle=76.9443, radius=24.3204] arc[start angle=-159.7388, end angle=-102.5498, radius=23.9733];
			\draw [white, line width=1mm](680.177, 644.006) .. controls (679.733, 640.02) and (671.474, 639.993) .. (672, 636);
			\draw (736.2651, 640.2855) arc[start angle=-82.3406, end angle=-7.1248, radius=32] .. controls (764.217, 672.015) and (756.286, 672.012) .. (755.673, 675.946) arc[start angle=9.4627, end angle=180.0167, radius=23.9837] arc[start angle=-82.8177, end angle=172.875, x radius=32.249, y radius=-32.249] .. controls (671.474, 639.993) and (679.733, 640.02) .. (680.177, 644.006) arc[start angle=-169.952, end angle=-97.0288, x radius=24.2217, y radius=-24.2217] arc[start angle=-165.1903, end angle=-96.5824, radius=31.9979];
			\draw [white, line width=1mm](763.753, 675.9694) .. controls (764.189, 671.9934) and (756.222, 671.9854) .. (755.673, 668.0544);
			\draw (700.001, 671.749) arc[start angle=179.5506, end angle=352.8752, x radius=32, y radius=-32] .. controls (764.189, 671.993) and (756.222, 671.985) .. (755.673, 668.054) arc[start angle=9.4639, end angle=82.3437, x radius=24.0111, y radius=-24.0111];
			\draw (474.266, 660.026) .. controls (473.559, 656.048) and (480.832, 656.015) .. (482.507, 652.004);
			\draw [white, line width=1mm](474.231, 652.0236) .. controls (473.503, 655.9666) and (480.859, 656.0456) .. (482.514, 660.0066);
			\draw (530.013, 655.9764) circle[radius=12];
			\draw (482.5077, 652.0043) arc[start angle=-160.8643, end angle=160.769, radius=12.1771] .. controls (480.859, 656.046) and (473.503, 655.967) .. (474.231, 652.024) arc[start angle=-168.7142, end angle=-25.2197, radius=20] arc[start angle=-154.6586, end angle=154.7791, radius=20] arc[start angle=25.3402, end angle=168.2058, radius=20];
			\draw [shift={(645.256, 651.958)}, rotate=-179.8973](0, 0) .. controls (-0.707, -3.978) and (6.566, -4.011) .. (8.241, -8.022);
			\draw [shift={(645.278, 659.962)}, rotate=-179.8973, white, line width=1mm](0, 0) .. controls (-0.728, 3.943) and (6.628, 4.022) .. (8.283, 7.983);
			\draw (589.5051, 655.9053) circle[radius=12];
			\draw [shift={(637.001, 659.965)}, rotate=-179.8973](0, 0) arc[start angle=-160.8643, end angle=160.769, radius=12.1771] .. controls (-1.6487, 4.0417) and (-9.0047, 3.9627) .. (-8.2767, 0.0197) arc[start angle=-168.7142, end angle=-25.2197, radius=20] arc[start angle=-154.6586, end angle=154.7791, radius=20] arc[start angle=25.3402, end angle=168.2058, radius=20];
			\node at (-32, 576) {$|\Aut(G)|$};
			\node at (64, 576) {$4$};
			\node at (160, 576) {$4$};
			\node at (272, 576) {$2$};
			\node at (400, 576) {$2$};
			\node at (720, 576) {$4$};
			\node at (512, 576) {$2$};
			\node at (608, 576) {$2$};
		\end{tikzpicture}
	\end{center}

	$\bullet\;(g,n) = (1,1)$: a one-holed torus or Klein bottle. We omit the labelling as there is only one face. The first two graphs are one-holed tori; the last four are one-holed Klein bottles.
	
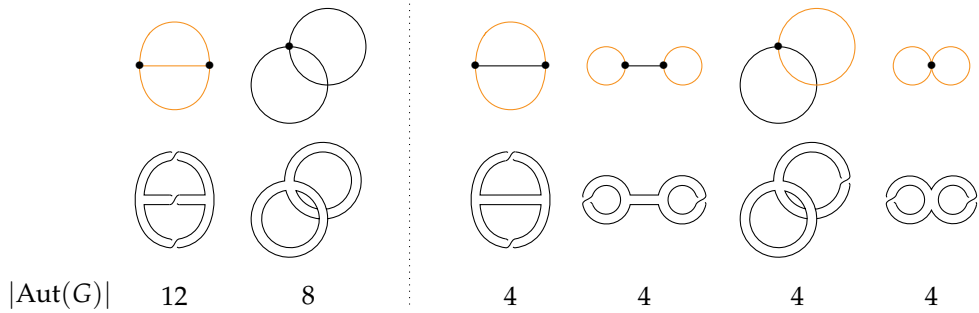
\begin{center}
		\begin{tikzpicture}[x=1pt,y=1pt,scale=.45]
			\draw (160, 752) circle[radius=32];
			\draw (192, 784) circle[radius=32];
			\node at (160, 784) {\scriptsize$\bullet$};
			\draw [BurntOrange] (42.6667, 794.6667) .. controls (53.3333, 808) and (74.6667, 808) .. (85.3333, 794.6667) .. controls (96, 781.3333) and (96, 754.6667) .. (85.3333, 741.3333) .. controls (74.6667, 728) and (53.3333, 728) .. (42.6667, 741.3333) .. controls (32, 754.6667) and (32, 781.3333) .. cycle;
			\draw [BurntOrange] (34.6667, 767.985) -- (93.3333, 768.001);
			\node at (34.6667, 767.985) {\scriptsize$\bullet$};
			\node at (93.3333, 768.001) {\scriptsize$\bullet$};

			\begin{scope}[xshift=2cm]
			\draw [BurntOrange] (624, 768) circle[radius=16];
			\draw [BurntOrange] (656, 768) circle[radius=16];
			\node at (640, 768) {\scriptsize$\bullet$};
			\draw (512, 752) circle[radius=32];
			\draw [BurntOrange] (544, 784) circle[radius=32];
			\node at (512, 784) {\scriptsize$\bullet$};
			\draw [BurntOrange] (368, 768) circle[radius=16];
			\draw [BurntOrange] (432, 768) circle[radius=16];
			\draw (384, 768) -- (416, 768);
			\node at (384, 768) {\scriptsize$\bullet$};
			\node at (416, 768) {\scriptsize$\bullet$};
			\draw [BurntOrange] (266.6667, 794.6667) .. controls (277.3333, 808) and (298.6667, 808) .. (309.3333, 794.6667) .. controls (320, 781.3333) and (320, 754.6667) .. (309.3333, 741.3333) .. controls (298.6667, 728) and (277.3333, 728) .. (266.6667, 741.3333) .. controls (256, 754.6667) and (256, 781.3333) .. cycle;
			\draw (258.6667, 767.985) -- (317.3333, 768.001);
			\node at (258.6667, 767.985) {\scriptsize$\bullet$};
			\node at (317.3333, 768.001) {\scriptsize$\bullet$};
			\end{scope}

			\begin{scope}[xshift=2cm]
			  \draw (602.266, 660.026) .. controls (601.559, 656.048) and (608.832, 656.015) .. (610.507, 652.004);
			\draw [white, line width=1mm] (602.231, 652.0236) .. controls (601.503, 655.9666) and (608.859, 656.0456) .. (610.514, 660.0066);
			\draw [shift={(677.641, 652.036)}, rotate=-179.7756] (0, 0) .. controls (-0.707, -3.978) and (6.566, -4.011) .. (8.241, -8.022);
			\draw [shift={(677.645, 660.038)}, rotate=-179.7756, white, line width=1mm] (0, 0) .. controls (-0.728, 3.943) and (6.628, 4.022) .. (8.283, 7.983);
			\draw (602.2662, 660.0259) arc[start angle=-168.2058, end angle=-25.3402, x radius=20, y radius=-20] { [rotate=-179.7756] arc[start angle=25.2197, end angle=168.7142, x radius=20, y radius=-20] } .. controls (678.3882, 656.0976) and (671.0326, 655.9898) .. (669.3931, 652.0223) { [rotate=-179.7756] arc[start angle=-160.769, end angle=160.8643, x radius=12.1771, y radius=-12.1771] };
			\draw (610.5077, 652.0043) arc[start angle=-160.8643, end angle=160.769, radius=12.1771] .. controls (608.859, 656.046) and (601.503, 655.967) .. (602.231, 652.024) arc[start angle=-168.7142, end angle=-25.2197, radius=20] arc[start angle=-154.4591, end angle=-11.5699, radius=20.0338];
			\draw (571.753, 668.031) .. controls (572.217, 672.015) and (564.286, 672.012) .. (563.673, 675.946);
			\draw [white, line width=1mm] (571.753, 675.9694) .. controls (572.189, 671.9934) and (564.222, 671.9854) .. (563.673, 668.0544);
			\draw (535.9991, 640.2472) arc[start angle=97.1907, end angle=165.1903, x radius=32.0054, y radius=-32.0054] arc[start angle=97.0291, end angle=436.9685, radius=24.0004] arc[start angle=-159.1175, end angle=-103.1938, radius=24.1283];
			\draw (543.9988, 640.2629) arc[start angle=-82.8463, end angle=-7.1248, radius=31.9841];
			\draw (563.6733, 675.9461) arc[start angle=9.4627, end angle=180.6005, radius=24.0002] arc[start angle=-82.8177, end angle=262.821, x radius=31.9999, y radius=-31.9999] arc[start angle=179.5506, end angle=352.8752, x radius=32, y radius=-32] .. controls (572.189, 671.993) and (564.222, 671.985) .. (563.673, 668.054) arc[start angle=9.4639, end angle=82.9822, x radius=23.9959, y radius=-23.9959];
			\draw (348.397, 659.985) .. controls (347.69, 656.007) and (354.963, 655.974) .. (356.688, 651.996);
			\draw [white, line width=1mm] (348.399, 652.006) .. controls (347.671, 655.949) and (355.027, 656.028) .. (356.682, 659.989);
			\draw [shift={(451.274, 652.015)}, scale=-1] (0, 0) .. controls (-0.707, -3.978) and (6.566, -4.011) .. (8.291, -7.989);
			\draw [shift={(451.272, 659.994)}, scale=-1, white, line width=1mm] (0, 0) .. controls (-0.728, 3.943) and (6.628, 4.022) .. (8.283, 7.983);
			\draw (451.274, 652.015) arc[start angle=-168.6222, end angle=-11.4212, x radius=-20.2, y radius=20.2] -- (388, 652) arc[start angle=11.4212, end angle=168.5961, x radius=20.2, y radius=-20.2] .. controls (347.689, 656.009) and (354.998, 655.987) .. (356.682, 659.989) arc[start angle=-160.5851, end angle=160.5081, x radius=12, y radius=-12];
			\draw (442.9833, 660.0041) arc[start angle=-160.5081, end angle=160.5851, radius=-12] .. controls (444.673, 656.013) and (451.982, 655.991) .. (451.272, 659.994) arc[start angle=-168.5961, end angle=-11.4212, radius=-20.2] -- (388, 660) arc[start angle=11.4212, end angle=168.6222, radius=20.2];
			\draw (283.9997, 696.091) .. controls (288.0197, 696.434) and (288.0094, 689.346) .. (292.1363, 688.648) .. controls (297.6783, 687.696) and (302.9232, 684.813) .. (306.6667, 680) .. controls (310.631, 674.903) and (312.9115, 667.641) .. (313.5081, 660.053) -- (262.4867, 659.987) .. controls (263.0878, 667.639) and (265.3684, 674.902) .. (269.3333, 680) .. controls (273.08, 684.817) and (278.3308, 687.701) .. (283.878, 688.65);
			\draw (283.878, 623.35) .. controls (288.0049, 622.588) and (288.0202, 615.529) .. (292.0195, 615.912) .. controls (299.6195, 616.843) and (306.8954, 620.428) .. (312, 626.667) .. controls (324, 641.333) and (324, 670.667) .. (312, 685.333) .. controls (306.8954, 691.572) and (299.6195, 695.157) .. (292.0195, 696.088);
			\draw [white, line width=1mm] (292.1363, 623.352) .. controls (288.0103, 622.799) and (288.0021, 615.565) .. (283.9997, 615.909);
			\draw [white, line width=1mm] (283.878, 688.65) .. controls (287.9989, 689.396) and (288.0007, 696.523) .. (292.0195, 696.088);
			\draw (283.9997, 696.091) .. controls (276.3928, 695.163) and (269.1089, 691.578) .. (264, 685.333) .. controls (252, 670.667) and (252, 641.333) .. (264, 626.667) .. controls (269.1089, 620.422) and (276.3928, 616.837) .. (283.9997, 615.909) .. controls (288.0021, 615.565) and (288.0103, 622.799) .. (292.1363, 623.352) .. controls (297.6783, 624.304) and (302.9232, 627.187) .. (306.6667, 632) .. controls (310.6293, 637.095) and (312.9095, 644.353) .. (313.5073, 651.937) -- (262.4889, 651.985) .. controls (263.0914, 644.35) and (265.3715, 637.094) .. (269.3333, 632) .. controls (273.08, 627.183) and (278.3308, 624.299) .. (283.878, 623.35);
			\draw (283.878, 688.65) .. controls (287.9989, 689.396) and (288.0007, 696.523) .. (292.0195, 696.088);
			\end{scope}
			\draw (184.9799, 640.1415) arc[start angle=95.4182, end angle=165.1895, x radius=32.0019, y radius=-32.0019] arc[start angle=99.1886, end angle=439.0423, radius=24.073] arc[start angle=-159.7432, end angle=-100.8718, radius=24.009];
			\draw (191.999, 640.2331) arc[start angle=-82.8539, end angle=180.4494, radius=32.0157] arc[start angle=97.1781, end angle=442.8177, radius=31.9998] arc[start angle=179.3023, end angle=442.9138, x radius=23.9596, y radius=-23.9596];
			\draw (60, 652) -- (38.4889, 651.985) .. controls (39.0914, 644.35) and (41.3715, 637.094) .. (45.3333, 632) .. controls (49.08, 627.1828) and (54.3308, 624.2994) .. (59.878, 623.3498) .. controls (64.0049, 622.5881) and (64.0202, 615.5286) .. (68.0195, 615.9117) .. controls (75.6195, 616.8428) and (82.8954, 620.4278) .. (88, 626.6667) .. controls (100, 641.3333) and (100, 670.6667) .. (88, 685.3333) .. controls (82.8954, 691.5722) and (75.6195, 695.1572) .. (68.0195, 696.0883);
			\draw [white, line width=1mm] (68.1363, 623.352) .. controls (64.0103, 622.799) and (64.0021, 615.565) .. (59.9997, 615.909);
			\draw (67.9948, 660.04) -- (89.5081, 660.053) .. controls (88.9115, 667.641) and (86.631, 674.903) .. (82.6667, 680) .. controls (78.9232, 684.813) and (73.6783, 687.6956) .. (68.1363, 688.6478) .. controls (64.0094, 689.3461) and (64.0197, 696.4344) .. (59.9997, 696.0907) .. controls (52.3928, 695.1633) and (45.1089, 691.5775) .. (40, 685.3333) .. controls (28, 670.6667) and (28, 641.3333) .. (40, 626.6667) .. controls (45.1089, 620.4225) and (52.3928, 616.8367) .. (59.9997, 615.9093) .. controls (64.0021, 615.5653) and (64.0103, 622.799) .. (68.1363, 623.3522) .. controls (73.6783, 624.3044) and (78.9232, 627.187) .. (82.6667, 632) .. controls (86.6293, 637.095) and (88.9095, 644.353) .. (89.5073, 651.937) -- (67.994, 651.924);
			\draw [white, line width=1mm] (59.878, 688.65) .. controls (63.9989, 689.396) and (64.0007, 696.523) .. (68.0195, 696.088);
			\draw (67.994, 651.924) .. controls (64, 652) and (64, 660) .. (60, 660);
			\draw [white, line width=1mm] (60, 652) .. controls (64, 652) and (64, 660) .. (67.9948, 660.04);
			\draw (60, 652) .. controls (64, 652) and (64, 660) .. (67.9948, 660.04);
			\draw (60, 660) -- (38.4867, 659.987) .. controls (39.0878, 667.639) and (41.3684, 674.902) .. (45.3333, 680) .. controls (49.08, 684.8172) and (54.3308, 687.7006) .. (59.878, 688.6502) .. controls (63.9989, 689.3959) and (64.0007, 696.5232) .. (68.0195, 696.0883);

			\begin{scope}[xshift=2cm]
			\node at (640, 576) {$4$};
			\node at (528, 576) {$4$};
			\node at (400, 576) {$4$};
			\node at (288, 576) {$4$};
			\end{scope}

			\begin{scope}[xshift=1cm]
			\draw[dotted] (232,820) -- (232,566);
			\end{scope}

			\node at (176, 576) {$8$};
			\node at (64, 576) {$12$};
			\node at (-32, 576) {$|\Aut(G)|$};
		\end{tikzpicture}
	\end{center}
\end{example}

\subsection{Moduli space of metric M\"obius graphs}
A metric on a M\"obius graph $G$ is an assignment of positive real values to its edges, that is, an element \smash{$\ell \in \R_{>0}^{E(G)}$}. A M\"obius graph equipped with a metric is called a \emph{metric M\"obius graph}. It is then natural to define an associated moduli space.

\begin{definition}
	For a given $(g,n) \in \frac{1}{2}\Z_{\ge 0} \times \Z_{>0}$ with $2g - 2 + n > 0$, define the \emph{moduli space of metric M\"obius graphs of type $(g,n)$} as
	\begin{equation}
		\Mod_{g,n}
		\coloneqq
		\Biggl(
			\bigsqcup_{G \in \MG_{g,n}} \frac{\R_{>0}^{E(G)}}{\Aut(G)}
		\Biggr)\Bigg/\!\!\sim\, ,
	\end{equation}
	where the orbicones are glued along boundary strata corresponding to edge degenerations. It is a real orbicell-complex\footnote{
		Strictly speaking, stabilisers may be non-trivial even at generic points, so $\Mod_{g,n}$ is not an orbifold; it is more naturally viewed as a stacky cell-complex. We will abuse terminology and call it an orbicell-complex.
	} of real dimension $6g-6+3n$.
\end{definition}

Define the perimeter map $p \colon \Mod_{g,n} \to \R_{>0}^n$, which sends a metric M\"obius graph to the $n$-tuple of lengths of its faces, ordered according to the prescribed labelling. From now on, we focus on the fibres of the perimeter map, that is, the moduli spaces of metric M\"obius graphs with fixed boundary lengths.

\begin{definition}
	For a fixed $L = (L_1,\dots,L_n) \in \R_{>0}^n$, define the \emph{moduli space of metric M\"obius graphs of type $(g,n)$ with fixed perimeters $L$} as
	\begin{equation}
		\Mod_{g,n}(L) \coloneqq p^{-1}(L).
	\end{equation}
	It is a real orbicell-complex of real dimension $6g-6+2n$. Notice that the dimension may be odd since the genus is a half-integer.
\end{definition}

The orbicell structure of $\Mod_{g,n}(L)$ can be described as follows. For a fixed $G \in \MG_{g,n}$, define the edge-face adjacency matrix $A_G = (a_{i,e})_{i=1,\ldots,n,\, e \in E(G)}$, where $a_{i,e}$ is the number of times the edge $e$ appears in the $i$-th face. In particular, $a_{i,e}$ is equal to $0$, $1$, or $2$, and the entries of every column of $A_G$ sum to $2$. The contribution of $G$ to $\Mod_{g,n}(L)$ is the orbifold polytope
\begin{equation}
	\frac{P_G(L)}{\Aut(G)},
	\qquad
	P_{G}(L) \coloneqq \Set{\ell \in \R_{>0}^{E(G)} \mid A_G\ell = L}.
\end{equation}
In what follows, we denote a metric M\"obius graph by $\bm{G}$, and use $G$ for the underlying M\"obius graph without metric. When necessary, the associated metric is denoted by $\ell_G$, or simply by $\ell$. Notice that the isotropy group of a point $\bm{G}$, denoted $\Aut(\bm{G})$, is the subgroup of $\Aut(G)$ that preserves the metric.

\subsection{Connected components and relation to the moduli of Riemann/Klein surfaces}
The moduli space of metric M\"obius graphs admits a natural decomposition according to whether the associated topological realisation is orientable:
\begin{equation}
	\Mod_{g,n}(L)
	=
	\Mod_{g,n}^+(L)
	\sqcup
	\Mod_{g,n}^-(L),
\end{equation}
where $\Mod_{g,n}^+(L)$ (resp.\ $\Mod_{g,n}^-(L)$) parametrises metric M\"obius graphs of type $(g,n)$ with fixed perimeters~$L$ whose topological realisation is orientable (resp.\ non-orientable).

\begin{figure}
	\centering
	\begin{tikzpicture}[x=1pt,y=1pt,scale=1]
	  \fill[black,opacity=.05](288, 696) -- (376, 696) -- (376, 752) -- cycle;
	  \fill[black,opacity=.05](376, 696) -- (464, 696) -- (376, 752) -- cycle;
	  \draw[densely dotted,opacity=.5](376, 640) -- (464, 696) -- (376, 752);
	  \draw[opacity=.5](288, 696) -- (376, 752) -- (376, 696);
	  \draw[opacity=.1](288, 696) -- (376, 640) -- (376, 696) -- (288, 696);
	  \draw[fill=white] (288, 696) circle[radius=1];
	  \draw[fill=white] (376, 640) circle[radius=1];
	  \draw[fill=white] (376, 752) circle[radius=1];
	  \draw[fill=white] (464, 696) circle[radius=1];
	  \draw(340, 688) -- (340, 688);
	  \fill[black,opacity=.05](256, 640) -- (192, 680) -- (192, 752) -- cycle;
	  \draw[opacity=.1](224, 696) -- (256, 640);
	  \draw[opacity=.1](256, 640) -- (192, 680) -- (192, 752);
	  \draw[opacity=.1](192, 680) -- (128, 640);
	  \draw[opacity=.1](256, 640) -- (128, 640) -- (192, 752);
	  \draw[opacity=.5] (224,696) -- (192,752);
	  \draw[fill=white] (192, 752) circle[radius=1];
	  \draw[fill=white] (256, 640) circle[radius=1];
	  \draw[fill=white] (128, 640) circle[radius=1];
	  \draw(204.3594, 722.371) circle[radius=8];
	  \draw(212.3594, 730.371) circle[radius=8];
	  \draw[BurntOrange](212, 701.3333) .. controls (217.3333, 701.3333) and (222.6667, 694.6667) .. (222.6667, 688) .. controls (222.6667, 681.3333) and (217.3333, 674.6667) .. (212, 674.6667) .. controls (206.6667, 674.6667) and (201.3333, 681.3333) .. (201.3333, 688) .. controls (201.3333, 694.6667) and (206.6667, 701.3333) .. cycle;
	  \draw[BurntOrange](201.333, 688) -- (222.667, 688);
	  \draw[->](196.797, 677.0019) arc[start angle=-32.0054, end angle=90, radius=5.6569];
	  \draw[->](192, 685.6569) arc[start angle=90, end angle=212.0054, radius=5.6569];
	  \draw[->](187.203, 677.0019) arc[start angle=-147.9946, end angle=-32.0054, radius=5.6569];
	  \draw[->](226.8066, 691.0885) arc[start angle=-60.2551, end angle=119.7449, radius=5.6569];
	  \node at (234, 706) {$\mathbb{Z}_2$};
	  \node at (192, 664) {$\mathbb{Z}_3$};
	  \node at (201.333, 688) {\tiny$\bullet$};
	  \node at (222.6667, 688) {\tiny$\bullet$};
	  \node at (204.3594, 730.371) {\tiny$\bullet$};
	  \node at (265, 688) {$\sqcup$};
	  \draw(328, 720) circle[radius=8];
	  \draw[BurntOrange](336, 728) circle[radius=8];
	  \draw[BurntOrange](370.343, 724) circle[radius=5.6569];
	  \draw[BurntOrange](381.6569, 724) circle[radius=5.6569];
	  \draw[BurntOrange](406.343, 696) circle[radius=5.6569];
	  \draw[BurntOrange](425.6569, 696) circle[radius=5.6569];
	  \draw(412, 696) -- (420, 696);
	  \draw[BurntOrange](344, 709.3333) .. controls (349.3333, 709.3333) and (354.6667, 702.6667) .. (354.6667, 696) .. controls (354.6667, 689.3333) and (349.3333, 682.6667) .. (344, 682.6667) .. controls (338.6667, 682.6667) and (333.3333, 689.3333) .. (333.3333, 696) .. controls (333.3333, 702.6667) and (338.6667, 709.3333) .. cycle;
	  \draw(333.3333, 695.9996) -- (354.6667, 695.9995);
	  \node at (412, 696) {\tiny$\bullet$};
	  \node at (420.0001, 696) {\tiny$\bullet$};
	  \node at (333.3333, 696.0004) {\tiny$\bullet$};
	  \node at (354.6667, 695.9995) {\tiny$\bullet$};
	  \node at (328, 728) {\tiny$\bullet$};
	  \node at (376, 724) {\tiny$\bullet$};
	  \draw[->](464, 684) arc[start angle=-90, end angle=90, radius=12];
	  \node at (480, 708) {$\mathbb{Z}_2$};
	  \node at (90, 688) {$\mathcal{N}_{1,1}(L_1) =$};
	\end{tikzpicture}
	\caption{
		A depiction of $\Mod_{1,1}(L_1)$, which has two connected components.
		The component $\Mod_{1,1}^+(L_1)$ (left) parametrises metric M\"obius graphs whose underlying surface is a one-holed torus. It consists of two cells: one $2$-dimensional cell, with a $\Z_3$-action as indicated, and one $1$-dimensional cell, with a $\Z_2$-action as indicated; in addition, a residual $\Z_4$ fixes every point in both cells.
		The component $\Mod_{1,1}^-(L_1)$ (right) parametrises metric M\"obius graphs whose underlying surface is a one-holed Klein bottle. It consists of four cells: two $2$-dimensional cells and two $1$-dimensional cells, each with a $\Z_2$-action as indicated; in addition, a residual $\Z_2$ fixes every point in all cells.
		The darker regions indicate a fundamental domain for the quotient by the corresponding automorphism group. The vertices of all cells are excluded from the moduli space, as are the dotted edges on the right.
	}
	\label{fig:moduli11}
\end{figure}
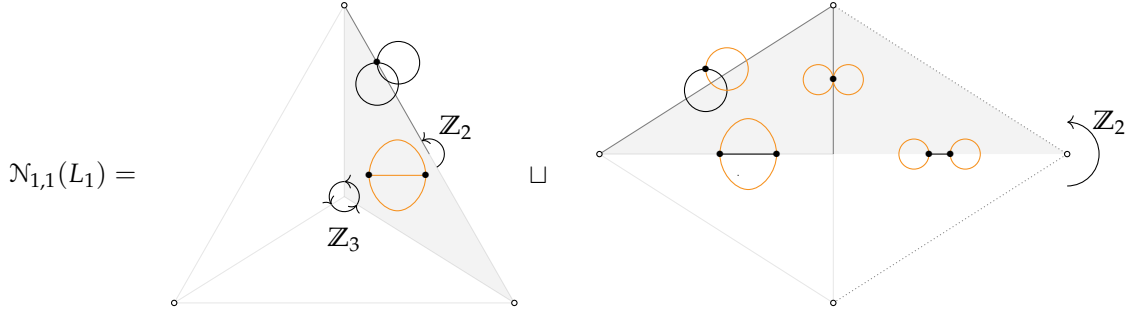

\subsubsection{The orientable component}
For $g\in\Z_{\ge 0}$, the space $\Mod_{g,n}^+(L)$ identifies with the moduli space of metric ribbon graphs with perimeters~$L$, and hence with the moduli space $\mc{M}_{g,n}$ of smooth curves (equivalently, Riemann surfaces) via Strebel differentials \cite{Str84} (see also \cite{Kon92,ABCGLW26}). Passing from oriented to orientable surfaces introduces a global orientation-reversing involution, and therefore, for fixed $L\in\R_{>0}^n$,
\begin{equation}\label{eq:iso:or}
	\Mod_{g,n}^+(L)
	\cong
	\frac{\mc{M}_{g,n}}{\Z_2}.
\end{equation}
Here and below, these identifications are understood as isomorphisms of topological real orbifolds. For $g\in\Z_{\ge 0}+\tfrac12$, the space $\Mod_{g,n}^+(L)$ is empty, since there are no orientable surfaces of genus~$g$.

\subsubsection{The non-orientable component}
The space $\Mod_{g,n}^-(L)$ is related to moduli of real curves (equivalently, Klein surfaces). Recall that a real curve (or symmetric Riemann surface) is a Riemann surface equipped with an anti-holomorphic involution; the associated Klein surface is its quotient, a (possibly non-orientable) surface endowed with a dianalytic structure.

The moduli space of smooth real curves has several connected components, classified by the number of connected components of the fixed locus of the anti-holomorphic involution \cite{GHJ01}. In this paper, we restrict to the component with empty fixed locus, equivalently the component consisting of Klein surfaces without boundary. Let $\mc{K}_{g,n}$ denote the moduli space of smooth genus $g$ Klein surfaces\footnote{
	We use the convention that the genus $g$ of a compact non-orientable surface is defined by $\chi=2-2g$, rather than the non-orientable genus $\tilde g$ defined by $\chi=1-\tilde g$. In \cite{GHJ01}, the genus is taken to be the non-orientable genus.
} with $n$ labelled marked points, together with a choice of local orientation at each marked point. For fixed $L\in\R_{>0}^n$, Goulden, Harer, and Jackson \cite{GHJ01} prove a homeomorphism (see also \cite{Bra12})
\begin{equation}\label{eq:iso:nor}
	\Mod_{g,n}^-(L) \cong \frac{\mc{K}_{g,n}}{\Z_2^n}.
\end{equation}
Again, this identification is meant as an isomorphism of topological real orbifolds. The quotient reflects the convention on local orientations at the marked points. The proof in \cite{GHJ01} again uses Strebel differentials, now on real curves, together with the fact that Strebel trajectories are invariant under the anti-holomorphic involution.

\subsection{A measure of non-orientability}\label{sec:MON}
Our next goal is to define a measure of non-orientability on metric M\"obius graphs, that is, a function $\rho$ quantifying how ``non-orientable'' a metric M\"obius graph is. Such a function should detect orientability: at $b=0$ it should coincide with the indicator function of the locus of orientable metric M\"obius graphs; at $b=1$ it should be the constant function~$1$; and for intermediate values $0<b<1$ it should interpolate continuously between these two cases.

In this section (and only in this section), we consider M\"obius graphs without face labelling and with no restriction on valency: one- and two-valent vertices are allowed. A metric is still an assignment of a positive real number to each edge. Thus, for such a M\"obius graph $G$, we seek a map
\begin{equation}
	\rho_G \colon \R_{>0}^{E(G)} \longrightarrow \R[b]
\end{equation}
with the properties above, following Chapuy and Do\l\k{e}ga \cite{CD22}. However, their algorithmic definition of $\rho_G$ is formulated for non-orientable maps with an ordered set of edges, whereas our graphs do not come equipped with such an ordering. To remedy this, we introduce the notion of a root and average over all possible choices of root. The definition of $\rho_G$ then proceeds recursively in the spirit of Tutte: at each step we remove the rooted edge, thereby simplifying the topology, and assign a weight depending on $b$ and the metric, in accordance with the topological type of the root removal.

\subsubsection{Rooting and induced face-orientations}
\label{ssub:rooting}

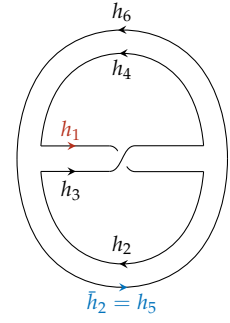
\begin{wrapfigure}[14]{R}{0.35\textwidth}
	\vspace{-1.75\baselineskip}
	\centering
	\begin{tikzpicture}[x=1pt,y=1pt,scale=1.2]
	  \draw (131.994, 747.924) .. controls (128, 748) and (128, 756) .. (124, 756);
	  \draw [white, line width=2mm] (124, 748) .. controls (128, 748) and (128, 756) .. (131.9948, 756.04);
	  \draw [postaction={
	    decorate,
	    decoration={markings, mark= at position 0.51 with {\arrow[BrickRed]{stealth}}}}] (102.4867, 755.987) -- (124, 756);
	  \node [text=BrickRed] at (112, 761) {\footnotesize$h_1$};
	  \draw (131.994, 747.924) -- (153.5073, 747.937);
	  \draw [postaction={
	    decorate,
	    decoration={markings, mark= at position 0.51 with {\arrow{stealth}}}}] (153.5073, 747.937) .. controls (152.9095, 740.353) and (150.6293, 733.095) .. (146.6667, 728) .. controls (137.3333, 716) and (118.6667, 716) .. (109.3333, 728) .. controls (105.3715, 733.094) and (103.0914, 740.35) .. (102.4889, 747.985);
	  \node at (128, 725) {\footnotesize$h_2$};
	  \draw [postaction={
	    decorate,
	    decoration={markings, mark= at position 0.51 with {\arrow{stealth}}}}] (102.4889, 747.985) -- (124, 748);
	  \draw (124, 748) .. controls (128, 748) and (128, 756) .. (131.9948, 756.04) -- (153.5081, 756.053);
	  \draw [postaction={
	    decorate,
	    decoration={markings, mark= at position 0.51 with {\arrow{stealth}}}}] (153.5081, 756.053) .. controls (152.9115, 763.641) and (150.631, 770.903) .. (146.6667, 776) .. controls (137.3333, 788) and (118.6667, 788) .. (109.3333, 776) .. controls (105.3684, 770.902) and (103.0878, 763.639) .. (102.4867, 755.987);
	  \node at (112, 742) {\footnotesize$h_3$};
	  \node at (128, 780) {\footnotesize$h_4$};
	  \node [text=RoyalBlue] at (128, 707) {\footnotesize$\bar{h}_2 = h_5$};
	  \node at (128, 798) {\footnotesize$h_6$};
	  \draw [postaction={
	    decorate,
	    decoration={markings, mark= at position 0.51 with {\arrow[RoyalBlue]{stealth reversed}}}}] (160.9998, 752.1415) .. controls (161.0257, 741.0944) and (158.0257, 730.0314) .. (152, 722.6667) .. controls (140, 708) and (116, 708) .. (104, 722.6667) .. controls (98.0257, 729.9686) and (95.0258, 740.9057) .. (95.0002, 751.8585);
	  \draw [postaction={
	    decorate,
	    decoration={markings, mark= at position 0.51 with {\arrow{stealth}}}}] (160.9998, 752.1415) .. controls (160.9742, 763.0943) and (157.9743, 774.0314) .. (152, 781.3333) .. controls (140, 796) and (116, 796) .. (104, 781.3333) .. controls (97.9743, 773.9686) and (94.9743, 762.9056) .. (95.0002, 751.8585);
	\end{tikzpicture}
	\caption{
		The face-orientation algorithm for a rooted M\"obius graph. The root is depicted in red; the first oriented half-edge in the second face is depicted in blue.
	}
	\label{fig:halfedge:ordering}
\end{wrapfigure}
A \emph{root} is the choice of one of the two sides of the ribbon corresponding to an edge, referred to as a \emph{half-edge}, together with an orientation of that half-edge. Graphically, a root is depicted on the topological realisation of a M\"obius graph by a small arrow drawn on the half-edge. Clearly, for any fixed edge there are four possible roots. In what follows, we denote a root by $r$, the half-edge and edge on which it lies by $h_r$ and $e_r$, respectively, and the set of all roots of a given graph $G$ by $R(G)$.

We now describe an algorithm that assigns an ordering to the set of half-edges of a connected rooted M\"obius graph $(G,r)$, and hence induces an orientation of each face. First, the root $r$ determines an ordering of the half-edges along the face containing it, say $(h_1 \coloneqq h_r, h_2, \ldots, h_m)$. To extend this ordering to all half-edges of the graph, we proceed as follows. Starting from the half-edge $h_1$, we traverse this face until we encounter a half-edge $h_i$ whose opposite half-edge $\bar{h}_i$ has not yet been assigned an order and therefore lies on a new face. We then declare $\bar{h}_i$ to be $h_{m+1}$, and assign to $h_{m+1}$ the orientation opposite to that of $h_i$. Using the orientation of $h_{m+1}$, we extend the ordering to the half-edges of this new face. Since $G$ is connected, this process eventually assigns an ordering to all half-edges of $G$. The ordering of half-edges obtained in this way induces an orientation of every face of $G$. Moreover, if $G$ is orientable, then these face orientations are compatible with the orientation of the associated surface, which is canonically determined by the root.

In summary, a root on a connected M\"obius graph canonically orients all of its faces; when $G$ is orientable, this agrees with the canonical surface orientation. See \cref{fig:halfedge:ordering} for a graphical illustration of this algorithm.

\subsubsection{Measure of non-orientability: definition and properties}
\label{ssub:MON}
We now define the measure of non-orientability by averaging over all possible rootings of a given M\"obius graph.

\begin{definition}\label{def:MON}
	Let $G$ be a M\"obius graph without face labelling and without any restriction on valency. The \emph{measure of non-orientability (MON)} is the function $\rho_G \colon \R_{>0}^{E(G)} \to \R[b]$ defined recursively on the number of edges by
	\begin{equation}\label{eq:def:rho}
		\rho_G(\ell;b)
		\coloneqq
		\frac{1}{\sum_{r \in R(G)} \ell_{e_r}}
		\sum_{r \in R(G)}
			\ell_{e_r} \, w_r(b) \, \rho_{G-e_r}(\ell - \ell_{e_r};b),
	\end{equation}
	where $G-e_r$ denotes the graph obtained from $G$ by removing the rooted edge $e_r$, endowed with the induced metric $(\ell - \ell_{e_r}) \in \R_{>0}^{E(G)-1}$ (that is, the metric obtained by deleting the entry $\ell_{e_r}$ from $\ell \in \R_{>0}^{E(G)}$). The weight $w_r(b)$ depends on the topological type of the root removal, and there are four mutually exclusive cases (see \cref{fig:weights}):
	\begin{enumerate}[label=(\roman*)]
		\item\label{D:disc}
		$G - e_r$ is disconnected. Set $w_r(b) = 1$.

		\item\label{R}
		$G - e_r$ is connected and has one fewer face than $G$. Set $w_r(b) = 1$.

		\item\label{E}
		$G - e_r$ is connected and has the same number of faces as $G$. Set $w_r(b) = b$.
		
		\item\label{D:conn}
		$G - e_r$ is connected and has one more face than $G$. Denote by $h'$ and $h''$ the half-edges immediately following and preceding the root, respectively; these inherit an orientation from the root. On $G-e_r$, take $h'$ with its orientation as the new root, which induces a new orientation on $h''$. If the old and new orientations of $h''$ agree, set $w_r(b) = 1$; if they disagree, set $w_r(b) = b$. See \cref{fig:G:Gtilde} for an example.
	\end{enumerate}
	For disconnected graphs, $\rho$ is defined multiplicatively as the product of the MONs of their connected components. For the base case, namely the graph $G=\bullet$ with no edges, set $\rho_{\bullet} \coloneqq 1$. Unless we need to emphasise the refinement parameter, we will drop the dependence on $b$ from the notation and simply write $\rho_G(\ell)$ and $w_r$.
\end{definition}

\begin{figure}[b]
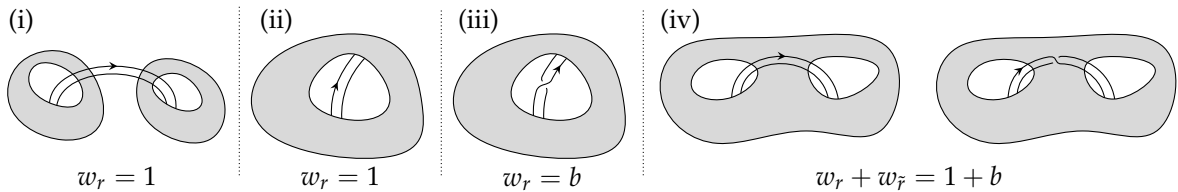

	\centering

	\caption{
		The four scenarios arising from a root-removal operation, together with their weights appearing in the MON.
	}
	\label{fig:weights}
\end{figure}

\begin{figure}[t]
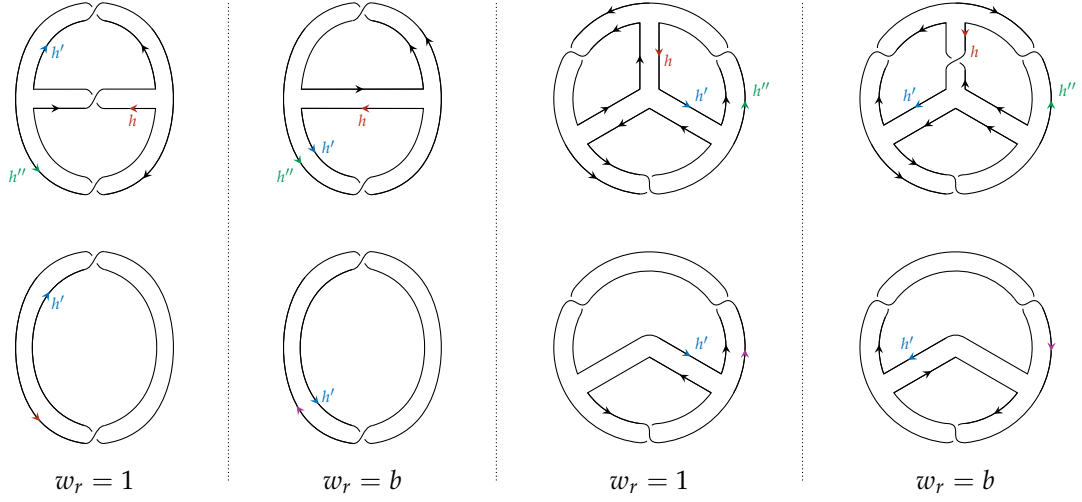

	\centering
	%
	\caption{
		Four examples of weight calculations in the situation of \labelcref{D:conn}. 
		The first row displays two M\"obius graphs of type $(1,1)$ (a torus and a Klein bottle), as well as two M\"obius graphs of type $(\tfrac{3}{2},1)$. 
		The root $h$ is shown in red, and the induced orientations on the half-edges immediately following and preceding the root, $h'$ and $h''$, are shown in blue and green, respectively. 
		The second row displays the graphs obtained after removing the rooted edge, with $h'$ taken as the new root; the new orientation on $h''$ induced by the new root is shown in purple.
	}
	\label{fig:G:Gtilde}
\end{figure}

It is easy to verify that the MON is well defined (i.e.~invariant under flips at vertices). Moreover, it is $\Aut(G)$-invariant, and therefore descends to a continuous function
\begin{equation}
	\rho \colon \Mod_{g,n}(L) \longrightarrow \R[b],
	\qquad
	\bm{G} \longmapsto \rho(\bm{G}),
\end{equation}
where $\rho(\bm{G})$ is shorthand for $\rho_G(\ell_G)$. We record here several basic properties of the MON, whose proofs are straightforward. These properties are particularly useful for practical computations, as they allow one to significantly simplify the combinatorics.

\begin{properties}\label{rem:removinglowvalent}
	The MON $\rho_G$ satisfies the following properties:
	\begin{description}
		\item[MON1]\label{prop:MON1}
		\emph{removing one-valent vertices.}
		The MON is unchanged by removing all tadpoles (i.e.~edges attached to a one-valent vertex).

		\item[MON2]\label{prop:MON2}
		\emph{removing two-valent vertices.}
		The MON is preserved under the removal of two-valent vertices, provided that the new edge obtained by merging two edges $e_1$ and $e_2$ is assigned length $\ell_{e_1}+\ell_{e_2}$ and colour equal to the sum of their colours in $\Z_2$.

		\item[MON3]\label{prop:MON3}
		\emph{orientable graphs.}
		If $G$ is orientable, then case~\labelcref{D:conn} always assigns weight $w_r=1$. Consequently, for orientable $G$ we have $\rho_G(\ell)=1$ for all metrics~$\ell$.

		\item[MON4]\label{prop:MON4}
		\emph{twisted vs.\ untwisted.}
		Suppose we are in case~\labelcref{D:conn} and that the assigned weight is $w_r=1$ (respectively, $w_r=b$). Construct a new graph $\tilde{G}$ that differs only in that the root has the opposite colour, and denote the resulting root by~$\tilde{r}$ (see \cref{fig:weights}). Then the corresponding weight is $w_{\tilde{r}}=b$ (respectively, $w_{\tilde{r}}=1$). In other words,
		\begin{equation}\label{eq:1:plus:b}
			w_r + w_{\tilde{r}} = 1 + b .
		\end{equation}
		Notice that the topological realisations of $G$ and $\tilde{G}$ may or may not be homeomorphic; see \cref{fig:G:Gtilde} for examples.
	\end{description}
\end{properties}

\begin{remark}
	One can obtain a simpler definition of the MON for unrooted graphs by summing over edges in \labelcref{eq:def:rho}, rather than over rootings. In that formulation, case~\labelcref{D:conn} (the only case in which the rooting is used) assigns the weight $\frac{1+b}{2}$, unless $G$ or $\tilde{G}$ is orientable. In the latter situation, one assigns weight $1$ to the orientable graph and $b$ to the other. An even simpler option is to always assign the weight $\frac{1+b}{2}$ in case~\labelcref{D:conn}, at the cost of losing the property that $\rho_G=1$ for orientable~$G$.
\end{remark}

Below we list all connected M\"obius graphs (up to flips) with zero, one, and two edges, together with their MONs. More intricate examples are given in \cref{app:MON}. We omit the metric whenever the MON is independent of it.
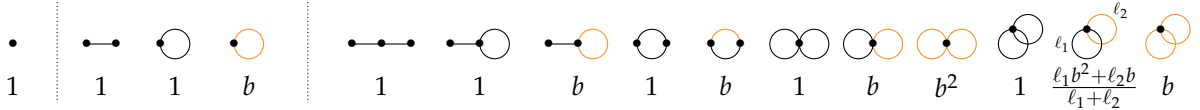
\begin{center}
\begin{tikzpicture}[scale=.65]
	\node at (0,0) {\scriptsize$\bullet$};
	\node at (0,-.85) {$1$};

	\node at (0,-.85) {$\vphantom{\frac{\ell_1 b^2 + \ell_2 b}{\ell_1 + \ell_2}}$};

	\draw[densely dotted] (.9,.85) -- (.9,-1.2);
	\begin{scope}[xshift=1.8cm]
		\draw (-.3,0) -- (.3,0);
		\node at (-.3,0) {\scriptsize$\bullet$};
		\node at (.3,0) {\scriptsize$\bullet$};
		\node at (0,-.85) {$1$};

		\begin{scope}[xshift=1.5cm]
			\draw (0,0) circle (.3cm);
			\node at (-.3,0) {\scriptsize$\bullet$};
			\node at (0,-.85) {$1$};
		\end{scope}

		\begin{scope}[xshift=3cm]
			\draw[BurntOrange] (0,0) circle (.3cm);
			\node at (-.3,0) {\scriptsize$\bullet$};
			\node at (0,-.85) {$b$};
		\end{scope}

		\node at (0,-.85) {$\vphantom{\frac{\ell_1 b^2 + \ell_2 b}{\ell_1 + \ell_2}}$};

		\draw[densely dotted] (4.2,.85) -- (4.2,-1.2);
	\end{scope}
	\begin{scope}[xshift=7.5cm]
		\draw (-.6,0) -- (0,0) -- (.6,0);
		\node at (-.6,0) {\scriptsize$\bullet$};
		\node at (0,0) {\scriptsize$\bullet$};
		\node at (.6,0) {\scriptsize$\bullet$};
		\node at (0,-.85) {$1$};

		\begin{scope}[xshift=2cm]
			\draw (-.6,0) -- (0,0);
			\draw (.3,0) circle (.3cm);
			\node at (-.6,0) {\scriptsize$\bullet$};
			\node at (0,0) {\scriptsize$\bullet$};
			\node at (0,-.85) {$1$};
		\end{scope}

		\begin{scope}[xshift=4cm]
			\draw (-.6,0) -- (0,0);
			\draw[BurntOrange] (.3,0) circle (.3cm);
			\node at (-.6,0) {\scriptsize$\bullet$};
			\node at (0,0) {\scriptsize$\bullet$};
			\node at (0,-.85) {$b$};
		\end{scope}

		\begin{scope}[xshift=5.5cm]
			\draw (0,0) circle (.3cm);
			\node at (-.3,0) {\scriptsize$\bullet$};
			\node at (.3,0) {\scriptsize$\bullet$};
			\node at (0,-.85) {$1$};
		\end{scope}

		\begin{scope}[xshift=7cm]
			\draw[BurntOrange] (.3,0) arc (0:180:.3cm);
			\draw (-.3,0) arc (180:360:.3cm);
			\node at (-.3,0) {\scriptsize$\bullet$};
			\node at (.3,0) {\scriptsize$\bullet$};
			\node at (0,-.85) {$b$};
		\end{scope}

		\begin{scope}[xshift=8.5cm]
			\draw (-.3,0) circle (.3cm);
			\draw (.3,0) circle (.3cm);
			\node at (0,0) {\scriptsize$\bullet$};
			\node at (0,-.85) {$1$};
		\end{scope}

		\begin{scope}[xshift=10cm]
			\draw (-.3,0) circle (.3cm);
			\draw[BurntOrange] (.3,0) circle (.3cm);
			\node at (0,0) {\scriptsize$\bullet$};
			\node at (0,-.85) {$b$};
		\end{scope}

			\begin{scope}[xshift=11.5cm]
			\draw[BurntOrange] (-.3,0) circle (.3cm);
			\draw[BurntOrange] (.3,0) circle (.3cm);
			\node at (0,0) {\scriptsize$\bullet$};
			\node at (0,-.85) {$b^2$};
		\end{scope}

		\begin{scope}[xshift=13cm]
			\draw (0.15,.3) circle (.3cm);
			\draw (-.15,0) circle (.3cm);
			\node at (-.15,.3) {\scriptsize$\bullet$};
			\node at (0,-.85) {$1$};
		\end{scope}

		\begin{scope}[xshift=14.5cm]
			\draw[BurntOrange] (0.15,.3) circle (.3cm);
			\draw (-.15,0) circle (.3cm);
			\node at (-.15,.3) {\scriptsize$\bullet$};
			\node at ($(0.15,.3) + (45:.55)$) {\tiny$\ell_2$};
			\node at ($(-.15,0) + (180:.5)$) {\tiny$\ell_1$};
			\node at (0,-.85) {$\frac{\ell_1 b^2 + \ell_2 b}{\ell_1 + \ell_2}$};
		\end{scope}

		\begin{scope}[xshift=16cm]
			\draw[BurntOrange] (0.15,.3) circle (.3cm);
			\draw[BurntOrange] (-.15,0) circle (.3cm);
			\node at (-.15,.3) {\scriptsize$\bullet$};
			\node at (0,-.85) {$b$};
		\end{scope}
	\end{scope}
\end{tikzpicture}
\end{center}

We conclude by collecting the main properties of the MON.

\begin{proposition}\label{prop:MON}
	The MON $\rho_G$ satisfies the following properties:
	\begin{enumerate}[label=(\arabic*)]
		\item\label{b0} It detects orientability: setting $b=0$, it vanishes if $G$ is non-orientable and equals $1$ if $G$ is orientable.

		\item\label{b1} Setting $b=1$ yields the constant function $1$.

		\item\label{poly:b} It is a polynomial in $b$ of degree at most $2g$ with positive real coefficients. If the metric on $G$ is integral, i.e.~all edge lengths are integers, then it is a polynomial in $b$ with positive rational coefficients.

		\item\label{rationality} It is a rational function of the edge lengths, homogeneous of degree zero. Moreover, for any fixed $b \in \R$, it is bounded on $\R_{>0}^{E(G)}$.
	\end{enumerate}
\end{proposition}

\begin{proof}
	Property~\labelcref{b0} follows by induction on the number of edges of $G$.
	Removing an edge of type~\labelcref{D:disc} or~\labelcref{R} does not change orientability, and the corresponding weights $w_r$ are equal to~$1$. Thus, in these cases the MON correctly detects orientability.
	Edges of type~\labelcref{E} cannot occur when $G$ is orientable, since this move decreases the genus by $\tfrac12$. As the corresponding weight is $w_r=b$, the claim follows in this case as well.

	It remains to consider the removal of an edge $e_r$ of type~\labelcref{D:conn}.
	If $G-e_r$ is non-orientable, then $G$ is necessarily non-orientable, and the induction hypothesis immediately yields the claim for~$G$.
	Assume instead that $G-e_r$ is orientable. If $G$ is also orientable, then orientability is preserved and the corresponding weight is~$1$.
	Conversely, if $G$ is non-orientable, then $\tilde{G}$ must be orientable. Hence the weight associated with removing $e_r$ from $G$ is $w_r=b$ by \cref{eq:1:plus:b}, and the desired property follows in this case as well.

	Property~\labelcref{b1} also follows by induction on the number of edges. Since $w_r(b)\big|_{b=1}=1$, we have
	\begin{equation}
		\rho_G(\ell)\big|_{b=1}
		=
		\frac{1}{\sum_{r \in R(G)} \ell_{e_r}}
		\sum_{r \in R(G)} \ell_{e_r}\,
		\underbrace{\rho_{G-e_r}(\ell - \ell_{e_r})\big|_{b=1}}_{=1}
		=
		1.
	\end{equation}

	For \labelcref{poly:b}, polynomiality in $b$ (with real or rational coefficients, depending on the metric) is immediate from the definition. As for the degree, note that contributions to the degree in $b$ arise only in cases~\labelcref{E} and~\labelcref{D:conn} in \cref{def:MON}. The genus decreases by $\tfrac12$ in case~\labelcref{E}, and by $1$ in case~\labelcref{D:conn}. Hence the degree is at most $2g$, with the bound governed by the former case.

	Finally, for \labelcref{rationality}, the fact that $\rho_G$ is a rational function of the edge lengths, homogeneous of degree zero, follows directly from the definition. For boundedness, set $M_G \coloneqq \sup_{\ell \in \R_{>0}^{E(G)}} |\rho_G(\ell)|$. Then
	\begin{equation}
		|\rho_G(\ell)|
		\le
		\frac{1}{\sum_{r \in R(G)} \ell_{e_r}}
		\sum_{r \in R(G)} \ell_{e_r}\, \max(1,|b|)\, M_{G-e_r}
		\le
		\max(1,|b|)\, \max_{r \in R(G)} M_{G-e_r},
	\end{equation}
	so $M_G \le \max(1,|b|)\max_{r} M_{G-e_r}$. The claim follows by induction on the number of edges.
\end{proof}

\section{The lattice point recursion}
\label{sec:lattice:point}

An interesting feature of the moduli space of metric M\"obius graphs is that it carries a natural integral structure, given by M\"obius graphs with integral edge lengths. This allows us to define a lattice point count weighted by the MON, which we call the \emph{refined lattice point count}. The main result of this section is the proof of \cref{thm:lattice:rec}, namely a recursion formula for the refined lattice point count.

\subsection{Integral structure and refined lattice point count}
The finite subset of $\Mod_{g,n}(L)$ consisting of all metric M\"obius graphs whose edge lengths are positive integers (rather than positive reals) defines an integral structure:
\begin{equation}
	\ZMod_{g,n}(L)
	\coloneqq
	\Set{
		\bm{G} \in \Mod_{g,n}(L) \mid \ell_G(e) \in \Z_{>0}
	}.
\end{equation}
It carries an orbifold structure induced by the integral polytopes $P_{G}^{\Z}(L) \coloneqq P_{G}(L) \cap \Z^{E(G)}$ modulo the automorphism group. Note that $\ZMod_{g,n}(L)$ is empty unless $\sum_{i=1}^n L_i$ is an even positive integer. Indeed, the sum of the perimeters is always twice the total edge length, namely
$\sum_{i=1}^n L_i = 2 \sum_{e \in E(G)} \ell_G(e)$.

We can now define the refined lattice point count of the moduli space of metric M\"obius graphs as the orbifold count weighted by the MON.

\begin{definition}\label{def:Ngn}
	For $L=(L_1,\ldots,L_n)\in \Z_{>0}^n$, define the \emph{refined lattice point count} by
	\begin{equation}
		N_{g,n}(L;b)
		\coloneqq
		\sum_{\bm{G} \in \ZMod_{g,n}(L)}
			\frac{\rho(\bm{G};b)}{|\Aut(\bm{G})|}.
	\end{equation}
	For notational simplicity, we omit the dependence on $b$.
\end{definition}

By the orbit-stabiliser theorem, this can be rewritten as
\begin{equation}
	N_{g,n}(L)
	=
	\sum_{G \in \MG_{g,n}}
		\frac{1}{|\Aut(G)|}
		\sum_{\ell \in P_G^\Z(L)} \rho_G(\ell).
\end{equation}
The following section is devoted to the computation of this count.

\subsection{The lattice point recursion}
We prove a recursion for the refined lattice point count, \cref{thm:lattice:rec}, which determines the $N_{g,n}$ uniquely from the initial conditions $N_{0,3}$, $N_{\frac{1}{2},2}$, and $N_{1,1}$ computed explicitly in \cref{app:base:top}. We proceed in three steps:
\begin{enumerate}
	\item
	Define ciliated integral metric M\"obius graphs together with a MON on them, and relate the corresponding counting problem to the original unciliated one.

	\item
	Establish, via a Tutte-like argument, a symmetric form of the lattice point recursion in which all boundary components are treated on the same footing.

	\item
	Prove that the symmetric recursion is equivalent to the asymmetric one (as presented in \cref{thm:lattice:rec}), in which the first boundary component is distinguished.
\end{enumerate}
See \cite{Nor10,DN11,CMB11} for similar arguments in the orientable setting.

\subsubsection{Ciliation}
Consider an integral metric M\"obius graph $\bm G$. Recall that a root $r$ is the choice of a half-edge of the graph $G$, together with an orientation of that half-edge. An integral metric can be visualised by subdividing each ribbon of $\bm G$ (in the topological realisation) into as many unit intervals as the length of the corresponding edge. A \emph{ciliation} on an integral metric M\"obius graph $\bm G$ consists of a choice of root, together with a choice of a unit-length segment on the rooted half-edge. Clearly, for a fixed $\bm G$ there are $4 \sum_{e \in E(G)} \ell_e$ possible ciliations. We represent a cilium by an arrow drawn on a unit-length segment of a half-edge, see \cref{fig:ciliation}.

\begin{figure}[ht]
	\begin{tikzpicture}[x=1pt,y=1pt,scale=.5]
	  \draw(64, 672) .. controls (106.6667, 704) and (154.6667, 720) .. (208, 720);
	  \draw[opacity=.5](60, 684) .. controls (102.6667, 716) and (150.6667, 732) .. (204, 732);
	  \draw(56, 696) .. controls (98.6667, 728) and (146.6667, 744) .. (200, 744);
	  \node at (60, 684) {\tiny$\bullet$};
	  \node at (204, 732) {\tiny$\bullet$};
	  \node at (112, 750) {$\ell_e = 3$};
	  \node at (272, 696) {$\Longrightarrow$};
	  \draw(352, 672) .. controls (394.6667, 704) and (442.6667, 720) .. (496, 720);
	  \draw[opacity=.5](348, 684) .. controls (390.6667, 716) and (438.6667, 732) .. (492, 732);
	  \node at (348, 684) {\tiny$\bullet$};
	  \node at (492, 732) {\tiny$\bullet$};
	  \draw[densely dotted](395.0018, 697.9948) .. controls (392.332, 705.9934) and (388.5897, 713.4778) .. (383.7749, 720.4482);
	  \draw[densely dotted](444.445, 714.6668) .. controls (443.0234, 722.9282) and (440.3565, 730.9282) .. (436.4442, 738.6666);
	  \draw(344, 696) .. controls (356.7864, 705.5898) and (370.0518, 713.7427) .. (383.7961, 720.4586);
	  \draw[postaction={
	    decorate,
	    decoration={markings, mark= at position 0.51 with {\arrow{stealth}}}}](383.7961, 720.4586) .. controls (400.6209, 728.6797) and (418.1634, 734.7476) .. (436.4237, 738.6622);
	  \draw(436.4237, 738.6622) .. controls (453.0227, 742.2207) and (470.2148, 744) .. (488, 744);
	  \draw(352, 736) .. controls (394.6667, 768) and (442.6667, 784) .. (496, 784);
	  \draw[opacity=.5](348, 748) .. controls (390.6667, 780) and (438.6667, 796) .. (492, 796);
	  \node at (348, 748) {\tiny$\bullet$};
	  \node at (492, 796) {\tiny$\bullet$};
	  \draw[densely dotted](395.002, 761.995) .. controls (392.332, 769.9937) and (388.5897, 777.478) .. (383.775, 784.448);
	  \draw[densely dotted](444.445, 778.667) .. controls (443.0237, 786.9283) and (440.3567, 794.9283) .. (436.444, 802.667);
	  \draw[postaction={
	    decorate,
	    decoration={markings, mark= at position 0.51 with {\arrow{stealth}}}}](344, 760) .. controls (356.7867, 769.59) and (370.052, 777.743) .. (383.796, 784.459);
	  \draw(383.796, 784.459) .. controls (400.6207, 792.6797) and (418.1633, 798.7473) .. (436.424, 802.662);
	  \draw(436.424, 802.662) .. controls (453.0227, 806.2207) and (470.2147, 808) .. (488, 808);
	  \draw(352, 608) .. controls (394.6667, 640) and (442.6667, 656) .. (496, 656);
	  \draw[opacity=.5](348, 620) .. controls (390.6667, 652) and (438.6667, 668) .. (492, 668);
	  \node at (348, 620) {\tiny$\bullet$};
	  \node at (492, 668) {\tiny$\bullet$};
	  \draw[densely dotted](395.002, 633.995) .. controls (392.332, 641.9937) and (388.5897, 649.478) .. (383.775, 656.448);
	  \draw[densely dotted](444.445, 650.667) .. controls (443.0237, 658.9283) and (440.3567, 666.9283) .. (436.444, 674.667);
	  \draw(344, 632) .. controls (356.7867, 641.59) and (370.052, 649.743) .. (383.796, 656.459);
	  \draw(383.796, 656.459) .. controls (400.6207, 664.6797) and (418.1633, 670.7473) .. (436.424, 674.662);
	  \draw[postaction={
	    decorate,
	    decoration={markings, mark= at position 0.51 with {\arrow{stealth}}}}](436.424, 674.662) .. controls (453.0227, 678.2207) and (470.2147, 680) .. (488, 680);
	\end{tikzpicture}
	\caption{
		A ciliation on an edge of length $3$: an oriented half-edge together with a choice of unit segment, shown by an arrow. On the right we display three of the twelve possible ciliation choices on this edge. The remaining choices are obtained by reversing the arrow orientation, or by placing the cilium on the opposite half-edge in its six possible positions and orientations.
	}
	\label{fig:ciliation}
\end{figure}
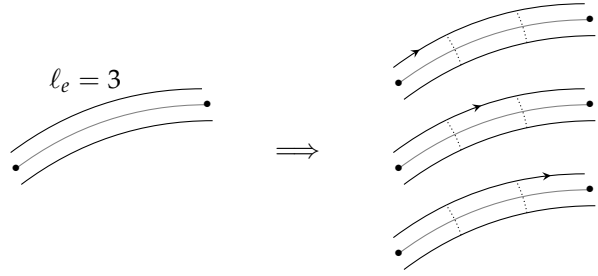

Denote by $\Mod'_{g,n}(L)$ the moduli space of ciliated metric M\"obius graphs of type $(g,n)$. Let $P'_G(L)$ be the set of integral metrics on $G$ with fixed perimeters $L$, together with all possible ciliations. This consists of $4 \sum_{e \in E(G)} \ell_e = 2(L_1+\cdots+L_n)$ copies of the integral polytopes $P_G^{\Z}(L)$:
\begin{equation}
	P'_G(L) = P_G^\Z(L)^{\sqcup \, 2(L_1 + \cdots + L_n)}.
\end{equation}
We denote an element of $P'_G(L)$ by $(\ell,c)$, where $c$ is a ciliation with root $r_c$ placed on an edge $e_c$.

\subsubsection{Trimming}
Removing an edge from a M\"obius graph $G$ whose vertices all have valency at least $3$ may produce a graph with lower-valent vertices. Since our lattice point recursion is based on an edge-removal procedure, we need a way to eliminate any one- or two-valent vertices that can arise in this process. We call this operation \emph{trimming}.

Let $(G,\ell)$ be a metric M\"obius graph of type $(g,n)$ with $2g-2+n>1$, all vertices of valency at least $3$, and let $e$ be a distinguished edge. We define the \emph{trimmed} M\"obius graph $\tr(G-e)$, endowed with the metric $\tr(\ell-\ell_e)$, by the following two-step procedure:
\begin{enumerate}
 	\item If $e$ is part of a lollipop (i.e.~a single edge attached to a loop, see \cref{fig:lollipop}), remove the entire lollipop. Otherwise, remove only the edge $e$. The metric is obtained by deleting the lengths of all edges that are removed.

	\item
	Remove all two-valent vertices. If a two-valent vertex lies between edges $e_1$ and $e_2$, then the new edge created by its removal is assigned length $\ell_{e_1}+\ell_{e_2}$ and colour equal to the sum of the colours in $\Z_2$.
\end{enumerate}
We note that trimming is needed only when $e$ is part of a lollipop and/or incident to a trivalent vertex, see \cref{fig:trimming}. It is also straightforward to check that trimming is compatible with flips.

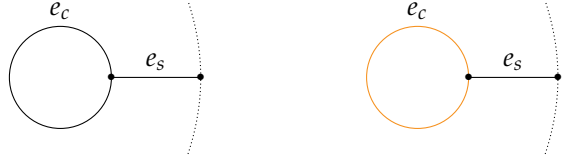
\begin{figure}[t]
	\begin{tikzpicture}[x=1pt,y=1pt,scale=.6]
	  \draw[densely dotted](128, 704) .. controls (138.6667, 736) and (138.6667, 768) .. (128, 800);
	  \draw(136, 752) -- (80, 752);
	  \draw(48, 752) circle[radius=32];
	  \node at (80, 752) {\tiny$\bullet$};
	  \node at (136, 752) {\tiny$\bullet$};
	  \node at (48, 794) {$e_c$};
	  \node at (108, 762) {$e_s$};
	  \draw[densely dotted](352, 704) .. controls (362.6667, 736) and (362.6667, 768) .. (352, 800);
	  \draw(360, 752) -- (304, 752);
	  \draw[BurntOrange](272, 752) circle[radius=32];
	  \node at (304, 752) {\tiny$\bullet$};
	  \node at (360, 752) {\tiny$\bullet$};
	  \node at (272, 794) {$e_c$};
	  \node at (332, 762) {$e_s$};
	\end{tikzpicture}
	\caption{
		A graphical representation of an untwisted and a twisted lollipop (left and right, respectively). The difference is whether the edge $e_c$ corresponding to the ``candy'' is twisted. The edge $e_s$ corresponding to the ``stick'' can always be untwisted by a flip. The candy has strictly positive length, whereas the stick may have length~$0$.
	}
	\label{fig:lollipop}
\end{figure}

\begin{figure}[t]
	\begin{tikzpicture}[x=1pt,y=1pt,scale=.4]
	  \draw(63.6593, 652.2551) -- (96, 704) -- (63.6584, 755.7442);
	  \draw[densely dotted](32, 736) .. controls (53.3333, 746.6667) and (74.6667, 762.6667) .. (96, 784);
	  \draw[densely dotted](32, 672) .. controls (53.3333, 661.3333) and (74.6667, 645.3333) .. (96, 624);
	  \draw(160, 704) circle[radius=32];
	  \draw(96, 704) -- (128, 704);
	  \node at (96, 704) {\tiny$\bullet$};
	  \node at (128, 704) {\tiny$\bullet$};
	  \node at (63.6593, 652.2551) {\tiny$\bullet$};
	  \node at (63.6584, 755.7442) {\tiny$\bullet$};
	  \draw[->](232, 704) -- (280, 704);
	  \draw[densely dotted](320, 736) .. controls (341.3333, 746.6667) and (362.6667, 762.6667) .. (384, 784);
	  \draw[densely dotted](320, 672) .. controls (341.3333, 661.3333) and (362.6667, 645.3333) .. (384, 624);
	  \node at (351.6593, 652.255) {\tiny$\bullet$};
	  \node at (351.6584, 755.744) {\tiny$\bullet$};
	  \draw (351.6593, 652.255) .. controls (373.2198, 686.7517) and (373.2195, 721.248) .. (351.6584, 755.744);
	  \node at (80, 738) [below left] {$e_1$};
	  \node at (80, 670) [above left] {$e_2$};
	  \node at (120, 720) {$e$};
	  \node at (380, 704) {$\hat{e}$};

	  \begin{scope}[xshift=18cm]
	    \draw[BurntOrange] (63.6593, 652.255) -- (96, 704);
	    \draw (96, 704) -- (63.6584, 755.744);
	    \draw[densely dotted](32, 736) .. controls (53.3333, 746.6667) and (74.6667, 762.6667) .. (96, 784);
	    \draw[densely dotted](32, 672) .. controls (53.3333, 661.3333) and (74.6667, 645.3333) .. (96, 624);
	    \node at (96, 704) {\tiny$\bullet$};
	    \node at (63.6593, 652.255) {\tiny$\bullet$};
	    \node at (63.6584, 755.744) {\tiny$\bullet$};
	    \node at (80, 738) [below left] {$e_1$};
	    \node at (80, 670) [above left] {$e_2$};
	    \node at (128, 717) {$e$};
	    \draw(96, 704) -- (160, 704);
	    \node at (160, 704) {\tiny$\bullet$};
	    \draw(160, 704) -- (195.0119, 753.0158);
	    \draw(160, 704) -- (209.723, 720.5746);
	    \draw(160, 704) -- (193.6005, 653.596);
	    \node at (193.6005, 653.596) {\tiny$\bullet$};
	    \node at (209.723, 720.5746) {\tiny$\bullet$};
	    \node at (195.0119, 753.0158) {\tiny$\bullet$};
	    \draw[densely dotted](184, 640) .. controls (189.3333, 650.6667) and (197.3333, 658.6667) .. (208, 664);
	    \draw[densely dotted](216, 712) .. controls (210.6667, 717.3333) and (208, 722.6667) .. (208, 728);
	    \draw[densely dotted](208, 744) .. controls (197.3333, 749.3333) and (189.3333, 757.3333) .. (184, 768);
	    \draw[densely dotted](320, 736) .. controls (341.3333, 746.6667) and (362.6667, 762.6667) .. (384, 784);
	    \draw[densely dotted](320, 672) .. controls (341.3333, 661.3333) and (362.6667, 645.3333) .. (384, 624);
	    \draw[BurntOrange] (351.659, 652.255) .. controls (373.2197, 686.7517) and (373.2193, 721.248) .. (351.658, 755.744);
	    \draw[->](240, 704) -- (288, 704);
	    \draw(448, 704) -- (483.012, 753.016);
	    \draw(448, 704) -- (497.723, 720.575);
	    \draw(448, 704) -- (481.601, 653.596);
	    \draw[densely dotted](472, 640) .. controls (477.3333, 650.6667) and (485.3333, 658.6667) .. (496, 664);
	    \draw[densely dotted](504, 712) .. controls (498.6667, 717.3333) and (496, 722.6667) .. (496, 728);
	    \draw[densely dotted](496, 744) .. controls (485.3333, 749.3333) and (477.3333, 757.3333) .. (472, 768);
	    \node at (351.6593, 652.255) {\tiny$\bullet$};
	    \node at (448, 704) {\tiny$\bullet$};
	    \node at (351.6584, 755.744) {\tiny$\bullet$};
	    \node at (481.6005, 653.596) {\tiny$\bullet$};
	    \node at (497.723, 720.5746) {\tiny$\bullet$};
	    \node at (483.0119, 753.0158) {\tiny$\bullet$};
	    \node at (380, 704) {$\hat{e}$};
	  \end{scope}
	\end{tikzpicture}
	\caption{
		On the left panel, the edge $e$ is part of a lollipop; on the right panel it is not. In both cases, the resulting edge $\hat{e}$ has length $\ell_{\hat{e}} = \ell_{e_1} + \ell_{e_2}$.
	}
	\label{fig:trimming}
\end{figure}
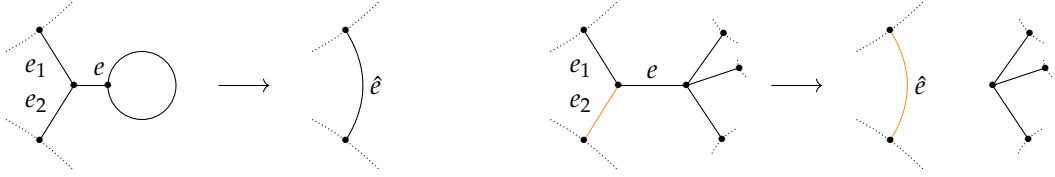

\subsubsection{Measure of non-orientability for ciliated graphs}
\label{ssub:mon:ciliated}
We define the MON for ciliated M\"obius graphs without face labelling, but with all vertices of valency at least $3$. Given a ciliated graph $(G,c,\ell)$ with root $r=r_c$, we define the MON $\rho'$ recursively by
\begin{equation}
	\rho'_{(G,r_c)}(\ell;b)
	\coloneqq
	w'_{r}(b) \, \rho_{\tr(G - e_r)}\bigl(\tr(\ell - \ell_{e_r})\bigr),
\end{equation}
where $G-e_r$ denotes the graph obtained from $G$ by removing the ciliated edge $e_r$. Note that the MON on the right-hand side is the unciliated one. Moreover, $\rho'$ depends only on the root, and not on the additional data of the ciliation.

The weight $w'_r$ differs slightly from the one in \cref{def:MON} for rooted graphs. When we remove the rooted edge $e_r$ from $G$, there are three mutually exclusive possibilities:
\begin{description}
	\item[$\boldsymbol{\mc{R}}$-type]
	If $\tr(G-e_r)$ is connected and has one fewer face than $G$, set $w'_r(b)=1$.

	\item[$\boldsymbol{\mc{E}}$-type]
	If $\tr(G-e_r)$ is connected and has the same number of faces as $G$, set $w'_r(b)=b$.

	\item[$\boldsymbol{\mc{D}}$-type]
	If $\tr(G-e_r)$ has one more face than $G$, we distinguish two cases:
	\begin{itemize}
		\item\textbf{$\boldsymbol{\mc{D}^{\rm c}}$-type:}
		If $\tr(G-e_r)$ is connected, define $w'_r$ as follows. The edge $e_r$ connects two faces of $\tr(G-e_r)$, say $f_1$ and $f_2$. On $\tr(G-e_r)$, place a new root $\bar r$ on the half-edge immediately following $r$, and let $f_1$ be the face containing $\bar r$. Use the orientation of $r$ to orient $f_2$. If this orientation agrees with the orientation of $f_2$ induced by the root $\bar r$ (via the algorithm of \cref{ssub:rooting}), set $w'_r(b)=1$; otherwise, set $w'_r(b)=b$.

		\item\textbf{$\boldsymbol{\mc{D}^{\rm d}}$-type:}
		If $\tr(G-e_r)$ is disconnected, set $w'_r(b)=1$.
	\end{itemize}
\end{description}

We define the ciliated refined lattice point count by
\begin{equation}\label{eq:cil:count}
	N'_{g,n}(L;b)
	\coloneqq
	\sum_{ G \in \MG_{g,n}} \frac{1}{|\Aut(G)|}
		\sum_{(\ell,c) \in P'_G(L)} \rho'_{(G,r_c)}(\ell;b),
\end{equation}
where we recall that $c$ is a ciliation with underlying root $r_c$. As before, we drop the dependence on $b$ from the notation in $N'_{g,n}(L)$.

\begin{lemma}\label{lem:cil:to:uncil}
	The refined lattice point counts in the ciliated and unciliated settings are related by
	\begin{equation}
		2\biggl( \sum_{i=1}^n L_i \biggr) N_{g,n}(L) = N'_{g,n}(L).
	\end{equation}
\end{lemma}

\begin{proof}
	Since the MON for ciliated graphs depends on the root but not on the specific location of the cilium, for a fixed graph $G$ the sum over ciliations in \cref{eq:cil:count} can be written as
	\begin{equation}
		\sum_{(\ell,c) \in P'_G(L)} \rho'_{(G,r_c)}(\ell)
		=
		\sum_{\ell \in P^{\Z}_G(L)} \sum_{r \in R(G)} \ell_{e_r}\, w'_{r}\, \rho_{\tr(G-e_r)}\bigl(\tr(\ell-\ell_{e_r})\bigr).
	\end{equation}
	On the other hand, the corresponding sum appearing in the definition of $N_{g,n}(L)$ can be rewritten as
	\begin{equation}
		\sum_{\ell \in P^{\Z}_G(L)} \rho_{G}(\ell)
		=
		\frac{1}{\sum_{r\in R(G)} \ell_{e_r}}
		\sum_{\ell \in P^{\Z}_G(L)} \sum_{r \in R(G)} \ell_{e_r}\, w_{r}\, \rho_{G-e_r}(\ell-\ell_{e_r}).
	\end{equation}
	Thus, using the identity $\sum_{r \in R(G)} \ell_{e_r} = 2 \sum_{i=1}^n L_i$, the lemma follows from the claim
	\begin{equation}\label{eq:claim}
		\sum_{r \in R(G)} \ell_{e_r}\, w'_{r}\, \rho_{\tr(G-e_r)}\bigl(\tr(\ell-\ell_{e_r})\bigr)
		\overset{?}{=}
		\sum_{r \in R(G)} \ell_{e_r}\, w_{r}\, \rho_{G-e_r}(\ell-\ell_{e_r}).
	\end{equation}
	In order to prove the claim, notice that if $G-e_r = \tr(G-e_r)$, then $w'_r=w_r$ and the corresponding terms in \cref{eq:claim} agree. Hence it suffices to consider the cases in which trimming is required, i.e. when $e_r$ belongs to a (twisted or untwisted) lollipop or is incident to a trivalent vertex. Consider, for instance, a twisted lollipop with stick $e_{s}$ and candy $e_{c}$. Then
	\begin{equation}
	\begin{split}
		\ell_{e_{s}} w_{e_{s}}\, \rho_{G-e_{s}} (\ell - \ell_{e_{s}}) + \ell_{e_{c}} w_{e_{c}}\, \rho_{G-e_{c}} (\ell - \ell_{e_{c}})
		&=
		\ell_{e_{s}}\, \rho_{G-e_{s}} (\ell - \ell_{e_{s}}) + \ell_{e_{c}}\, b\, \rho_{G-e_{c}} (\ell - \ell_{e_{c}}) \\
		&=
		(\ell_{e_{s}}+\ell_{e_{c}})\, b\, \rho_{G-e_{s}-e_{c}} (\ell - \ell_{e_{s}}-\ell_{e_{c}}) \\
		&=
		(\ell_{e_{s}}+\ell_{e_{c}})\, b\, \rho_{\tr(G-e_{s})} \!\bigl(\tr(\ell - \ell_{e_{s}})\bigr),
	\end{split}
	\end{equation}
	where the second equality uses that $\rho_{G-e_{s}} (\ell - \ell_{e_{s}}) = b\, \rho_{G-e_{s}-e_{c}} (\ell - \ell_{e_{s}}-\ell_{e_{c}})$ (since $G-e_{s}$ is disconnected and the M\"obius strip contributes a factor $b$), and that $\rho_{G-e_{c}} (\ell - \ell_{e_{c}}) = \rho_{G-e_{s}-e_{c}} (\ell - \ell_{e_{s}}-\ell_{e_{c}})$ (since $e_{s}$ is a tadpole in $G-e_{c}$ and removing tadpoles leaves the MON unchanged by property~\hyperref[prop:MON1]{MON1}). The last equality follows because removing two-valent vertices leaves the MON unchanged (by property~\hyperref[prop:MON2]{MON2}). The case of an untwisted lollipop is completely analogous and is therefore omitted.

	Finally, if $e_r$ is incident to a trivalent vertex, then removing all two-valent vertices in $G-e_r$ does not affect the MON (again by property~\hyperref[prop:MON2]{MON2}), and the claim follows.
\end{proof}

\subsubsection{The symmetric recursion}
Our next goal is to prove a recursion for the ciliated refined lattice point count via a Tutte-style decomposition obtained by removing the rooted edge. Using the notation $L_{[n]} \coloneqq (L_1,\ldots,L_n)$, root removal followed by trimming defines a surjective map
\begin{multline}\label{eq:pi:surj}
	\pi \colon \Mod'_{g,n}\big(L_{[n]}\big)
	\longrightarrow
	\bigsqcup_{\substack{i,j=1 \\ i \neq j}}^n \bigsqcup_{p=1}^{L_i + L_j - 1}
	\ZMod_{g,n-1}\big(p,L_{[n]\setminus \{i,j\}}\big) \sqcup
	\bigsqcup_{i=1}^n \bigsqcup_{p=1}^{L_i-1}
	\ZMod_{g-\frac{1}{2},n}\big(p,L_{[n]\setminus\{i\}}\big)
	\\
	\sqcup
	\bigsqcup_{i=1}^n \bigsqcup_{\substack{p,\, q = 1 \\ p+q < L_i}}^{L_i-2}
	\Bigg( \ZMod_{g-1,n+1} \big(p,q,L_{[n]\setminus \{i\}}\big)
	\sqcup \bigsqcup_{\substack{g_1+g_2=g \\ I_1 \sqcup I_2=[n] \setminus \{i\}}}
	\ZMod_{g_1,1+|I_1|} \big(p,L_{I_1}\big)
	\times
	\ZMod_{g_2,1+|I_2|} \big(q,L_{I_2}\big) \Bigg).
\end{multline}
The terms on the right-hand side correspond, in order, to root removals of type $\mc{R}$, $\mc{E}$, $\mc{D}^{\rm c}$, and $\mc{D}^{\rm d}$; the roles of $p$ and $q$ are explained case by case below. We then count the fibres of $\pi$ with MON and automorphism weights. Since the MON is constant on fibres and agrees with the MON downstairs up to a potential factor of $b$, summing over the base yields the ciliated count.

With this setup, we obtain the following symmetric version of the lattice point recursion.

\begin{theorem}\label{thm:lattice:rec:sym}
	For $2g-2+n > 1$, the refined lattice point count satisfies the recursion relation
	\begin{equation}\label{eq:lattice:rec:sym}
	\begin{split}
		2 \left( \sum_{i=1}^n L_i \right) N_{g,n}\big(L_{[n]}\big)
		&=
		\sum_{\substack{i,j = 1 \\ i \neq j}}^n \sum_{p >0}
			p \, [L_i+L_j-p]_+ \,
			N_{g,n-1}\big(p,L_{[n]\setminus \{i,j\}} \big) \\
		&\qquad
		+
		b \sum_{i=1}^n \sum_{p >0}
			p \, (L_i-1) \, [L_i-p]_+ \,
			N_{g-\frac{1}{2},n}\big(p,L_{[n]\setminus\{i\}}\big) \\
		&\qquad\qquad
		+
		\sum_{i=1}^n \sum_{ p,\,q >0 }
			p q \, [L_i-p-q]_+
			\Bigg(
				(1+b) N_{g-1,n+1}\big(p,q,L_{[n]\setminus \{i\}}\big) \\
		&\qquad\qquad\qquad\qquad 
				+
				2 \sum_{\substack{g_1+g_2 = g \\ I_1 \sqcup I_2 = [n]\setminus\{i\}}}
					N_{g_1,1+|I_1|}(p,L_{I_1})
					N_{g_2,1+|I_2|}(q,L_{I_2})
			\Bigg) ,
	\end{split}
	\end{equation}
	where $[x]_+ \coloneqq \max(x,0)$ is the ramp function. 
\end{theorem}

\begin{proof}
	The idea of the proof is to express $N'_{g,n}(L)$ as a weighted count of the fibres of the map $\pi$ defined in \cref{eq:pi:surj}. More precisely, denoting the target of $\pi$ by $\widehat{\mathcal N}_{g,n}(L)$, we rewrite $N'_{g,n}(L)$ as a sum over fibres over (unciliated) integral metric graphs $(\widehat{G},\widehat{\ell})$ obtained by removing the root from a ciliated integral metric graph $(G,\ell,c)$:
	\begin{equation} \label{eq:Ngn over fibers}
		N'_{g,n}(L)
		=
		\sum_{(\widehat{G},\widehat{\ell}) \in \widehat{\mathcal N}_{g,n}(L)} \;
		\sum_{\substack{(G,\ell,c) \in \Mod'_{g,n}(L) \\ \pi(G,\ell,c) = (\widehat{G},\widehat{\ell})}}
			\frac{\rho'_{G,r_c}(\ell)}{|\Aut(G,\ell,c)|}.
	\end{equation}
	For a fixed base point $(\widehat{G},\widehat{\ell})$, we call \emph{attachment data} the discrete choices needed to reconstruct a ciliated graph in $\mathcal N'_{g,n}(L)$ mapping to $(\widehat{G},\widehat{\ell})$: namely, a choice of how to attach the removed edge together with a choice of ciliation on the resulting rooted half-edge. We denote by $\Att(\widehat{G},\widehat{\ell})$ the set of all such attachment data. The automorphism group $\Aut(\widehat{G},\widehat{\ell})$ acts naturally on $\Att(\widehat{G},\widehat{\ell})$, and two attachment data that differ by such an automorphism produce isomorphic graphs in $\mathcal N'_{g,n}(L)$.

	If $a\in \Att(\widehat{G},\widehat{\ell})$ yields the ciliated metric graph $\bm{G}_a\in \mathcal N'_{g,n}(L)$, then $\Aut(\bm{G}_a)$ is naturally identified with the stabiliser of $a$ in $\Aut(\widehat{G},\widehat{\ell})$. Thus, by the orbit-stabiliser theorem,
	\begin{equation}\label{eq:orbit-stabilizer}
		\frac{1}{|\Aut(\bm{G}_a)|} = \frac{|O_a|}{|\Aut(\widehat{G},\widehat{\ell})|},
	\end{equation}
	where $O_a$ is the orbit of $a$ under the action of $\Aut(\widehat{G},\widehat{\ell})$. Thus, the inner sum in \cref{eq:Ngn over fibers} can be rewritten as
	\begin{multline}
		\sum_{\substack{(G,\ell,c) \in \Mod'_{g,n}(L) \\ \pi(G,\ell,c) = (\widehat{G},\widehat{\ell})}}
		\frac{\rho'_{G,r_c}(\ell)}{|\Aut(G,\ell,c)|}
		=
		\sum_{\substack{(G,\ell,c) \in \Mod'_{g,n}(L) \\ \pi(G,\ell,c) = (\widehat{G},\widehat{\ell})}}
		\frac{w'_{r_c}\,\rho_{\widehat{G}}(\widehat{\ell})}{|\Aut(G,\ell,c)|}
		=
		\sum_{[a] \in \Att(\widehat{G},\widehat{\ell})/\Aut(\widehat{G},\widehat{\ell})}
		\frac{w_{\widehat{G}}\,\rho_{\widehat{G}}(\widehat{\ell})}{|\Aut(\bm{G}_a)|} \\ 
		=
		w_{\widehat{G}}\,\rho_{\widehat{G}}(\widehat{\ell})
		\sum_{[a] \in \Att(\widehat{G},\widehat{\ell})/\Aut(\widehat{G},\widehat{\ell})}
		\frac{|O_a|}{|\Aut(\widehat{G},\widehat{\ell})|}
		=
		w_{\widehat{G}}\,
		\frac{\rho_{\widehat{G}}(\widehat{\ell})}{|\Aut(\widehat{G},\widehat{\ell})|}
		\,|\Att(\widehat{G},\widehat{\ell})|\, .
	\end{multline}
	Here the first equality is the definition of the MON for ciliated graphs. The second equality uses that isomorphism classes of $(G,\ell,c)$ mapping to $(\widehat{G},\widehat{\ell})$ are in bijection with $\Aut(\widehat{G},\widehat{\ell})$-orbits in $\Att(\widehat{G},\widehat{\ell})$. We denote the common value of $w'_{r_c}$, which is constant along the fibre of $\pi$, by $w_{\widehat{G}}$. The third equality uses \cref{eq:orbit-stabilizer}, and the last equality follows because $\Att(\widehat{G},\widehat{\ell})$ is partitioned into $\Aut(\widehat{G},\widehat{\ell})$-orbits.

	Thus, we can rewrite \cref{eq:Ngn over fibers} as
	\begin{equation}
		N'_{g,n}(L)
		=
		\sum_{(\widehat{G},\widehat{\ell}) \in \widehat{\mathcal N}_{g,n}(L)}
		w_{\widehat{G}}\,|\Att(\widehat{G},\widehat{\ell})|\,
		\frac{\rho_{\widehat{G}}(\widehat{\ell})}{|\Aut(\widehat{G},\widehat{\ell})|}.
	\end{equation}
	The rest of the proof is devoted to computing $|\Att(\widehat{G},\widehat{\ell})|$ explicitly in the four cases where $\widehat{G}$ is of type $\mc{R}$, $\mc{E}$, $\mc{D}^{\rm c}$, or $\mc{D}^{\rm d}$. This yields the corresponding contributions to $N'_{g,n}(L)$. Throughout the proof, we use $\widehat{G}$ to denote the trimmed graph $\tr(G-e_r)$.
	
	\smash{\fbox{Type $\mc{R}$}}
	In this situation, the root $r$ connects two different faces of $G$ which merge together in $\widehat{G}$. There are two subcases, depending on whether the rooted edge is part of an untwisted lollipop; see \cref{fig:typeR}.

	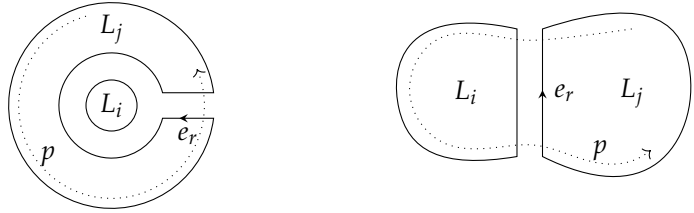
\begin{figure}[h]
		\begin{tikzpicture}[x=1pt,y=1pt,scale=.6]
		  \draw[postaction={
		      decorate,
		      decoration={markings, mark= at position 0.975 with {\arrow{stealth reversed}}}}](192, 696) arc[start angle=7.125, end angle=352.875, x radius=64.4981, y radius=-64.4981] -- (160, 712) arc[start angle=14.0362, end angle=345.9638, radius=32.9848] -- cycle;
		  \draw(128, 704) circle[radius=16];
		  \node at (128, 752) {$L_j$};
		  \node at (128, 704) {$L_i$};
		  \node at (176, 685) {$e_r$};
		  \draw[dotted, ->](112, 760) arc[start angle=105.9454, end angle=383.1986, radius=58.2409];
		  \node at (88, 672) {$p$};

		  \begin{scope}[xshift=5cm]
		    \draw(240, 672) -- (240, 752) .. controls (168, 768) and (160, 728) .. (166, 702) .. controls (172, 676) and (192, 664) .. (240, 672) -- cycle;
		    \draw[postaction={
		      decorate,
		      decoration={markings, mark= at position 0.125 with {\arrow{stealth}}}}](256, 672) -- (256, 752) .. controls (328, 776) and (344, 752) .. (348, 728) .. controls (352, 704) and (344, 680) .. (328, 668) .. controls (312, 656) and (288, 656) .. (256, 672) -- cycle;
		    \node at (208, 712) {$L_i$};
		    \node at (312, 712) {$L_j$};
		    \node at (270, 712) {$e_r$};
		    \draw[dotted, ->](312, 752) .. controls (256, 744) and (248, 744) .. (235.3333, 746) .. controls (222.6667, 748) and (205.3333, 752) .. (192.6667, 746.6667) .. controls (180, 741.3333) and (172, 726.6667) .. (172.6667, 712.6667) .. controls (173.3333, 698.6667) and (182.6667, 685.3333) .. (195.3333, 680) .. controls (208, 674.6667) and (224, 677.3333) .. (235.3333, 678.6667) .. controls (246.6667, 680) and (253.3333, 680) .. (267.6667, 675) .. controls (282, 670) and (304, 660) .. (324, 676);
		    \node at (292, 676) {$p$};
		  \end{scope}
		\end{tikzpicture}
		\caption{
			Edge removal of type $\mc{R}$: removing an untwisted lollipop (left); removing a rooted edge that connects two faces (right). The resulting new face of length $p$ is indicated by a dotted line.
		}
		\label{fig:typeR}
	\end{figure}

	We begin with the untwisted lollipop case. Assume without loss of generality that $L_i\le L_j$. Let $L_i$ be the perimeter of the face inside the candy and $L_j$ the perimeter of the exterior face (in $G$). Denote the stick by $e_s$ and the candy by $e_c$. After removal and trimming, the resulting graph $\widehat{G}$ has a new face of length $p \coloneqq L_j - L_i - 2\ell_{e_s}$; since $\ell_{e_s} \ge 0$, we have the constraint $p \le L_j-L_i$. There are $p$ ways to attach the lollipop to the new face of $\widehat{G}$, and the number of possible ciliations on the lollipop is $4\ell_{e_s}+4\ell_{e_c} = 2(L_i+L_j-p)$. Hence $|\Att(\widehat{G},\widehat{\ell})| = 2p(L_i+L_j-p)$, and this subcase contributes
	\begin{equation}
		\sum_{p=1}^{|L_i-L_j|} 2p(L_i+L_j-p) \, N_{g,n-1}\bigl(p,L_{[n]\setminus\{i,j\}}\bigr).
	\end{equation}

	We now turn to the case where the rooted edge $e_r$ connects two faces of lengths $L_i$ and $L_j$ but is not part of a lollipop. As before, the new face in $\widehat{G}$ has length $p \coloneqq (L_i-\ell_{e_r})+(L_j-\ell_{e_r})=L_i+L_j-2\ell_{e_r}$. Since $L_i-\ell_{e_r}>0$ and $L_j-\ell_{e_r}>0$, this forces $p>|L_i-L_j|$, and hence $\ell_{e_r}=(L_i+L_j-p)/2$. There are $p$ ways to attach the start of the rooted edge to the new face of $\widehat{G}$ so that the $i$-th face lies on its left; the other endpoint is then uniquely determined by the requirement that the two face lengths in $G$ are $L_i$ and $L_j$. The number of possible ciliations on the rooted edge is $4\ell_{e_r}=2(L_i+L_j-p)$. Thus again $|\Att(\widehat{G},\widehat{\ell})|=2p(L_i+L_j-p)$, and this subcase contributes
	\begin{equation}
		\sum_{p=|L_i-L_j|+1}^{L_i+L_j-1} 2p(L_i+L_j-p)\, N_{g,n-1}\bigl(p,L_{[n]\setminus\{i,j\}}\bigr).
	\end{equation}

	Putting the two subcases together, and noting that $w_{\widehat{G}} = 1$, we obtain the total $\mc{R}$-contribution
	\begin{equation}\label{eq:Rcontribution}
		\frac{1}{2}
		\sum_{\substack{i,j=1 \\ i\neq j}}^n
		\sum_{p=1}^{L_i+L_j-1}
			2p(L_i+L_j-p) \,
			N_{g,n-1}\bigl(p,L_{[n]\setminus\{i,j\}}\bigr),
	\end{equation}
	where the prefactor $\frac{1}{2}$ accounts for the fact that the outer sum runs over unordered pairs $(i,j)$, while the construction distinguishes the label $i$.

	\smash{\fbox{Type $\mc{E}$}}
	This is the non-orientable analogue of the previous case. There are two subcases, depending on whether the root $r$ is part of a twisted lollipop or not; see \cref{fig:typeE}.

	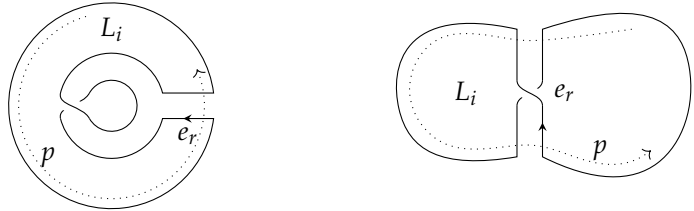
\begin{figure}[h]
		\begin{tikzpicture}[x=1pt,y=1pt,scale=.6]
		  \draw(96, 696) .. controls (94.1692, 704.0018) and (109.801, 703.9907) .. (114.1436, 712);
		  \draw[white, line width=2mm](96, 712) .. controls (94.1258, 703.976) and (109.881, 704.014) .. (114.144, 696);
		  \draw(114.1436, 696) arc[start angle=-150, end angle=150, radius=16];
		  \draw(96, 712) .. controls (94.1258, 703.9764) and (109.8814, 704.0143) .. (114.1436, 696);
		  \draw[postaction={
		      decorate,
		      decoration={markings, mark= at position 0.84 with {\arrow{stealth}}}}] (96, 712) arc[start angle=-165.9637, end angle=-14.0362, x radius=32.9848, y radius=-32.9848] -- (192, 712) arc[start angle=7.125, end angle=352.875, radius=64.4981] -- (160, 696) arc[start angle=14.0362, end angle=165.9638, x radius=32.9848, y radius=-32.9848];
		  \node at (128, 752) {$L_i$};
		  \node at (176, 685) {$e_r$};
		  \draw[dotted, ->](112, 760) arc[start angle=105.9454, end angle=383.1986, radius=58.2409];
		  \node at (88, 672) {$p$};

		  \begin{scope}[xshift=5cm]
		      \draw(240, 704) .. controls (240, 712) and (256, 712) .. (256, 720);
		      \draw[white, line width=2mm](256, 704) .. controls (256, 712) and (240, 712) .. (240, 720);
		      \draw(256, 704) .. controls (256, 712) and (240, 712) .. (240, 720);
		      \draw(240, 720) -- (240, 752) .. controls (168, 768) and (160, 728) .. (166, 702) .. controls (172, 676) and (192, 664) .. (240, 672) -- (240, 704);
		      \draw[postaction={
		      decorate,
		      decoration={markings, mark= at position 0.05 with {\arrow{stealth reversed}}}}](256, 704) -- (256, 672) .. controls (288, 656) and (312, 656) .. (328, 668) .. controls (344, 680) and (352, 704) .. (348, 728) .. controls (344, 752) and (328, 776) .. (256, 752) -- (256, 720);
		      \node at (208, 712) {$L_i$};
		      \node at (270, 712) {$e_r$};
		      \draw[dotted, ->](312, 752) .. controls (256, 744) and (248, 744) .. (235.3333, 746) .. controls (222.6667, 748) and (205.3333, 752) .. (192.6667, 746.6667) .. controls (180, 741.3333) and (172, 726.6667) .. (172.6667, 712.6667) .. controls (173.3333, 698.6667) and (182.6667, 685.3333) .. (195.3333, 680) .. controls (208, 674.6667) and (224, 677.3333) .. (235.3333, 678.6667) .. controls (246.6667, 680) and (253.3333, 680) .. (267.6667, 675) .. controls (282, 670) and (304, 660) .. (324, 676);
		      \node at (292, 676) {$p$};
		  \end{scope}
		\end{tikzpicture}
		\caption{
			Edge removal of type $\mc{E}$: removing an untwisted lollipop (left); removing a rooted edge that leaves the number of faces unchanged (right). The resulting new face of length $p$ is indicated by a dotted line.
		}
		\label{fig:typeE}
	\end{figure}

	Assume first that the root $r$ is part of a twisted lollipop attached inside a face of length $L_i$. Denote the stick by $e_s$ and the candy by $e_c$, and let $p$ be the length of the new face in $\widehat{G}$. There are $p$ ways to attach the lollipop to the new face. The number of possible ciliations is $4\ell_{e_s}+4\ell_{e_c} = 2(L_i-p)$. Moreover, unlike the untwisted lollipop case of type $\mc{R}$, the lengths of the candy and the stick may vary as long as their sum is fixed; this contributes a factor $\ell_{e_s}+\ell_{e_c}=(L_i-p)/2$. Thus
	\begin{equation}
		|\Att(\widehat{G},\widehat{\ell})|
		=
		p \, 2(L_i-p) \, \frac{L_i-p}{2}
		=
		p(L_i-p)^2.
	\end{equation}

	Next, assume that the root is not part of a twisted lollipop. Again let $p$ be the length of the new face in $\widehat{G}$. The rooted edge has two endpoints on this new face, and they can be chosen freely provided they do not coincide. Hence there are $\frac{p(p-1)}{2}$ ways to attach the rooted edge. (The excluded case of coinciding endpoints corresponds exactly to a twisted lollipop with vanishing stick, already accounted for above.) The number of possible ciliations is $4\ell_{e_r}=2(L_i-p)$. Therefore
	\begin{equation}
		|\Att(\widehat{G},\widehat{\ell})|
		=
		\frac{p(p-1)}{2} \, 2(L_i-p)
		=
		p(p-1)(L_i-p).
	\end{equation}

	Since $w_{\widehat{G}}=b$ throughout the $\mc{E}$ case, we obtain the contribution
	\begin{equation} \label{eq:Econtribution}
		b \sum_{i=1}^n \sum_{p=1}^{L_i-2}
		\underbrace{\bigl(p(L_i-p)^2+p(p-1)(L_i-p)\bigr)}_{=\,p(L_i-1)(L_i-p)}
		N_{g-\frac{1}{2},n}\bigl(p,L_{[n]\setminus\{i\}}\bigr).
	\end{equation}

	\smash{\fbox{Type $\mc{D}^{\rm c}$}}
	Consider now the connected $\mc{D}$ case. Removing the root $r$ from the ciliated graph $G$ yields a connected graph $\widehat{G}$, splitting a face of length $L_i$ into two new faces of lengths $p$ and $q$; see left panel in \cref{fig:typeD}.

	\begin{figure}[h]
		\begin{tikzpicture}[x=1pt,y=1pt,scale=.6]
		  \draw[dotted, ->](111.7823, 741.9087) .. controls (89.0433, 731.5821) and (84.5216, 707.791) .. (97.2608, 689.8955) .. controls (110, 672) and (140, 660) .. (186.4146, 691.8203);
		  \draw[dotted, ->](320, 720) .. controls (296, 752) and (270, 736) .. (266, 714) .. controls (262, 692) and (280, 664) .. (308, 688);
		  \draw[postaction={
		      decorate,
		      decoration={markings, mark= at position 0.0675 with {\arrow{stealth reversed}}}}] (264.7068, 731.5771) .. controls (237.5689, 755.859) and (209.2497, 755.1983) .. (179.7491, 729.5949) .. controls (175.74, 731.7579) and (171.7119, 733.8789) .. (168, 736) .. controls (149.3333, 746.6667) and (138.6667, 757.3333) .. (122.6667, 754.6667) .. controls (106.6667, 752) and (85.3333, 736) .. (82.6667, 714.6667) .. controls (80, 693.3333) and (96, 666.6667) .. (125.3333, 664) .. controls (154.6667, 661.3333) and (197.3333, 682.6667) .. (205.3333, 698.6667) .. controls (210.1008, 708.2016) and (202.5569, 715.8424) .. (191.7315, 722.7179) .. controls (213.2438, 742.2393) and (235.503, 741.3474) .. (258.5089, 720.0422) .. controls (256.8951, 715.6462) and (256, 711.0818) .. (256, 706.6667) .. controls (256, 688) and (272, 672) .. (285.3333, 669.3333) .. controls (298.6667, 666.6667) and (309.3333, 677.3333) .. (320, 685.3333) .. controls (330.6667, 693.3333) and (341.3333, 698.6667) .. (336, 712) .. controls (330.6667, 725.3333) and (309.3333, 746.6667) .. (290.6667, 746.6667) .. controls (280.5068, 746.6667) and (271.1369, 740.3469) .. (264.7068, 731.5771) -- cycle;
		  \node at (276, 700) {$q$};
		  \node at (112, 688) {$p$};
		  \node at (224, 760) {$e_r$};
		  \node at (80, 752) {$L_i$};

		  \begin{scope}[xshift=12cm]
		    \draw(216.007, 736.021) .. controls (223.1026, 737.4232) and (224.0387, 750.12) .. (231.9467, 748.6841);
		    \draw[white, line width=2mm](231.999, 735.639) .. controls (224.03, 737.498) and (223.257, 749.827) .. (216.046, 748.614);
		    \draw[dotted, ->](111.782, 741.909) .. controls (89.0433, 731.582) and (84.5216, 707.791) .. (97.2608, 689.8955) .. controls (110, 672) and (140, 660) .. (186.415, 691.82);
		    \draw[dotted, ->](320, 720) .. controls (296, 752) and (270, 736) .. (266, 714) .. controls (262, 692) and (280, 664) .. (308, 688);
		    \node at (276, 700) {$q$};
		    \node at (112, 688) {$p$};
		    \node at (224, 760) {$\tilde{e}_r$};
		    \node at (80, 752) {$L_i$};
		    \draw(231.999, 735.6387) .. controls (224.0298, 737.498) and (223.2573, 749.8273) .. (216.046, 748.6143);
		    \draw[postaction={
		      decorate,
		      decoration={markings, mark= at position 0.95 with {\arrow{stealth}}}}] (216.0068, 736.0208) .. controls (207.8122, 734.4016) and (199.7206, 729.9673) .. (191.732, 722.718) .. controls (202.557, 715.842) and (210.101, 708.202) .. (205.333, 698.667) .. controls (197.333, 682.667) and (154.667, 661.333) .. (125.333, 664) .. controls (96, 666.667) and (80, 693.333) .. (82.6667, 714.667) .. controls (85.3333, 736) and (106.667, 752) .. (122.667, 754.667) .. controls (138.667, 757.333) and (149.333, 746.667) .. (168, 736) .. controls (171.712, 733.879) and (175.74, 731.758) .. (179.749, 729.595) .. controls (192.0535, 740.274) and (204.1525, 746.6138) .. (216.046, 748.6143);
		    \draw(231.903, 735.6618) .. controls (240.6578, 733.5671) and (249.5265, 728.3605) .. (258.509, 720.042) .. controls (256.895, 715.646) and (256, 711.082) .. (256, 706.667) .. controls (256, 688) and (272, 672) .. (285.333, 669.333) .. controls (298.667, 666.667) and (309.333, 677.333) .. (320, 685.333) .. controls (330.667, 693.333) and (341.333, 698.667) .. (336, 712) .. controls (330.667, 725.333) and (309.333, 746.667) .. (290.667, 746.667) .. controls (280.507, 746.667) and (271.137, 740.347) .. (264.707, 731.577) -- (264.707, 731.577) .. controls (253.9717, 741.1825) and (243.0516, 746.8848) .. (231.9467, 748.6841);
		  \end{scope}
		\end{tikzpicture}
		\caption{
			Left: edge removal of type $\mc{D}^{\rm c}$ or $\mc{D}^{\rm d}$: removing the rooted edge $e_r$ creates two new faces of lengths $p$ and $q$, indicated by a dotted line. The graph may or may not be disconnected in the process. Right: construction of the partner ciliated graph $(\tilde{G},\tilde{c})$ from $(G,c)$.
		}
		\label{fig:typeD}
	\end{figure}
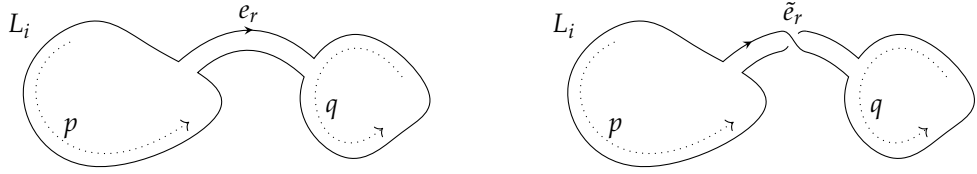

	In this situation, the set $\Att(\widehat{G},\widehat{\ell})$ can be partitioned into pairs: in each pair, one element has an untwisted rooted edge $e_r$, and the other is identical except that $e_r$ is twisted. More precisely, given a ciliated graph $(G,c)$ with root $r$, we define its partner $(\tilde{G},\tilde{c})$ with root $\tilde{r}$ as follows. Let $\tilde{G}$ be obtained from $G$ by reversing the $\Z_2$-colouring of the rooted edge $e_r$; denote the corresponding edge in $\tilde{G}$ by $\tilde{e}_r$. Let $h'$ be the half-edge of $G$ that immediately precedes the half-edge containing the root $r$ (with respect to the ordering from \cref{ssub:rooting}), and let $f'$ be the face of $\widehat{G}$ containing $h'$. We define $\tilde{r}$ to be the unique rooting on $\tilde{e}_r$ such that, for the ordering induced by $\tilde{r}$, (i) $h'$ is the half-edge of $\tilde{G}$ immediately preceding the half-edge containing $\tilde{r}$, and (ii) the face $f'$ has the same orientation as in $\tr(\tilde{G}-\tilde{e}_r)$. Finally, $\tilde{c}$ is defined by placing the cilium on the same unit segment (counted from $\tilde{r}$) as in $c$. See right panel in \cref{fig:typeD}. By construction, $(G,c)$ and $(\tilde{G},\tilde{c})$ are distinct as ciliated graphs. Indeed, after removing $r$ and $\tilde{r}$, the induced orientations of the face preceding the root agree, while the induced orientations of the face succeeding the root disagree, since the rooted edges in $G$ and $\tilde{G}$ have opposite colour. In particular, the corresponding MON weights satisfy $w_r + w_{\tilde{r}}=1+b$ (see the $\mc{D}^{\rm c}$ rule in \cref{ssub:mon:ciliated}).

	We now compute $|\Att(\widehat{G},\widehat{\ell})|$ in the $\mc{D}^{\rm c}$ case. The number of ways to attach the edge $e_r$ to $\widehat{G}$ is $pq$. The number of ways to ciliate this edge is $4\ell_{e_r}=2(L_i-p-q)$. Grouping attachments into partner pairs and using $w_r+w_{\tilde{r}}=1+b$, we obtain the total contribution
	\begin{equation}\label{eq:Dccontribution}
		\frac{1+b}{2} \sum_{i=1}^n \sum_{\substack{p,q>0 \\ p+q<L_i}}
			2pq(L_i-p-q)\,
			N_{g-1,n+1}\bigl(p,q,L_{[n]\setminus \{i\}}\bigr),
	\end{equation}
	where the prefactor $\frac{1}{2}$ accounts for the symmetry under exchanging $p$ and $q$.

	\smash{\fbox{Type $\mc{D}^{\rm d}$}}
	Finally, consider the disconnected $\mc{D}$ case. Removing the root $r$ from the ciliated graph $G$ yields a disconnected graph $\widehat{G}$, splitting a face of length $L_i$ into two new faces of lengths $p$ and $q$; see \cref{fig:typeD} again.

	Like in the connected $\mc{D}^{\rm c}$ case, the set $\Att(\widehat{G},\widehat{\ell})$ splits into partner pairs: for any attachment, one may flip the colour of the new edge, producing a distinct lift to a ciliated M\"obius graph while leaving the graph $\widehat{G}$ unchanged. In the disconnected setting both lifts have weight~$1$, so each pair contributes a factor $2=1+1$. As for $|\Att(\widehat{G},\widehat{\ell})|$: there are $pq$ choices for attaching the edge $e_r$ to $\widehat{G}$, and $4\ell_{e_r}=2(L_i-p-q)$ choices for ciliation. Grouping attachments into partner pairs therefore yields the total contribution.
	\begin{equation}\label{eq:Ddcontribution}
		\frac{1+1}{2}
		\sum_{i=1}^n \sum_{\substack{p,q>0\\ p+q <L_i}} 2pq(L_i-p-q)
		\sum_{\substack{g_1+g_2 = g \\ I_1 \sqcup I_2 = [n]\setminus\{i\}}}
			N_{g_1,1+|I_1|}(p,L_{I_1})\,
			N_{g_2,1+|I_2|}(q,L_{I_2}),
	\end{equation}
	where again the prefactor $\frac{1}{2}$ accounts for the symmetry under exchanging $p$ and $q$.

	\noindent\textit{Conclusion of the proof.}
	Putting together the four contributions from \cref{eq:Rcontribution,eq:Econtribution,eq:Dccontribution,eq:Ddcontribution}, and using the relation $N'_{g,n}(L) = 2( \sum_i L_i ) N_{g,n}(L)$ from \cref{lem:cil:to:uncil} yields the statement of the theorem.
\end{proof}

We note that together with the initial conditions $N_{0,3}$, $N_{\frac{1}{2},2}$ and $N_{1,1}$ computed in \cref{app:base:top}, the above theorem gives an explicit recursive computation of the $N_{g,n}$ for any $(g,n)$ such that $2g-2+n > 1$.

\subsubsection{The asymmetric recursion}\label{sec:the asymmetric recursion}
Finally, we are ready to prove the following recursion for $N_{g,n}$, stated in \cref{thm:lattice:rec}:
\begin{equation}\label{eq:rec:toprove}
\begin{split}
	N_{g,n}(L_{[n]})
	&=
	\sum_{m=2}^n \sum_{p > 0}
		p \, \mc{R}(L_1,L_m,p)
		N_{g,n-1}(p,L_2,\ldots,\widehat{L_m},\ldots,L_n) \\
	&\qquad
	+
	b \sum_{p > 0}
		p(L_1 - 1) \, \mc{E}(L_1,p)
		N_{g-\frac{1}{2},n}(p,L_2,\ldots,L_n) \\
	&\qquad\qquad
	+
	\sum_{p,q > 0}
		p q \, \mc{D}(L_1,p,q)
		\Bigg(
			\frac{1+b}{2} N_{g-1,n+1}(p,q,L_2,\ldots,L_n) \\
	&\qquad\qquad\qquad 
			+
			\sum_{\substack{g_1+g_2 = g \\ I_1 \sqcup I_2 = \{2,\ldots,n\}}}
				N_{g_1,1+|I_1|}(p,L_{I_1})
				N_{g_2,1+|I_2|}(q,L_{I_2})
		\Bigg) .
\end{split}
\end{equation}
Here $\mc{R}$, $\mc{E}$, and $\mc{D}$ are given explicitly by
\begin{equation}
\begin{aligned}
	\mc{R}(L_1,L_m,p) &= \frac{1}{2L_1} \Bigl(
		[L_1 + L_m - p]_+
		-
		[-L_1 + L_m - p]_+
		+
		[L_1 - L_m - p]_+ 
	\Bigr), \\
	\mc{E}(L_1,p) &= \frac{1}{2L_1} [L_1 - p]_+ , \\
	\mc{D}(L_1,p,q) &= \frac{1}{L_1} [L_1 - p - q]_+ .
\end{aligned}
\end{equation}
Note that \labelcref{eq:rec:toprove} singles out the first boundary component, and is therefore asymmetric in the $L_i$, in contrast with the symmetric recursion of \cref{thm:lattice:rec:sym}. This asymmetry mirrors the distinguished role played by the variable $z_1$ in the refined topological recursion formula discussed in \cref{sec:RTR}.

\begin{proof}[Proof of \cref{thm:lattice:rec}]
	To prove \cref{thm:lattice:rec}, we first note that both the symmetric recursion \labelcref{eq:lattice:rec:sym} and the asymmetric one \labelcref{eq:rec:toprove} uniquely determine the refined lattice point count $N_{g,n}$ for $2g-2+n>0$. Since the two recursions share the same initial conditions, it suffices to show that \labelcref{eq:rec:toprove} implies \labelcref{eq:lattice:rec:sym}.

	Let $\ms{N}_{g,n}$ denote the count produced by the asymmetric recursion. A priori, it is not clear that $\ms{N}_{g,n}$ is symmetric in the boundary lengths $L_1,\ldots,L_n$. However, we prove in \cref{ssec:Weber} that \labelcref{eq:rec:toprove} is the discrete Laplace-transformed form of refined topological recursion on the Weber curve. Since refined topological recursion produces symmetric differentials, and since the $\ms{N}_{g,n}$ are the expansion coefficients of these differentials at $z_i=0$, it follows that $\ms{N}_{g,n}$ are symmetric in $L_1,\ldots,L_n$. By relabelling the boundary components, we may therefore write a version of \cref{eq:rec:toprove} in which any $L_i$ (for $i=1,\ldots,n$) plays the distinguished role instead of $L_1$. Multiplying the corresponding equation by $L_i$ and summing over $i=1,\ldots,n$, we obtain
	\begin{equation} \label{eq:asymmetric_sum}
	\begin{split}
		\biggl(\sum_{i=1}^n L_i\biggr)
		\ms{N}_{g,n}\big(L_{[n]}\big)
		&=
		\sum_{\substack{i,j = 1 \\ i \neq j}}^n \sum_{p >0}
			p \frac{[L_i+L_j-p]_+}{2}\,
			\ms{N}_{g,n-1}\big(p,L_{[n]\setminus \{i,j\}} \big) \\
		&\qquad
		+
		b \sum_{i=1}^n \sum_{p >0}
			p (L_i-1) \frac{[L_i-p]_+}{2}\,
			\ms{N}_{g-\frac{1}{2},n}\big(p,L_{[n]\setminus\{i\}}\big) \\
		&\qquad\qquad
		+
		\sum_{i=1}^n \sum_{ p,\,q >0 }
			p q [L_i-p-q]_+\,
			\Bigg(
				\frac{1+b}{2}
				\ms{N}_{g-1,n+1}\big(p,q,L_{[n]\setminus \{i\}}\big) \\
		&\qquad\qquad\qquad\qquad 
				+
				\sum_{\substack{g_1+g_2 = g \\ I_1 \sqcup I_2 = [n]\setminus\{i\}}}
					\ms{N}_{g_1,1+|I_1|}(p,L_{I_1})\,
					\ms{N}_{g_2,1+|I_2|}(q,L_{I_2})
			\Bigg)
			+ \Delta_{g,n}(L_{[n]}).
	\end{split}
	\end{equation}
	Here
	\begin{equation}
		\Delta_{g,n}(L_{[n]})
		=
		\frac{1}{2} \sum_{\substack{i,j = 1 \\ i \neq j}}^n \sum_{p>0}
		\bigl( [L_i - L_j - p]_+ - [L_j - L_i - p]_+ \bigr)\,
		\ms{N}_{g,n-1}\big(p,L_{[n]\setminus \{i,j\}}\big).
	\end{equation}
	Apart from $\Delta_{g,n}$, \cref{eq:asymmetric_sum} matches the symmetric recursion \labelcref{eq:lattice:rec:sym}. Moreover, $\Delta_{g,n} = 0$, since the summand is odd under the exchange of $L_i$ and $L_j$. It follows that $\ms{N}_{g,n}$ satisfies the symmetric recursion. By uniqueness of solutions to the recursion with the given initial conditions, we conclude that $\ms{N}_{g,n}=N_{g,n}$, which completes the proof.
\end{proof}

\section{The volume recursion}
\label{sec:volume}

The moduli space of metric M\"obius graphs not only possesses an integral structure but is also equipped with a natural measure, which we refer to as the Euclidean measure. These two structures are closely related: by rescaling the lattice to a finer and finer mesh, the integral structure tends to the Euclidean measure in the limit. This leads to a continuous analogue of the refined lattice point count, the \emph{refined volumes}. They satisfy an integral (rather than discrete) recursion, stated in \cref{thm:volume:rec}.

\subsection{Euclidean measure and refined volumes}
Let $G$ be a M\"obius graph. On the cell $P_G(L)$ we define $d\mu_G(\ell)$ to be the unique translation-invariant measure characterised by
\begin{equation}
	d\mu_G(\ell) \prod_{i=1}^n dL_i
	=
	\prod_{e \in E(G)} d\ell_e.
\end{equation}
Equivalently, $d\mu_G$ is the fibre measure for which the linear change of variables $\ell \mapsto (L,\textup{fibre coordinates})$ has unit Jacobian. This in turn defines a measure $d\mu$ on $\Mod_{g,n}(L)$ by setting\footnote{
	 $\Mod_{g,n}(L)$ need not be orientable, but this is not an issue. We work with a measure (a positive density), not a volume form. The orbicell decomposition by polytopes $P_G(L)/\Aut(G)$ provides a canonical translation-invariant measure on each top-dimensional cell, and lower-dimensional cells have $d\mu$-measure zero. The factor $|\Aut(G)|$ implements orbifold weighting.
}
\begin{equation}\label{eq:Eucl:measure}
	\int_{\Mod_{g,n}(L)} f(\bm{G}) \, d\mu(\bm{G})
	\coloneqq
	2^{2g-2+n}
	\sum_{\substack{G \in \MG_{g,n} \\ \textup{trivalent}}} \frac{1}{|\Aut(G)|} \int_{P_{G}(L)} f(\ell) \, d\mu_G(\ell).
\end{equation}
for any (say) bounded continuous function $f$ on $\Mod_{g,n}(L)$. The sum runs only over trivalent graphs, since graphs with higher-valent vertices contribute measure zero. On the right-hand side, $f$ denotes, by abuse of notation, the unique $\Aut(G)$-equivariant lift of $f$ to $P_G(L)$. The factor $2^{2g-2+n}$ is included by convention to match the unrefined setting.

The integral structure introduced in \cref{sec:defs} is closely related to the measure $d\mu$. More precisely, for a bounded continuous function $f$ on $\Mod_{g,n}(L)$, we have
\begin{equation}\label{eq:lattice:to:Eucl}
	\sum_{\bm{G} \in \ZMod_{g,n}(\lambda^{-1}L)}
			\frac{f(\bm{G})}{|\Aut(\bm{G})|}
	\sim
	\frac{1}{\lambda^{6g-6+2n}}
	\frac{2}{2^{2g-2+n}}
	\int_{\Mod_{g,n}(L)} f(\bm{G}) \, d\mu(\bm{G})
\end{equation}
as $\lambda \to 0^+$, with $\lambda^{-1} \sum_i L_i \in 2\Z_{>0}$. In other words, the weighted count of integer points in the rescaled lattice is asymptotic to the corresponding integral against $d\mu$. The factor $2^{-(2g-2+n)}$ comes from the normalisation in \cref{eq:Eucl:measure}, while the additional factor of $2$ reflects the parity constraint: lattice points occur only when $\sum_i L_i$ is even. More concretely, for each fixed trivalent $G$,
\begin{equation}
	\sum_{\ell \in P_G^{\Z}(\lambda^{-1}L)} f(\ell)
	\sim
	\frac{1}{\lambda^{\dim P_G(L)}}
	\frac{1}{\mathrm{covol}(\Lambda_G)}
	\int_{P_G(L)} f(\ell) \, d\mu_G(\ell),
\end{equation}
where $\Lambda_G = \set{ \ell\in\Z^{E(G)} | A_G\ell=0 }$ is the kernel lattice and $\mathrm{covol}(\Lambda_G)$ denotes its covolume with respect to $d\mu_G$. In our situation, the image lattice $A_G(\Z^{E(G)})\subset \Z^n$ has index $2$: it consists precisely of those boundary length vectors whose coordinate sum is even. With the normalisation of $d\mu_G$ above, this implies $\mathrm{covol}(\Lambda_G)=\tfrac12$.

Motivated by \cref{eq:lattice:to:Eucl}, we define the refined volumes as the integrals of the MON with respect to $d\mu$. Notice that the MON is continuous and bounded thanks to \cref{prop:MON}.

\begin{definition}
	For $L=(L_1,\ldots,L_n)\in \R_{>0}^n$, define the \emph{refined volumes} by
	\begin{equation}
		V_{g,n}(L;b)
		\coloneqq
		\int_{\Mod_{g,n}(L)} \rho(\bm{G};b) \, d\mu(\bm{G}) .
	\end{equation} The refined volume $V_{g,n}(L;b)$ is a polynomial of degree at most $2g$ in $b$. For notational simplicity, we omit the dependence on $b$ henceforth.
\end{definition}

From \cref{eq:lattice:to:Eucl} it follows that, when $\sum_i L_i$ is even,
\begin{equation}
	N_{g,n}(\lambda^{-1}L)
	\sim
	\frac{1}{\lambda^{6g-6+2n}}
	\frac{2}{2^{2g-2+n}} \,
	V_{g,n}(L)
	\qquad
	\text{as $\lambda \to 0^+$.}
\end{equation}
This relation, together with the recursion for the refined lattice point count, implies a recursion for the refined volumes, analogous to the one satisfied by the Witten--Kontsevich volumes \cite{ABCGLW26} (see also \cite{BCEG25} for a derivation \`{a}~la Tutte in the context of $r$-spin intersection numbers, which corresponds to the Witten--Kontsevich result when $r=2$). It is also worth noting that, in the orientable setting (corresponding to $b=0$), the same Euclidean measure and the associated volumes can be used to access finer geometric and dynamical features of moduli spaces, including counts of multicurves \cite{ABCGLW26,BCDGW22}, statistics of the length spectrum \cite{JL23,BGL25}, and applications to hyperbolic and flat geometry \cite{AC25,Tal25,ABCDGLW23,DGZZ21}.

\subsection{The volume recursion}
In this section, we prove \cref{thm:volume:rec} by induction on $2g-2+n$.

The base cases follow from a direct calculation, either by extracting the leading term from \cref{eq:base:top:lattice} or by a direct computation as in \cref{app:base:top}. For the recursion step, the key observation is that the sums over $p$ and $q$ in the discrete recursion \labelcref{eq:lattice:rec} can be viewed as Riemann sums which, as $\lambda\to 0^+$, converge to the corresponding Riemann integrals. More precisely, rescale the boundary lengths by $L\mapsto \lambda^{-1}L$ and multiply both sides by $\lambda^{6g-6+2n}\,2^{2g-3+n}$. The left-hand side tends to $V_{g,n}(L)$. On the right-hand side, after the change of summation variables $p\mapsto \lambda^{-1}p$ and $q\mapsto \lambda^{-1}q$, and a straightforward bookkeeping of the powers of $\lambda$ and $2$, we obtain, setting $I=\set{2,\ldots,n}$,
\begin{multline}
	(2\lambda)
	\sum_{m=2}^n \sum_{p\in \lambda \Z_{>0}}
		p \, \mc{R}(L_1,L_m,p)\,
		N_{g,n-1}^{\lambda}(p,L_{I\setminus\{m\}})
	+
	b(2\lambda) \sum_{p \in \lambda \Z_{>0}}
		p (L_1-\lambda)\, \mc{E}(L_1,p)\,
		N_{g-\frac{1}{2},n}^{\lambda}(p,L_I)\\
	+
	\frac{1+b}{2}\,
	(2\lambda^{2})
	\sum_{p,q \in \lambda \Z_{>0}}
		pq\, \mc{D}(L_1,p,q)\,
		N_{g-1,n+1}^{\lambda}(p,q,L_I) 
		\\
	+
	(4\lambda^{2})
	\sum_{p,q \in \lambda \Z_{>0}}
	\sum_{\substack{g_1+g_2 = g \\ I_1 \sqcup I_2 = I}}
		pq\, \mc{D}(L_1,p,q)\,
		N_{g_1,1+|I_1|}^{\lambda}(p,L_{I_1})\,
		N_{g_2,1+|I_2|}^{\lambda}(q,L_{I_2}).
\end{multline}
Here we have set $N^{\lambda}_{g_0,n_0}(L) \coloneqq \lambda^{6g_0-6+2n_0}\,2^{2g_0-3+n_0}\, N_{g_0,n_0}(\lambda^{-1}L)$ for the rescaled lattice point count. Since the kernels $\mc{R}$, $\mc{E}$, and $\mc{D}$ are homogeneous of degree zero, they are unchanged by the rescaling. Note also that the sums over $p$ and $q$ now run over the rescaled lattice $\lambda\Z_{>0}$, and that the four terms carry different prefactors in $\lambda$ and~$2$.

Before proceeding, we record a parity constraint that is crucial for the Riemann-sum-to-Riemann-integral limit. The sums over $p$ and $q$ are not taken over all of $\lambda\Z_{>0}$, but only over those values compatible with the fact that the lattice point count vanishes unless the sum of the corresponding boundary lengths is even. More precisely, assuming throughout that $\lambda^{-1}\sum_i L_i \in 2\Z_{>0}$, we obtain:
\begin{itemize}
	\item
	In the $\mc{R}$-sum, $\lambda^{-1}(p + L_2 + \cdots + \widehat{L_m} + \cdots + L_n)$ must be even. Equivalently, $2\lambda \mid p - L_1 - L_m$.

	\item
	In the $\mc{E}$-sum, $\lambda^{-1}(p + L_2 + \cdots + L_n)$ must be even. Equivalently, $2\lambda \mid p - L_1$.

	\item
	In the connected $\mc{D}$-sum, $\lambda^{-1}(p+q + L_2 + \cdots + L_n)$ must be even. Equivalently, $2\lambda \mid p+q - L_1$.

	\item
	In the disconnected $\mc{D}$-sum, for any splitting $I_1 \sqcup I_2 = \{2,\ldots,n\}$, both $\lambda^{-1}(p + \sum_{i \in I_1} L_i)$ and $\lambda^{-1}(q + \sum_{j \in I_2} L_j)$ must be even. We write these conditions as $2\lambda \mid p + L_{I_1}$ and $2\lambda \mid q + L_{I_2}$.
\end{itemize}
These restrictions will be imposed on the sums below. Passing to the limit $\lambda\to 0^+$, and using the inductive asymptotic relation $N^{\lambda}_{g_0,n_0}\sim V_{g_0,n_0}$, one finds:

\fbox{Type $\mc{R}$, $\mc{E}$ and $\mc{D}^{\rm c}$}
For the $\mc{R}$-, $\mc{E}$-, and connected $\mc{D}$-terms there is a single divisibility condition by $2\lambda$, so the sums run over half of the rescaled lattice. Consequently, the corresponding Riemann sums converge to the full integrals:
\begin{equation}
\begin{aligned}
	&(2\lambda)
	\sum_{\substack{p \in \lambda\Z_{>0} \\ 2\lambda \mid p - L_1 - L_m}}
			p \, \mc{R}(L_1,L_m,p)\,
			N_{g,n-1}^{\lambda}(p,L_{I \setminus \{m\}})
	\sim
	\int_{0}^{+\infty}
			p \, \mc{R}(L_1,L_m,p)\,
			V_{g,n-1}(p,L_{I \setminus \{m\}})\, dp,
	\\
	&(2\lambda)
	\sum_{\substack{p \in \lambda\Z_{>0} \\ 2\lambda \mid p - L_1}}
		p (L_1 - \lambda)\, \mc{E}(L_1,p)\,
		N_{g-\frac{1}{2},n}^{\lambda}(p,L_I)
	\sim
	\int_{0}^{+\infty}
		p L_1\, \mc{E}(L_1,p)\,
		V_{g-\frac{1}{2},n}(p,L_I)\, dp,
	\\
	&(2\lambda^2)
	\sum_{\substack{p ,q \in \lambda\Z_{>0} \\ 2\lambda \mid p + q - L_1}}
		pq\, \mc{D}(L_1,p,q)\,
		N_{g-1,n+1}^{\lambda}(p,q,L_I)
	\sim
	\int_{0}^{+\infty}\!\!\int_{0}^{+\infty}
		pq\, \mc{D}(L_1,p,q)\,
		V_{g-1,n+1}(p,q,L_I)\, dp dq.
\end{aligned}
\end{equation}

\fbox{Type $\mc{D}^{\rm d}$}
For the disconnected $\mc{D}$-term there are two independent divisibility conditions by $2\lambda$, one for $p$ and one for $q$, so each sum runs over half of the rescaled lattice (hence an overall factor of $1/4$ appears). Accordingly,
\begin{multline}
	(4\lambda^2)
	\sum_{\substack{p,q \in \lambda\Z_{>0} \\ 2\lambda \mid p + L_{I_1},\;\; 2\lambda \mid q + L_{I_2}}}
			pq\, \mc{D}(L_1,p,q)\,
			N_{g_{1},1+|I_1|}^{\lambda}(p,L_{I_1})\,
			N_{g_{2},1+|I_2|}^{\lambda}(q,L_{I_2}) \\
	\sim
	\int_{0}^{+\infty} \int_{0}^{+\infty}
			pq\, \mc{D}(L_1,p,q)\,
			V_{g_{1},1+|I_1|}(p,L_{I_1})\,
			V_{g_{2},1+|I_2|}(q,L_{I_2})\, dp dq.
\end{multline}
Putting these limits together yields the recursion formula stated in \cref{thm:volume:rec}.

\section{Laplace transform and refined topological recursion} 
\label{sec:RTR}

The goal of this section is to relate refined topological recursion on the Weber and Airy curves to the refined lattice point count and refined volumes, respectively. After recalling the definition of refined topological recursion, we prove the following.

\begin{proposition}[Refined topological recursion counts M\"obius graphs]
	\label{prop:RTR}
	\leavevmode
	\begin{enumerate}
		\item\label{Weber}
		Consider the refined Weber curve with $\mu=-1$. Under the identification of refinement parameters \smash{$\mf{b} = - \frac{b}{\sqrt{1+b}}$}, the associated correlation differentials encode the refined lattice point count via a discrete Laplace transform:
		\begin{equation}\label{eq:Web:lattice}
			\omega_{g,n}^{\textup{Web}}(z_1,\ldots,z_n)
			=
			(-1)^{n} \frac{2}{(1+b)^g}
			\sum_{L_1,\ldots,L_n > 0}
			N_{g,n}(L_1,\ldots,L_n)
			\prod_{i=1}^n L_i \, z_i^{L_i - 1} \, dz_i.
		\end{equation}

		\item\label{Airy}
		Consider the refined Airy curve. Under the same identification \smash{$\mf{b} = - \frac{b}{\sqrt{1+b}}$}, the associated correlation differentials encode the refined volumes via a (continuous) Laplace transform:
		\begin{equation}
			\omega_{g,n}^{\textup{Airy}}(z_1,\ldots,z_n)
			=
			\frac{2}{(1+b)^g}
			\int_{0}^{+\infty} \cdots \int_{0}^{+\infty}
			V_{g,n}(L_1,\ldots,L_n)
			\prod_{i=1}^n L_i \, e^{-z_i L_i} \, dL_i \, dz_i.
		\end{equation}
	\end{enumerate}
\end{proposition}

The normalisation factor $(1+b)^g$ also appears in the relation between refined topological recursion and $b$-Hurwitz numbers proved in \cite{CDO26}. We note, however, that the sign $(-1)^{n}$ is required in the Weber case, whereas no such sign appears for the Airy curve. The factor of $2$ is purely conventional and reflects our choice of normalisation (in particular, the convention for working with an ``unoriented'' rather than ``oriented'' counting).

\subsection{Refined topological recursion}
The definition of refined spectral curves and refined topological recursion is given in \cite{KO23,Osu24a}. While the original proposal of Chekhov--Eynard \cite{CE06} arose from $\beta$-deformed matrix models, \cite{KO23,Osu24a} both resolves a number of foundational subtleties and formulates refined topological recursion intrinsically at the level of refined spectral curves---much as the Eynard--Orantin construction \cite{EO07} recasts topological recursion as a theory independent of its random-matrix origins.

For the cases treated here, namely the Weber and Airy curves, we present a simplified version that can be directly deduced from loc.~cit. In both cases, the underlying Riemann surface is $\P^1$. Recall that on the projective line there is a unique fundamental bidifferential (also known as the Bergmann kernel), given in any global coordinate by
\begin{equation}
	B(z_1,z_2) \coloneqq \frac{dz_1\,dz_2}{(z_1 - z_2)^2}.
\end{equation}

\begin{definition}\label{def:Weber:Airy}
	The \emph{refined Weber spectral curve} is defined by the data $\mc{S}^{\mathrm{Web}} = (x,\omega_{0,1},\omega_{0,2},\omega_{\frac12,1})$, where:
	\begin{itemize}
		\item
		$x$ is the meromorphic function
		\begin{equation}
			x(z) \coloneqq z + \frac{1}{z},
		\end{equation}
		which is invariant under the involution $\sigma\colon z \mapsto \frac{1}{z}$. The ramification locus of $x\colon \P^1 \to \P^1$ is $\mc{R}=\{+1,-1\}$.

		\item
		$\omega_{0,1}$, $\omega_{0,2}$, and $\omega_{\frac12,1}$ are the meromorphic (bi)differentials
		\begin{equation}
			\omega_{0,1}(z_1) \coloneqq \frac{(1-z_1^2)^2}{2 z_1^3}\,dz_1,
			\quad
			\omega_{0,2}(z_1,z_2) \coloneqq - B(z_1,\sigma(z_2)),
			\quad
			\omega_{\frac12,1}(z_1) \coloneqq \mf{b} \left( \frac{z_1}{1 - z_1^2} + \frac{1+\mu}{2z_1} \right)dz_1.
		\end{equation}
	\end{itemize}
	The \emph{refined Airy spectral curve} is the collection of data $\mc{S}^{\mathrm{Airy}} = (x,\omega_{0,1},\omega_{0,2},\omega_{\frac12,1})$, where:
	\begin{itemize}
		\item
		$x$ is the meromorphic function
		\begin{equation}
			x(z) \coloneqq \frac{z^2}{2},
		\end{equation}
		which is invariant under the involution $\sigma\colon z \mapsto -z$. The ramification locus of $x\colon \P^1 \to \P^1$ is $\mc{R}=\{0,\infty\}$.

		\item
		$\omega_{0,1}$, $\omega_{0,2}$, and $\omega_{\frac12,1}$ are the meromorphic (bi)differentials
		\begin{equation}
			\omega_{0,1}(z_1) \coloneqq -z_1^2\,dz_1,
			\qquad
			\omega_{0,2}(z_1,z_2) \coloneqq - B(z_1,\sigma(z_2)),
			\qquad
			\omega_{\frac12,1}(z_1) \coloneqq -\mf{b}\,\frac{dz_1}{2z_1}.
		\end{equation}
	\end{itemize}
\end{definition}

Both refined spectral curves depend on the parameter $\mf{b}\in\C$ (and, in the Weber case, also on $\mu\in\C$). For simplicity, we suppress this dependence in the notation.

\begin{definition}
	Let $\mc{S}$ be the refined Airy or Weber spectral curve. The \emph{refined topological recursion} produces a family of meromorphic multidifferentials $\omega_{g,n}$ on $\P^1$, for $g \in \frac12\Z_{\ge 0}$ and $n \in \Z_{>0}$ with $2g-2+n>0$, defined by
	\begin{equation}
		\omega_{g,n}(z_1,\ldots,z_n)
		\coloneqq
		\left(
			\sum_{i=1}^n \biggl( \Res_{z=z_i} - \Res_{z=\sigma(z_i)} \biggr)
			-
			\sum_{r \in \mc{R}} \Res_{z=r}
		\right)
		K(z_1,z)\,
		\Rec_{g,n}(z;z_2,\ldots,z_n),
	\end{equation}
	where the recursion kernel is $K(z_1,z) \coloneqq \frac{\int_{\sigma(z)}^{z} \omega_{0,2}(z_1,\cdot)}{4\,\omega_{0,1}(z)}$, and the recursion input is
	\begin{equation}
	\begin{split}
		\Rec_{g,n}(z;z_2,\ldots,z_n)
		&\coloneqq
		\sum_{m=2}^n
			\frac{dx(z)\,dx(z_m)}{(x(z) - x(z_m))^2} 
			\omega_{g,n-1}(z,z_2,\ldots,\widehat{z_m},\ldots,z_n) \,
		+
		\mf{b} \, dx(z) \, d_z \frac{\omega_{g-\frac{1}{2},n}(z,z_2,\ldots,z_n)}{dx(z)}
		\\
		&\qquad +
		\omega_{g-1,n+1}(z,z,z_2,\ldots,z_n)
		+
		\sum_{\substack{g_1+g_2=g \\ I_1 \sqcup I_2 = \{2,\ldots,n\}}}'
			\omega_{g_1,1+|I_1|}(z,z_{I_1})
			\omega_{g_2,1+|I_2|}(z,z_{I_2})
		.
	\end{split}
	\end{equation}
	Here the primed sum runs over splittings with $2g_i-2+(1+|I_i|) \ge 0$, i.e.~it excludes the unstable term $(g_i,1+|I_i|)=(0,1)$, while allowing $(g_i,1+|I_i|)=(\tfrac12,1)$ and $(g_i,1+|I_i|)=(0,2)$. Moreover, $d_z$ denotes the exterior derivative with respect to the variable $z$.
\end{definition}

The first few correlators, corresponding to $2g-2+n=1$, are given by
\begin{equation}\label{eq:omega:Weber:base}
\begin{aligned}
	&\omega_{0,3}^{\textup{Web}}(z_1,z_2,z_3)
	=
	-d_1 d_2 d_3 \left(
		\frac{z_1 z_2 z_3\,(z_1+z_2+z_3+z_1 z_2 z_3)}{(1-z_1^2)(1-z_2^2)(1-z_3^2)}
	\right)
	, \\
	&\omega_{\frac12,2}^{\textup{Web}}(z_1,z_2)
	=
	-\mf{b} \, d_1 d_2 \left(
		\frac{(z_1+z_2)^2 (1-z_1z_2)^2 - z_1z_2 (1-z_1^2) (1-z_2^2)}{(1-z_1^2)^2 (1-z_2^2)^2 (1-z_1z_2)}
		+
		\frac{z_1 z_2 (\mu z_1 z_2 + \mu - 1)}{2 (1-z_1^2)(1-z_2^2)}
	\right)
	, \\
	&\omega_{1,1}^{\textup{Web}}(z_1)
	=
	-d_1\left(
		\frac{
			2 (3z_1^2-1)
			+
			\mf{b}^2 ( 3(\mu^2 -4\mu +3) z_1^4 -6(\mu^2 -2\mu -2)z_1^2 +3\mu^2 -1 )
		}{24(1-z_1^2)^3}
	\right)
	,
\end{aligned}
\end{equation}
for the Weber curve, and by
\begin{equation}\label{eq:omega:Airy:base}
\begin{aligned}
	&\omega_{0,3}^{\textup{Airy}}(z_1,z_2,z_3)
	=
	- d_1 d_2 d_3 \left( \frac{1}{z_1 z_2 z_3} \right)
	, \\
	&\omega_{\frac12,2}^{\textup{Airy}}(z_1,z_2)
	=
	-\mf{b} \, d_1 d_2 \left( \frac{z_1^2+z_1z_2+z_2^2}{2 z_1^2 z_2^2(z_1+z_2)} \right)
	, \\
	&\omega_{1,1}^{\textup{Airy}}(z_1)
	=
	- d_1 \left( \frac{1 + 5\mf{b}^2}{24 z_1^3} \right)
	,
\end{aligned}
\end{equation}
for the Airy curve.

A basic consequence of the recursive definition is that $\omega_{g,n}$ is a polynomial in $\mf{b}$ of degree at most $2g$, with parity congruent to $2g \pmod{2}$. Moreover, setting $\mf{b}=0$ recovers the Chekhov--Eynard--Orantin correlators on the corresponding spectral curve. Further properties of the refined topological recursion correlators are summarised below.

\begin{theorem}[{\cite{KO23,Osu24b,KO25}}]\label{thm:RTR}
	Let $\omega_{g,n}$ be the refined topological recursion correlators associated with either the Weber or Airy spectral curve. Then:
	\begin{description}
		\item[RTR1]\label{prop:RTR1}
		$\omega_{g,n}$ is symmetric in its $n$ variables.

		\item[RTR2]\label{prop:RTR2}
		For $2g-2+n>0$, the poles of $\omega_{g,n}$ in the variable $z_1$ lie in $\mc{R}\cup\{\sigma(z_2),\ldots,\sigma(z_n)\}$.

		\item[RTR3]\label{prop:RTR3}
		For $2g-2+n>0$, $\omega_{g,n}$ has vanishing residues in each variable.
	\end{description}
	For the refined Weber spectral curve, the iterated integrals from $0$ to $\infty$ can be computed explicitly:
	\begin{description}
		\item[RTR4]\label{prop:RTR4}
		For $2g-2+n>0$,
		\begin{equation}
			\int_0^\infty \cdots \int_0^\infty
			\omega^{\textup{Web}}_{g,n}
			=
			(-1)^{n} \,
			\Gamma(2g-2+n) \,\frac{
				B_{2,2g}(\frac{\mu+1}{2}\mf{b} \,|\, \beta^{1/2}, -\beta^{-1/2})
			}{(2g)!},
		\end{equation}
		where $\mf{b}=\beta^{1/2}-\beta^{-1/2}$ and the double Bernoulli polynomials $B_{2,k}(a\,|\,u_1,u_2)$ are defined by
		\begin{equation}
			\frac{w^2 e^{a w}}{(e^{u_1 w}-1)(e^{u_2 w}-1)}
			=
			\sum_{k\ge 0}
				B_{2,k}(a \,|\, u_1,u_2)\,\frac{w^k}{k!} .
		\end{equation}
	\end{description}
\end{theorem}

\begin{proof}
	The first three properties are proved in \cite[theorem~2.17]{KO23}. The last statement follows from \cite[corollary~4.3]{Osu24b} together with \cite[theorem~3.6]{KO25}. More explicitly, rescale $x$ and $y$ by $\sqrt{t}$ for $t \in\C^*$. In particular, $\omega^{\textup{Web}}_{0,1}$ is rescaled by $t$. The paper \cite{KO25} computes the \emph{free energies} $F_g(t)$ of this rescaled Weber curve in terms of double Bernoulli polynomials for $g>1$ (with $F^{\textup{Web}}_0$, $F^{\textup{Web}}_{1/2}$, and $F^{\textup{Web}}_1$ treated separately). Moreover, the variational formula of \cite{Osu24b} implies that for $2g-2+n>1$,
	\begin{equation}\label{eq:variational-formula}
		\int_0^\infty \cdots \int_0^\infty
		\omega^{\textup{Web}}_{g,n}(z_1,\ldots,z_n;t)
		=
		\frac{\partial^n}{\partial t^n}
		F^{\textup{Web}}_g(t) .
	\end{equation}
	Substituting the explicit expression for $F^{\textup{Web}}_g(t)$ into \labelcref{eq:variational-formula} and setting $t=1$ yields \hyperref[prop:RTR4]{RTR4}. Note that $t$ above is denoted $m$ in \cite{KO25}, while $\mu$ in the present paper has the opposite sign convention to the one used in \cite{KO25}.
\end{proof}

\subsection{Laplace transforming the Weber correlators}
\label{ssec:Weber}
In this section, we prove point~(\labelcref{Weber}) of \cref{prop:RTR}, which concerns the refined lattice point count $N_{g,n}$. To this end, define $\ms{N}_{g,n}$ by discrete Laplace transforming the refined correlators, i.e.~expanding them at $z_i=0$:
\begin{equation}
	\omega_{g,n}^{\mathrm{Web}}(z_1,\ldots,z_n)
	\eqqcolon
	(-1)^{n} \frac{2}{(1+b)^{g}}
	\sum_{L_1,\ldots,L_n > 0}
	\ms{N}_{g,n}(L_1,\ldots,L_n)
	\prod_{i=1}^n L_i \, z_i^{L_i - 1} \, dz_i .
\end{equation}
Since $z_i=0$ is not a ramification point, property~\hyperref[prop:RTR2]{RTR2} of \cref{thm:RTR} ensures that this expansion is well defined, while \hyperref[prop:RTR1]{RTR1} implies that $\ms{N}_{g,n}$ is symmetric in its $n$ variables. Point~(\labelcref{Weber}) of \cref{prop:RTR} amounts to the identity $N_{g,n} = \ms{N}_{g,n}\big|_{\mu=-1}$. The proof proceeds by showing that both quantities satisfy the same recursion for $2g-2+n>1$, and that the initial data in the cases $2g-2+n=1$ coincide. The recursion for $\ms{N}_{g,n}$ is obtained by discrete Laplace transforming the refined topological recursion; for completeness, we derive it for general~$\mu$. The matching with the recursion for $N_{g,n}$ will then rely on the specialisation $\mu=-1$.

For $2g-2+n=1$, a direct computation from \cref{eq:omega:Weber:base} gives
\begin{equation}
\begin{aligned}
	\ms{N}_{0,3}(L_1,L_2,L_3)
	&=
	\frac{1 + (-1)^{L_1+L_2+L_3}}{2}
		\frac{1}{2}, \\
	\ms{N}_{\frac{1}{2},2}(L_1,L_2)
	&=
	\frac{1 + (-1)^{L_1+L_2}}{2} \,
		b \frac{\max(L_1,L_2) + \mu}{4}, \\
	\ms{N}_{1,1}(L_1)
	&=
	\frac{1 + (-1)^{L_1}}{2} \, \frac{
		(1+b)(L_1^2 - 4)
		+
		b^2 (5L_1^2 + 12\mu L_1 + 6\mu^2 - 2)
	}{96} .
\end{aligned}
\end{equation}
We emphasise again the specialisation $\mu=-1$, which ensures that the $(\tfrac{1}{2},2)$ and $(1,1)$ initial data coincide with those of $N_{g,n}$.

For $2g-2+n>1$, the recursion for $\ms{N}_{g,n}$ follows from that for $\omega_{g,n}$. The argument proceeds in three steps: the recursion input $\Rec_{g,n}$ is first split into three pieces corresponding to the $\mc{R}$-, $\mc{E}$-, and $\mc{D}$-terms; the resulting residues are then evaluated by contour deformation; finally, each contribution is expanded at $z_i=0$ to read off the recursion for $\ms{N}_{g,n}$.

The splitting of $\Rec_{g,n}$ is
\begin{equation}\label{RTRdecomposition}
	\Rec_{g,n}(z;z_2,\ldots,z_n)
	=
	\sum_{m=2}^n \omega_{\mc{R}_m}(z) + \omega_{\mc{E}}(z) + \omega_{\mc{D}}(z) ,
\end{equation}
where
\begin{equation}
\begin{aligned}
	\omega_{\mc{R}_m}(z)
	&\coloneqq
	\omega_{g,n-1}(z,z_2,\ldots,\widehat{z_m},\ldots,z_n)
	\left(
		2\omega_{0,2}(z,z_m) + \frac{dx(z)dx(z_m)}{(x(z) - x(z_m))^2}
	\right)
	,\\
	\omega_{\mc{E}}(z)
	&\coloneqq
	2 \omega_{\frac12,1}(z) \, \omega_{g-\frac{1}{2},n}(z,z_2,\ldots,z_n)
	+
	\mf{b} \, dx(z) \, d_z \frac{\omega_{g-\frac{1}{2},n}(z,z_2,\ldots,z_n)}{dx(z)}
	,\\
	\omega_{\mc{D}}(z)
	&\coloneqq
	\omega_{g-1,n+1}(z,z,z_2,\ldots,z_n)
	+
	\sum_{\substack{g_1+g_2=g \\ I_1 \sqcup I_2 = \{2,\ldots,n\}}}^{\textup{stable}}
	\omega_{g_1,1+|I_1|}(z,z_{I_1})
	\omega_{g_2,1+|I_2|}(z,z_{I_2})
	.
\end{aligned}
\end{equation}
The remaining variables play a spectator role and are hence suppressed from the notation. Note that the sum in $\omega_{\mc{D}}$ is restricted to \emph{stable} topologies; the unstable contributions $(0,2)$ and $(\tfrac12,1)$ have been absorbed into the $\mc{R}$- and $\mc{E}$-terms, respectively.

Next, observe that the integrand in the recursion formula has poles in the integration variable $z$ only at the ramification points, and at $z=z_i$ and $z=\sigma(z_i)$. Since we work on $\P^1$, the residue theorem gives
\begin{equation}
	\left(
		\sum_{i=1}^n \biggl( \Res_{z=z_i} - \Res_{z=\sigma(z_i)} \biggr)
		-
		\sum_{r \in \mc{R}} \Res_{z=r}
	\right)
	=
	2 \sum_{i=1}^n \Res_{z=z_i} .
\end{equation}
We now evaluate these residues case by case.

Let us begin with the $\mc{E}$-term, which is specific to the refined setting. By \hyperref[prop:RTR2]{RTR2} in \cref{thm:RTR}, the form $\omega_{\mc{E}}(z)$ has no poles at $z=z_i$. For the Weber curve, however, the kernel $K(z_1,z)$ has a simple pole at $z=z_1$, and is given by
\begin{equation}
	K(z_1,z)
	=
	- \frac{z^3 \, dz_1}{2(1-z^2)(z-z_1)(1-zz_1) \, dz} .
\end{equation}
Therefore,
\begin{equation}
	2 \sum_{i=1}^n \Res_{z=z_i} \, K(z_1,z) \, \omega_{\mc{E}}(z)
	=
	- \frac{z_1^3}{(1-z_1^2)^2 \, dz_1}
	\omega_{\mc{E}}(z_1) .
\end{equation}
Expanding $\omega_{\mc{E}}$ at $z_i=0$ yields
\begin{multline}
	- \frac{z_1^3}{(1-z_1^2)^2 \, dz_1}\,
	\omega_{\mc{E}}(z_1)
	= 
	(-1)^{n} \frac{2}{(1+b)^{g}}
	\\
	\times
	b
	\sum_{k,p,L_2,\ldots,L_n > 0}
		\frac{1+(-1)^k}{2}\,
		\frac{(k + p + \mu)k}{2}\,
		\ms{N}_{g-\frac12,n}(p,L_2,\ldots,L_n)\,
		p \, z_1^{k+p-1} dz_1
		\prod_{i=2}^n L_i \, z_i^{L_i - 1} \, dz_i .
\end{multline}
Extracting the coefficient of $(-1)^{n} \frac{2}{(1+b)^{g}}\bigl[\prod_{i=1}^n L_i \, z_i^{L_i - 1} \, dz_i\bigr]$, and assuming $\sum_i L_i$ is even, we obtain
\begin{equation}
	b
	\sum_{p > 0}
		p (L_1 + \mu)
		\frac{[L_1 -p]_+}{2 L_1}
		\ms{N}_{g-\frac12,n}(p,L_2,\ldots,L_n).
\end{equation}
The analogous computations for the $\mc{R}$- and $\mc{D}$-terms do not differ from the unrefined setting, and were carried out in detail in \cite{Nor13,GMM}. Collecting all contributions gives
\begin{equation}
\begin{split}
	N_{g,n}(L_1,\ldots,L_n)
	&=
	\sum_{m=2}^n \sum_{p > 0}
		p \, \mc{R}(L_1,L_m,p)
		N_{g,n-1}(p,L_2,\ldots,\widehat{L_m},\ldots,L_n) \\
	&\qquad
	+
	b \sum_{p > 0}
		p(L_1 + \mu) \, \mc{E}(L_1,p)
		N_{g-\frac{1}{2},n}(p,L_2,\ldots,L_n) \\
	&\qquad\qquad
	+
	\sum_{p,q > 0}
		p q \, \mc{D}(L_1,p,q)
		\Bigg(
			\frac{1+b}{2} N_{g-1,n+1}(p,q,L_2,\ldots,L_n) \\
	&\qquad\qquad\qquad 
			+
			\sum_{\substack{g_1+g_2 = g \\ I_1 \sqcup I_2 = \{2,\ldots,n\}}}
				N_{g_1,1+|I_1|}(p,L_{I_1})
				N_{g_2,1+|I_2|}(q,L_{I_2})
		\Bigg) .
\end{split}
\end{equation}
Here $\mc{R}$, $\mc{E}$, and $\mc{D}$ are as in \cref{eq:kernels}. Specialising to $\mu=-1$ makes the prefactor in front of the $\mc{E}$-kernel match the corresponding $b$-term in \cref{thm:lattice:rec}, so that the recursion above agrees with the refined lattice point recursion, completing the proof of point~(\labelcref{Weber}) of \cref{prop:RTR}.

\subsection{Laplace transforming the Airy correlators}
\label{ssec:Airy}
In this section, we prove point~(\labelcref{Airy}) of \cref{prop:RTR}. The strategy parallels the Weber case, but now uses the (continuous) Laplace transform. For $\Re(z_i)>0$, $i \in [n]$, suppose that there exists a continuous piecewise polynomial $\ms{V}_{g,n}$, defined for $2g-2+n\ge 1$, such that
\begin{equation}
	\omega^{\textup{Airy}}_{g,n}(z_1,\ldots,z_n)
	\eqqcolon
	\frac{2}{(1+b)^g}
	\int_0^{+\infty} \cdots \int_0^{+\infty}
		\ms{V}_{g,n}(L_1,\ldots,L_n)
		\prod_{i=1}^n L_i\,e^{-z_i L_i}\,dL_i\,dz_i.
\end{equation}
Unlike in the Weber case, the existence of $\ms{V}_{g,n}$ is not automatic. If it exists, however, it is unique: continuity rules out measure-zero ambiguities. Existence will follow from showing that $\ms{V}_{g,n}=V_{g,n}$.

The proof proceeds by induction on $2g-2+n$. The base cases $2g-2+n=1$ follow from direct computation. Assume that $\ms{V}_{g_0,n_0}=V_{g_0,n_0}$ holds for all $(g_0,n_0)$ with $2g_0-2+n_0 < 2g-2+n$. We now split the recursion input into $\mc{R}$-, $\mc{E}$-, and $\mc{D}$-terms as in \cref{RTRdecomposition}. Again, we begin with the $\mc{E}$-term, which is specific to the refined setting. For the Airy curve, the recursion kernel is
\begin{equation}
	K(z_1,z)
	=
	\frac{dz_1}{2(z-z_1)(z+z_1)\,dz},
\end{equation}
and has a simple pole at $z=z_1$. The $\omega_{\mc{E}}$ contribution therefore reads
\begin{equation}
\begin{split}
	\frac{\omega_{\mc{E}}(z_1)}{2z_1^2\,dz_1}
	&=
	\frac{b}{(1+b)^{g}}
	\int_0^{+\infty}\cdots\int_0^{+\infty}
	\frac{pz_1+2}{z_1^3}\,
	V_{g-\frac12,n}(p,L_2,\ldots,L_n)\,
	p\,e^{-pz_1}\,dp\,dz_1\,
	\prod_{i=2}^n L_i\,e^{-z_i L_i}\,dL_i\,dz_i \\
	&=
	\frac{b}{(1+b)^{g}}
	\int_0^{+\infty}\cdots\int_0^{+\infty}
	pq(p+q)\,
	V_{g-\frac12,n}(p,L_2,\ldots,L_n)\,
	e^{-(p+q)z_1}\,dp\,dq\,dz_1\,
	\prod_{i=2}^n L_i\,e^{-z_i L_i}\,dL_i\,dz_i \\
	&=
	\frac{2b}{(1+b)^{g}}
	\int_0^{+\infty} \cdots \int_0^{+\infty}
		p \frac{[L_1-p]_+}{2}\,
		V_{g-\frac12,n}(p,L_2,\ldots,L_n)\,dp\,
		\prod_{i=1}^n L_i\,e^{-z_i L_i}\,dL_i\,dz_i \\
	&=
	\frac{2}{(1+b)^{g}}
	\int_0^{+\infty} \cdots \int_0^{+\infty}
	b\left(\int_0^{+\infty}
		pL_1\,\mc{E}(L_1,p)\,
		V_{g-\frac12,n}(p,L_2,\ldots,L_n)\,dp\right)
	\prod_{i=1}^n L_i\,e^{-z_i L_i}\,dL_i\,dz_i.
\end{split}
\end{equation}
The first equality uses the induction hypothesis together with the Laplace representation of $\omega^{\textup{Airy}}_{g-1/2,n}$. For the second equality, we write $(pz_1+2)z_1^{-3}$ as a Laplace transform in $q$. The third equality follows from the change of variables $(p,q)\mapsto (p,L_1)$ with $L_1=p+q$, and the last equality is simply the definition of the kernel $\mc{E}$.

The contributions of the $\mc{R}$- and $\mc{D}$-terms are the same as in the unrefined setting, and we omit them. Recognising the sum of all contributions as the Laplace transform of the right-hand side of the volume recursion \labelcref{eq:volume:rec}, we obtain
\begin{equation}
	\int_0^{+\infty}\cdots\int_0^{+\infty}
	\ms{V}_{g,n}(L_1,\ldots,L_n)\prod_{i=1}^n L_i\,e^{-z_i L_i}\,dL_i\,dz_i
	=
	\int_0^{+\infty}\cdots\int_0^{+\infty}
	V_{g,n}(L_1,\ldots,L_n)\prod_{i=1}^n L_i\,e^{-z_i L_i}\,dL_i\,dz_i.
\end{equation}
Thus $\ms{V}_{g,n}$ and $V_{g,n}$ agree almost everywhere. Since both are continuous, it follows that $\ms{V}_{g,n}=V_{g,n}$, completing the induction.

\section{Properties of the refined lattice point count}
\label{sec:properties}

The goal of this section is to deduce the properties of the refined lattice point count stated in \cref{thm:polynomiality}. The main inputs are the result of \cref{sec:RTR} asserting that refined topological recursion on the Weber curve computes the lattice point count, and the the recursion for $N_{g,n}$ proved in \cref{sec:defs}.

\subsection{Polynomiality properties}
We begin by proving a structural description of the refined lattice point count $N_{g,n}(L)$ as a consequence of refined topological recursion. Consider the polytope $Q(L)\subset \R_{\ge 0}^{n^2}$ defined by the matrix equation $M\alpha=L$, where $M$ is the $n\times n^2$ matrix whose columns are the vectors $2e_i$ for $i\in[n]$ and $e_i+e_j$ for distinct $i,j\in[n]$, and where $e_i$ denotes the $i$th standard basis vector of $\R^n$. For $L\in\Z_{>0}^n$, denote by $Q^{\Z}(L)$ the set of integral points in $Q(L)$. The refined lattice point count $N_{g,n}(L)$ can be expressed as a finite sum of polynomially weighted lattice point counts in shifted polytopes $Q(L-m)$, for certain shifts $m\in\Z_{\ge 0}^n$.

\begin{lemma}\label{lem:poly:weighted:count}
	The refined lattice point count can be written as
	\begin{equation}
		N_{g,n}(L)
		=
		\sum_{m \in \Z^n_{\geq 0}} \sum_{\alpha \in Q^{\Z}(L-m)} p_m(\alpha),
	\end{equation}
	where each $p_m(\alpha)$ is a polynomial in $\alpha$, and all but finitely many $p_m$ vanish.
\end{lemma}

The point of \cref{lem:poly:weighted:count} is that it reduces the analysis of $N_{g,n}$ to (finite sums of) \emph{polynomially} weighted lattice point counts. A priori, the definition of $N_{g,n}$ only exhibits it as a sum of \emph{rationally} weighted counts, since the MON is a rational function of the edge lengths.

\begin{proof}
	We use point~(\labelcref{Weber}) of \cref{prop:RTR}, which states that the refined lattice point counts $N_{g,n}$ are encoded in the symmetric differentials $\omega_{g,n}^{\textup{Web}}$ produced by refined topological recursion as
	\begin{equation}
		\omega_{g,n}^{\textup{Web}}(z_1,\ldots,z_n)
		=
		(-1)^{n} \frac{2}{(1+b)^g}
		\sum_{L_1,\ldots,L_n > 0}
		N_{g,n}(L_1,\ldots,L_n)
		\prod_{i=1}^n L_i \, z_i^{L_i - 1} \, dz_i.
	\end{equation}
	Fix $g\in\frac{1}{2}\Z_{\ge 0}$ and $n\in\Z_{>0}$ with $2g-2+n>0$. By \cite[lemma~3.5]{KO23}, the iterated integral
	\begin{equation}\label{def of F_{g,n}}
		F_{g,n}(z_1,\ldots,z_n)
		\coloneqq
		\int_{0}^{z_1}\cdots\int_{0}^{z_n} \omega_{g,n}^{\textup{Web}}
	\end{equation}
	is a meromorphic function of $z_1,\ldots,z_n$ whose poles lie only at $z_i=\sigma(z_j)$, for $i,j\in[n]$. In particular, one can write
	\begin{equation}
		F_{g,n}(z_1,\ldots,z_n)
		=
		\frac{P_{g,n}(z_1,\ldots,z_n)}{\prod_{i,j=1}^n (1-z_i z_j)^{d_{i,j}+1}},
	\end{equation}
	for some integers $d_{i,j}\in\Z_{\ge 0}$ and some polynomial $P_{g,n}$. Expanding at $z_i=0$ gives
	\begin{equation}
		F_{g,n}(z_1,\ldots,z_n)
		=
		P_{g,n}(z_1,\ldots,z_n)
		\prod_{i,j=1}^n
		\sum_{\alpha_{i,j} \ge 0}
			\binom{d_{i,j}+\alpha_{i,j}}{d_{i,j}}
			(z_i z_j)^{\alpha_{i,j}}.
	\end{equation}
	Write $P_{g,n}(z_1,\ldots,z_n) = \sum_m c_m\,z_1^{m_1}\cdots z_n^{m_n}$, where the sum over $m \in \Z_{\ge 0}^n$ is finite. Extracting the coefficient of $z_i^{L_i}$ then expresses $N_{g,n}$ as a finite sum of terms of the form
	\begin{equation}\label{eq:poly_weighted}
		(-1)^{n}
		\frac{(1+b)^g}{2}
		\sum_{\alpha \in Q^{\Z}(L-m)}
			c_m
			\prod_{i,j=1}^n
				\binom{d_{i,j}+\alpha_{i,j}}{d_{i,j}},
	\end{equation}
	where $Q(L-m)\subset\R_{\ge 0}^{n^2}$ is the polytope determined by $M\alpha=L-m$ (equivalently, by $\sum_{j=1}^n(\alpha_{i,j}+\alpha_{j,i})=L_i-m_i$ for $i\in[n]$). Since the binomial coefficients in \cref{eq:poly_weighted} are polynomials in $\alpha_{i,j}$, this has the desired form.
\end{proof}

\begin{proposition}\label{prop:poly_props}
	The refined lattice point count $N_{g,n}(L)$ is a symmetric, rational, piecewise quasipolynomial of period $2$ in the boundary lengths $(L_1,\ldots,L_n)$. Moreover, it is a polynomial in $b$ of degree at most $2g$.
\end{proposition}

In other words, once we fix the parities of the $L_1,\ldots, L_n$, the function $N_{g,n}(L)$ agrees with a piecewise polynomial in $(L_1,\ldots,L_n)$ with rational coefficients.

\begin{proof}
	The symmetry and rationality of $N_{g,n}$ are immediate from \cref{def:Ngn}: symmetry follows from summing over labelled M\"obius graphs, while rationality comes from the orbifold weights $1/|\Aut(G)|$.

	To prove piecewise quasipolynomiality, we use \cref{lem:poly:weighted:count}, which expresses $N_{g,n}$ as a finite sum of polynomially weighted counts of integer points in polytopes of the form $Q(L-m)$. Such a polytope is defined by the matrix equation $M\alpha=L-m$, where $M$ is the $n\times n^2$ matrix whose columns are $2e_i$ for $i\in[n]$ and $e_i+e_j$ for $i,j\in[n]$ with $i\neq j$. An extension of Ehrhart theory developed in \cite{BBDKV19} implies that each polynomially weighted count in \cref{eq:poly_weighted} is a piecewise quasipolynomial in $L-m$, with period dividing the least common multiple of the denominators of the vertices of the polytope. Since the columns of $M$ generate an index-$2$ sublattice of $\Z^n$, the vertices of $Q(L-m)$ have coordinates with denominators at worst $2$. It follows that the period of $N_{g,n}$ is $2$.

	Finally, polynomiality in $b$ follows from the fact that the MON is a polynomial in $b$ of degree at most $2g$, see \cref{prop:MON}.
\end{proof}

\subsection{Wall-and-chamber structure}
In the previous section, we proved that the $N_{g,n}$ are piecewise quasipolynomials in the boundary lengths. The following result gives a precise description of the wall-and-chamber structure.

\begin{proposition}\label{lem:walls}
	The refined lattice point count $N_{g,n}$ is a piecewise quasipolynomial with walls in $\R_{\ge 0}^n$ given by the equations
	\begin{equation}
		\sum_{i=1}^n \epsilon_i L_i = 0,
		\qquad
		\epsilon_i \in \set{+1,-1,0}.
	\end{equation}
	Moreover, $N_{g,n}(L)$ is continuous across these walls. In particular, after fixing the parity class of $(L_1,\ldots,L_n)$, it extends to a continuous function on all of $\R_{\ge 0}^n$.
\end{proposition}

\begin{proof}
	Using the structural expression of $N_{g,n}$ as a polynomially weighted count of lattice points in the polytopes $Q(L-m)$ from \cref{lem:poly:weighted:count}, we first describe the possible walls. By \cite{BBDKV19}, the walls arise from hyperplanes spanned by $n-1$ linearly independent columns of the matrix $M$. Let $H_I\subset \R^n$ be such a wall, spanned by columns $\{M_i\}_{i\in I}$ with $I\subset[n^2]$ and $|I|=n-1$. A normal vector $u_I\in\R^n$ to $H_I$ satisfies
	\begin{equation}
		u_I\cdot M_i = 0 \qquad \text{for all } i\in I.
	\end{equation}
	Since each column of $M$ is either of the form $2e_r$ or $e_r+e_s$, these orthogonality can be used to force $u_I$ to have entries in $\{+1,-1,0\}$. Taking into account the shifts $L\mapsto L-m$ in \cref{lem:poly:weighted:count}, this shows that the walls have equations of the form $\sum_{i=1}^n \epsilon_i L_i = s_m$ with $\epsilon_i\in\{+1,-1,0\}$.

	It remains to show that in fact $s_m=0$, i.e.\ that the walls pass through the origin. This is proved by induction on $2g-2+n$, using the symmetric recursion \labelcref{eq:lattice:rec:sym}. The base cases follow from the explicit formulas in \cref{app:base:top}. Fix a parity class by imposing $L_i\equiv \delta_i \pmod 2$ with $\delta_i\in\{0,1\}$. Consider the $\mc{R}$-term in \cref{eq:lattice:rec:sym}:
	\begin{equation}\label{eq:R-term:chambers}
		\sum_{p>0}
			p\,[L_i+L_j-p]_+\,
			N_{g,n-1}\bigl(p,L_{[n]\setminus\{i,j\}}\bigr).
	\end{equation}
	Fix a chamber $\mf{c}$ in the $(n-1)$ variables $(p,L_{[n]\setminus\{i,j\}})$. On $\mf{c}$ the function $N_{g,n-1}$ is represented by a polynomial, and it vanishes unless $p\equiv \delta_i+\delta_j \pmod 2$; denote this polynomial by $N^{\mf{c}}_{g,n-1}(p,L_{[n]\setminus\{i,j\}})$. As $p$ ranges between $0$ and $L_i+L_j$, passing between chambers only changes the summation bounds, which by the induction hypothesis are given by linear equations of the form $\epsilon p+\sum_{k\neq i,j}\epsilon_k L_k=0$. Consequently, \cref{eq:R-term:chambers} can be written as a finite sum of expressions of the form
	\begin{equation}\label{eq:R-term:sums}
		\sum_{\substack{b^{\mf{c}} \le p \le B^{\mf{c}} \\ p \equiv \delta_i+\delta_j \;(\mathrm{mod}\,2)}}
			p(L_i+L_j-p)\,
			N^{\mf{c}}_{g,n-1}\bigl(p,L_{[n]\setminus\{i,j\}}\bigr),
	\end{equation}
	where $b^{\mf{c}}$ and $B^{\mf{c}}$ are linear functions of $L_{[n]\setminus\{i,j\}}$. By Faulhaber's formula, each such sum is a polynomial in $(L_1,\ldots,L_n)$. When any two of the linear bounds $b^{\mf{c}}$ and $B^{\mf{c}}$ (across all the chambers) coincide, the splitting of \labelcref{eq:R-term:chambers} into sums of the form \labelcref{eq:R-term:sums} changes. Hence, \cref{eq:R-term:chambers} defines a piecewise polynomial, with walls given by linear equations in $(L_1,\ldots,L_n)$ with integer coefficients. In particular, the walls pass through the origin. The $\mc{E}$- and $\mc{D}$-terms are treated in the same way and are omitted.

	Finally, continuity across the walls is not automatic for general weighted lattice point counts, but follows in our case directly from the recursion \labelcref{eq:lattice:rec:sym}.
\end{proof}

\subsection{Degree}
We prove that $N_{g,n}(L)$ has degree $6g-6+2n$ in the boundary lengths.

\begin{proposition}
	The piecewise quasipolynomials $N_{g,n}(L)$ have degree $6g-6+2n$ in $(L_1,\ldots,L_n)$.
\end{proposition}

\begin{proof}
	The proof is by induction on $2g-2+n$, using the symmetric recursion \labelcref{eq:lattice:rec:sym}. The base cases follow from the explicit formulas in \cref{app:base:top}. Fix a parity class by choosing $\delta_i\in\{0,1\}$ such that $L_i\equiv \delta_i \pmod 2$ throughout.

	As in the proof of \cref{lem:walls}, each summation in the recursion can be decomposed chamberwise: after fixing a chamber $\mf{c}$, the relevant $N$ is represented by a polynomial (with the appropriate parity constraint), and the summation bounds become linear functions of the remaining boundary lengths. We use this repeatedly below.

	Consider first the $\mc{R}$-term. In each chamber $\mf{c}$ it is a finite sum of expressions of the form
	\begin{equation}\label{eq:R-term:chambers:encore}
		\sum_{\substack{b^{\mf{c}} \le p \le B^{\mf{c}} \\ p \equiv \delta_i+\delta_j \;(\mathrm{mod}\,2)}}
			p(L_i+L_j-p)\,
			N^{\mf{c}}_{g,n-1}\bigl(p,L_{[n]\setminus\{i,j\}}\bigr),
	\end{equation}
	where $b^{\mf{c}}$ and $B^{\mf{c}}$ are linear in $L_{[n]\setminus\{i,j\}}$. By the induction hypothesis, $N^{\mf{c}}_{g,n-1}$ has degree $6g-6+2n-2$. Faulhaber's formula then shows that \cref{eq:R-term:chambers:encore} has degree $6g-6+2n+1$ in $(L_1,\ldots,L_n)$. After the overall division by $L_1+\cdots+L_n$ in \cref{eq:lattice:rec:sym}, this contributes degree $6g-6+2n$ to $N_{g,n}$. The $\mc{E}$-term is treated in the same way and is omitted.

	A similar argument applies to the $\mc{D}$-terms. We treat only the connected contribution; the disconnected case is identical. Chamberwise, the relevant double sums are of the form
	\begin{equation}
		\sum_{\substack{b_1^{\mf{c}}\le p\le B_1^{\mf{c}} \\ b_2^{\mf{c}}\le q\le B_2^{\mf{c}} \\ p+q \equiv \delta_i \;(\mathrm{mod}\,2)}}
			pq(L_i-p-q)\,
			N^{\mf{c}}_{g-1,n+1}\bigl(p,q,L_{[n]\setminus\{i\}}\bigr),
	\end{equation}
	where $b_1^{\mf{c}},B_1^{\mf{c}}$ are linear in $(L_1,\ldots,L_n)$ and $b_2^{\mf{c}},B_2^{\mf{c}}$ are linear in $(p,L_1,\ldots,L_n)$. By the induction hypothesis, $N^{\mf{c}}_{g-1,n+1}$ has degree $6g-6+2n-4$. Applying Faulhaber's formula twice shows that such a double sum has degree $6g-6+2n+1$ before the division by $L_1+\cdots+L_n$, hence contributes degree $6g-6+2n$ to $N_{g,n}$.

	Alternatively, the degree can be read off directly from \cref{prop:RTR}: one can prove that the refined topological recursion differential $\omega^{\textup{Web}}_{g,n}$ has poles of order at most $6g-4+2n$, which implies that the coefficients $N_{g,n}(L)$ have degree $6g-6+2n$.
\end{proof}

Putting together the results of this section completes the proof of \cref{thm:polynomiality}.

\section{The refined Euler characteristic}
\label{sec:Euler}
The goal of this section is twofold: to show that the refined lattice point count computes a refined orbifold Euler characteristic of the moduli space of metric M\"obius graphs, and to evaluate this quantity explicitly using refined topological recursion. The refined Euler characteristic is defined as the usual signed orbifold count, weighted by the average MON of each cell. As a by-product of this evaluation, we recover the orbifold Euler characteristics of the moduli spaces of Riemann and Klein surfaces, thereby providing a new proof of the Harer--Zagier \cite{HZ86} and Goulden--Harer--Jackson \cite{GHJ01} formulas, and exhibiting a one-parameter refinement that interpolates between them.

\begin{definition}
	Define the \emph{refined Euler characteristic} of the moduli space of metric M\"obius graphs $\mathcal N_{g,n}(L)$ by
	\begin{equation}
		\chi_{g,n}(b)
		\coloneqq
		\sum_{G \in \MG_{g,n}}
			(-1)^{\dim P_{G}(L)}
			\frac{\braket{\rho_G(b)}}{|\Aut(G)|},
		\qquad
		\braket{\rho_G(b)} \coloneqq \rho_G(1,\ldots,1;b).
	\end{equation}
\end{definition}

The quantity $\braket{\rho_G(b)}$ will be referred to as the \emph{average MON} of the M\"obius graph $G$. It is the value of $\rho_G$ at the uniform metric $(1,\ldots,1)$ on $G$. Since $\rho_G$ is a homogeneous rational function of degree zero in the edge lengths, the same value is obtained at any uniform metric $(\ell,\ldots,\ell)$ with $\ell>0$. In this sense, $\braket{\rho_G(b)}$ captures the average non-orientability of the cell associated with~$G$.

To relate $\chi_{g,n}(b)$ to the lattice point count, introduce the formal power series
\begin{equation}
	S_{g,n}(z;b)
	\coloneqq
	\sum_{L_1,\ldots,L_n > 0}
		N_{g,n}(L_1,\ldots,L_n;b)\, z^{L_1+\cdots+L_n}.
\end{equation}
As before, we suppress the dependence on $b$ when it is clear from the context.

The key point is to evaluate $S_{g,n}$ at $z=\infty$ (after analytic continuation) in three different ways. First, using the polytopal cell decomposition of the moduli space, we relate $S_{g,n}(\infty)$ to the refined Euler characteristic. Second, using the piecewise quasipolynomiality of the lattice point count, we identify $S_{g,n}(\infty)$ with the constant term of the lattice point polynomial. Third, we express the same quantity in terms of the iterated integral of refined topological recursion correlators, obtaining an explicit closed formula in terms of double Bernoulli polynomials.

\subsection{Via the polytopal structure}
The series $S_{g,n}$ can be resummed into a meromorphic function of $z$.

\begin{lemma}
	The series $S_{g,n}$ equals the following meromorphic function of $z$ with coefficients in $\Q[b]$:
	\begin{equation}\label{eq:Sgn}
		S_{g,n}(z)
		=
		\sum_{G \in \MG_{g,n}} \frac{\braket{\rho_G}}{|\Aut(G)|}
		\left( \frac{z^2}{1-z^2} \right)^{|E(G)|}.
	\end{equation}
	In particular, $S_{g,n}(\infty) = (-1)^n \chi_{g,n}(b)$.
\end{lemma}

\begin{proof}
	Rewrite the definition of $S_{g,n}(z)$ as
	\begin{equation}
	\begin{split}
		S_{g,n}(z)
		&=
		\sum_{G \in \MG_{g,n}}
			\frac{1}{|\Aut(G)|}
				\sum_{\ell \in \Z_{>0}^{E(G)}}
					\rho_G(\ell)\,
					z^{\sum_{i=1}^n\sum_{e \in E(G)} a_{i,e} \ell_e} \\
		&=
		\sum_{G \in \MG_{g,n}}
			\frac{1}{|\Aut(G)|}
			\sum_{\ell \in \Z_{>0}^{E(G)}}
				\rho_G(\ell)\,
				z^{2 \sum_{e \in E(G)}\ell_e} \\
		&=
		\sum_{G \in \MG_{g,n}}
			\frac{1}{|\Aut(G)|}
			\sum_{T \ge |E(G)|} z^{2T}
				\sum_{\substack{\ell_1,\ldots,\ell_{|E(G)|} > 0 \\ \ell_1 + \cdots +\ell_{|E(G)|} = T}}
					\rho_G(\ell),
	\end{split}
	\end{equation}
	where the first equality is \cref{def:Ngn}, the second uses $\sum_{i=1}^n a_{i,e}=2$, and the third groups terms by the total edge length. For notational convenience, set $E=|E(G)|$ and label the edges by $1,\ldots,E$. We claim that the innermost sum satisfies
	\begin{equation}\label{eq:claim:averageMON}
		\sum_{\substack{\ell_1,\ldots,\ell_{E} > 0 \\ \ell_1 + \cdots +\ell_{E} = T}} \rho_G(\ell)
		=
		\braket{\rho_G}\,
		\sum_{\substack{\ell_1,\ldots,\ell_{E} > 0 \\ \ell_1 + \cdots +\ell_{E} = T}} 1 .
	\end{equation}
	Assuming this claim, \cref{eq:Sgn} follows since
	\begin{equation}
		\sum_{T \ge E} z^{2T}
		\sum_{\substack{\ell_1,\ldots,\ell_{E} > 0 \\ \ell_1 + \cdots +\ell_{E} = T}} 1
		=
		\sum_{T \ge E} \binom{T-1}{E-1} z^{2T}
		=
		\biggl(\frac{z^2}{1-z^2}\biggr)^{E}.
	\end{equation}
	It remains to prove \cref{eq:claim:averageMON}. We argue by induction on $E$. For this part, we drop the face-labelling and valency restrictions, since the claim is meaningful for arbitrary M\"obius graphs as well. The base cases are immediate. Assume the claim holds for $E-1$. Using the definition of $\rho_G$ in \labelcref{eq:def:rho}, we obtain
	\begin{equation}
	\begin{split}
		\sum_{\substack{\ell_1,\ldots,\ell_{E} > 0 \\ \ell_1 + \cdots +\ell_{E} = T}}
			\rho_G(\ell)
		&=
		\sum_{e=1}^{E} \frac{w_e}{4T}
			\sum_{\substack{\ell_1,\ldots,\ell_{E} > 0 \\ \ell_1 + \cdots +\ell_{E} = T}}
				\ell_e \, \rho_{G-e}(\ell-\ell_e) \\
		&=
		\sum_{e=1}^{E} \frac{w_e}{4T}
			\sum_{\ell_e=1}^{T-E+1}
			\ell_e
				\sum_{\substack{\ell_1,\ldots,\widehat{\ell_e},\ldots,\ell_E > 0 \\
				\ell_1 + \cdots +\widehat{\ell_e}+\cdots+\ell_E = T - \ell_e}}
					\rho_{G-e}(\ell-\ell_e) \\
		&=
		\sum_{e=1}^{E} \frac{w_e}{4T}
			\sum_{\ell_e=1}^{T-E+1}
			\ell_e \braket{\rho_{G-e}}
				\sum_{\substack{\ell_1,\ldots,\widehat{\ell_e},\ldots,\ell_E > 0 \\
				\ell_1 + \cdots +\widehat{\ell_e}+\cdots+\ell_E = T - \ell_e}}
					1,
	\end{split}
	\end{equation}
	where the last equality uses the induction hypothesis. Reintroducing $\ell_e$ into the sum yields
	\begin{equation}
		\sum_{\substack{\ell_1,\ldots,\ell_{E} > 0 \\ \ell_1 + \cdots +\ell_{E} = T}}
			\rho_G(\ell)
		=
		\sum_{e=1}^{E} \frac{w_e}{4T}\,\braket{\rho_{G-e}}
		\sum_{\substack{\ell_1,\ldots,\ell_{E} > 0 \\ \ell_1 + \cdots +\ell_{E} = T}}
			\ell_e
		=
		\sum_{e=1}^{E} \frac{w_e}{4E}\,\braket{\rho_{G-e}}
			\sum_{\substack{\ell_1,\ldots,\ell_{E} > 0 \\ \ell_1 + \cdots +\ell_{E} = T}} 1.
	\end{equation}
	To justify the last step, note that the sum $\sum_{\ell_1+\cdots+\ell_E=T} \ell_e$ is independent of $e$ by symmetry. Averaging over $e$ therefore gives\footnote{
		Geometrically, this says that the barycentre of the lattice points in the simplex $\set{ \ell_i>0 | \sum_i \ell_i=T }$ lies on the diagonal, so each coordinate has average $T/E$.
	}
	\begin{equation}
		\frac{1}{T} \sum_{\substack{\ell_1,\ldots,\ell_{E} > 0 \\ \ell_1 + \cdots +\ell_{E} = T}} \ell_e
		=
		\frac{1}{T \, E}\sum_{e=1}^E
		\sum_{\substack{\ell_1,\ldots,\ell_{E} > 0 \\ \ell_1 + \cdots +\ell_{E} = T}} \ell_e
		=
		\frac{1}{E}
		\sum_{\substack{\ell_1,\ldots,\ell_{E} > 0 \\ \ell_1 + \cdots +\ell_{E} = T}} 1.
	\end{equation}
	The defining recursion \labelcref{eq:def:rho}, evaluated at the uniform metric $\ell=(1,\ldots,1)$, gives $\braket{\rho_G}=\sum_{e=1}^E \frac{w_e}{4E}\,\braket{\rho_{G-e}}$, which proves the claimed \cref{eq:claim:averageMON}.

	Finally, \cref{eq:Sgn} implies $S_{g,n}(\infty)=(-1)^n\chi_{g,n}(b)$, since $\dim P_G(L)=|E(G)|-n$.
\end{proof}

\subsection{Via the piecewise quasipolynomial structure}
On the other hand, $S_{g,n}(\infty)$ can be computed using the piecewise quasipolynomiality of $N_{g,n}$. 


\begin{lemma}
	The value of $S_{g,n}$ at infinity is given by $S_{g,n}(\infty)=(-1)^n N^{[0]}_{g,n}(0,\ldots,0)$, where $N^{[0]}_{g,n}$ denotes the piecewise polynomial governing $N_{g,n}$ on inputs where all $L_i$ are even.
\end{lemma}

\begin{proof}
	Consider the generating function $\mathsf{F}_{g,n}(z_1,\ldots,z_n)$ defined as 
	\begin{equation}
		\mathsf{F}_{g,n}(z_1,\ldots,z_n)
		\coloneqq
		\sum_{L_1,\ldots, L_n > 0 } N_{g,n}(L_1,\ldots, L_n) \prod_{i=1}^n z_i^{L_i}.
	\end{equation}
	Note that $\mathsf{F}_{g,n}$ is the multivariable integral of the refined topological recursion correlators $\omega_{g,n}^{\textup{Web}}$ up to the normalization factors appearing in \cref{prop:RTR}. Then, \cite[Lemma~4.1]{KO25} shows that
	\begin{equation}\label{eq:StoF}
		S_{g,n}(\infty)
		=
		\lim_{z_1 \to \infty} \cdots \lim_{z_n \to \infty}
			\mathsf{F}_{g,n}(z_1,\ldots,z_n).
	\end{equation}
	The above identity means that taking the successive limits $z_i\to\infty$ in $\mathsf{F}_{g,n}(z_1,\ldots,z_n)$ is equivalent to first specialising $z_1=\cdots=z_n=z$ and then letting $z\to\infty$. For brevity, we denote the right-hand side of \labelcref{eq:StoF} by $\mathsf{F}_{g,n}(\infty,\ldots,\infty)$.

	To handle the quasipolynomiality, fix the parity class of each $L_i$ and write
	\begin{equation}
		\mathsf{F}_{g,n}(z_1,\ldots,z_n)
		=
		\sum_{\delta \in \{0,1\}^n}
		\sum_{\substack{L_1,\ldots,L_n > 0 \\ L_i \equiv \delta_i \; (\mathrm{mod}\, 2)}}
		N_{g,\delta}(L_1,\ldots,L_n)\, \prod_{i=1}^n z_i^{L_i}.
	\end{equation}
	Here $N_{g,\delta}$ denotes the continuous piecewise polynomial describing $N_{g,n}$ on the parity class $L_i \equiv \delta_i \pmod{2}$. Denote the inner sum by $\mathsf{F}_{g,\delta}(z_1,\ldots,z_n)$. We claim that $\mathsf{F}_{g,\delta}$ is rational and that
	\begin{equation}\label{eq:S:delta:claim}
		\mathsf{F}_{g,\delta}(\infty,\ldots,\infty)
		=
		\begin{cases}
			(-1)^n N_{g,\delta}(0,\ldots,0) & \text{if $\delta=(0,\ldots,0)$,} \\
			0 & \text{otherwise.}
		\end{cases}
	\end{equation}
	In view of \labelcref{eq:StoF}, the lemma follows immediately from this claim. The remainder of the proof is devoted to proving it.

	By \cref{thm:polynomiality}, $N_{g,\delta}$ is a continuous piecewise polynomial whose walls are given by equations $\sum_i \epsilon_i L_i=0$, with $\epsilon_i\in\{+1,-1,0\}$. These walls decompose the orthant $\R_{>0}^n$ into finitely many relatively open cones $\mf{c}$ with apex at the origin and, on each $\mf{c}$, the function $N_{g,\delta}$ is represented by a polynomial $N_{g,\delta}^{\mf{c}}$. After subdividing if necessary, assume that the cones are simplicial and unimodular with respect to $\Z^n$. In particular, for each cone $\mf{c}$ there exists a unimodular matrix $A^{\mf{c}}=(A^{\mf{c}}_{i,j})$ such that
	\begin{equation}
		A^{\mf{c}}\bigl(\Z_{>0}^{\dim(\mf{c})} \times \{0\}^{\codim(\mf{c})}\bigr)
		=
		\mf{c}\cap \Z^n.
	\end{equation}
	Equivalently, the columns of $A^{\mf{c}}$ are the primitive generators of $\mf{c}$. Then, the multivariate generating function of lattice points in $\mf{c}$ can be resummed as
	\begin{equation}
		\sum_{L \in \mf{c}\cap\Z^n} \prod_{i=1}^n z_i^{L_i}
		=
		\prod_{j=1}^{\dim(\mf{c})}
			\frac{\prod_{i=1}^n z_i^{A^{\mf{c}}_{i,j}}}{1 - \prod_{i=1}^n z_i^{A^{\mf{c}}_{i,j}}}.
	\end{equation}
	Applying the differential operator $N_{g,\delta}^{\mf{c}}(z_1\partial_{z_1},\ldots,z_n\partial_{z_n})$ and summing over cones $\mf{c}$ give
	\begin{equation} \label{eq:Fasrational}
		\mathsf{F}_{g,\delta}(z_1,\ldots,z_n)=\sum_{\mf{c}}\sum_{L \in \mf{c}\cap\Z^n}
			N_{g,\delta}^{\mf{c}}(L)\prod_{i=1}^n z_i^{L_i}
		=
		\sum_{\mf{c}}N_{g,\delta}^{\mf{c}}(z_1\partial_{z_1},\ldots,z_n\partial_{z_n})
		\prod_{j=1}^{\dim(\mf{c})}
			\frac{\prod_{i=1}^n z_i^{A^{\mf{c}}_{i,j}}}{1 - \prod_{i=1}^n z_i^{A^{\mf{c}}_{i,j}}}.
	\end{equation}
	Moreover, the restriction to the parity class $L_i\equiv\delta_i$ can be enforced by the standard parity projector:
	\begin{equation}
		\mathsf{F}_{g,\delta}(z_1,\ldots,z_n)
		=
		\frac{1}{2^n} \sum_{p \in \{+1,-1\}^n}
			\left(\prod_{i=1}^n p_i^{\delta_i}\right)
			\mathsf{F}_{g,\delta}(p_1z_1,\ldots,p_nz_n).
	\end{equation}
	On the other hand, taking the limit $z_i\to\infty$ in \labelcref{eq:Fasrational}, all terms involving at least one operator $z_i\partial_{z_i}$ vanish, so only the constant term of each $N_{g,\delta}^{\mf{c}}$ contributes. Since this constant term does not depend on $p_i$ as $z_i\to\infty$, we have
	\begin{equation}
		\mathsf{F}_{g,\delta}(\infty,\ldots,\infty)
		=
		\frac{1}{2^n} \sum_{p \in \{+1,-1\}^n}
			\left(\prod_{i=1}^n p_i^{\delta_i}\right)
			\sum_{\mf{c}}
				(-1)^{\dim(\mf{c})}\,N_{g,\delta}^{\mf{c}}(0,\ldots,0).
	\end{equation}
	The remaining sum over $p$ vanishes unless $\delta=(0,\ldots,0)$, in which case it equals $1$. This proves \cref{eq:S:delta:claim} for $\delta\neq 0$, and yields $\mathsf{F}_{g,(0,\ldots,0)}(\infty,\ldots,\infty)=\sum_{\mf{c}} (-1)^{\dim(\mf{c})} N_{g,(0,\ldots,0)}^{\mf{c}}(0,\ldots,0)$ for $\delta=0$.

	To conclude for $\delta=0$, apply inclusion-exclusion to the fan of cones. One has $\mb{1}_{\R_{\ge 0}^n}=\sum_{\mf{c}} (-1)^{\codim(\mf{c})}\mb{1}_{\overline{\mf{c}}}$. Since $N_{g,(0,\ldots,0)} = N_{g,n}^{[0]}$ is continuous, it follows that for all $L\in\R_{\ge 0}^n$,
	\begin{equation}
		N_{g,n}^{[0]}(L_1,\ldots,L_n)
		=
		\sum_{\mf{c}}
			(-1)^{\codim(\mf{c})} \,N_{g,(0,\ldots,0)}^{\mf{c}}(L)\,\mb{1}_{\overline{\mf{c}}}(L_1,\ldots,L_n).
	\end{equation}
	Evaluating at $L=0$, and noting that $\dim(\mf{c})+\codim(\mf{c})=n$, gives
	\begin{equation}
		N_{g,n}^{[0]}(0,\ldots,0)
		= (-1)^n
		\sum_{\mf{c}}
			(-1)^{\dim(\mf{c})} \,N_{g,(0,\ldots,0)}^{\mf{c}}(0,\ldots,0).
	\end{equation}
	The right-hand side equals $(-1)^n \mathsf{F}_{g,(0,\ldots,0)}(\infty,\ldots,\infty)$, which completes the proof.
\end{proof}

\subsection{Via refined topological recursion: a proof of \cref{thm:EC}}
We conclude by observing that, by the definition of $S_{g,n}$ and \cref{eq:Web:lattice}, one has
\begin{equation}
	\int_{0}^{\infty} \cdots \int_{0}^{\infty}
	\omega^{\textup{Web}}_{g,n}\big|_{\mu=-1}
	=
	(-1)^{n} \frac{2}{(1+b)^{g}}\,
	S_{g,n}(\infty).
\end{equation}
Combining this identity with property \hyperref[prop:RTR4]{RTR4} of \cref{thm:RTR} and the lemmas above, yields \cref{thm:EC}, namely
\begin{equation}
	\chi_{g,n}(b)
	=
	N^{[0]}_{g,n}(0,\ldots,0;b)
	=
	(-1)^{n}\,
	\Gamma(2g-2+n)\,
	\frac{
		B_{2,2g}(0 \,|\, \beta^{1/2}, -\beta^{-1/2})
	}{2 \beta^g (2g)!},
\end{equation}
where $\beta=\frac{1}{1+b}$ relates the refinement parameters.

Finally, \cref{eq:iso:or,eq:iso:nor} imply
\begin{equation}
	\chi(\mc{M}_{g,n}) = 2\,\chi_{g,n}(b)\big|_{b=0},
	\qquad
	\chi(\mc{K}_{g,n})
	=
	2^{n} \Bigl( \chi_{g,n}(b)\big|_{b=1} - \chi_{g,n}(b)\big|_{b=0} \Bigr),
\end{equation}
completing the proof of \cref{thm:EC}.

\appendix

\section{Computing the measure of non-orientability}
\label{app:MON}

In this appendix, we give an example computation of the MON. Consider the graph $G$ of type $(1,1)$ drawn on a Klein bottle, shown in \cref{fig:KBgraph}.

\begin{figure}[b]
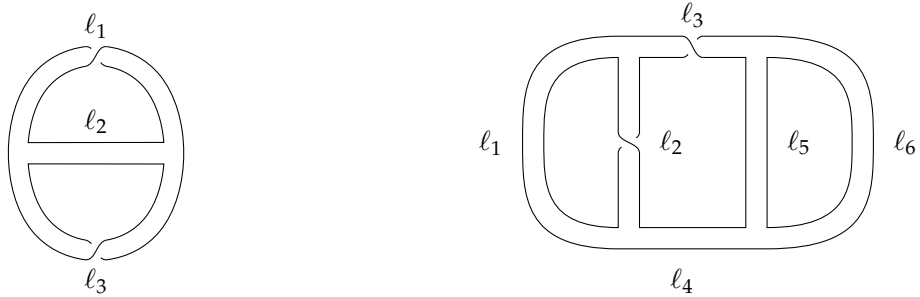

	\begin{subfigure}{.5\textwidth}
		\centering

		\caption{A metric M\"obius graph of type (1,1).}
		\label{fig:KBgraph}
	\end{subfigure}%
	\begin{subfigure}{.5\textwidth}
		\centering
		%
		\caption{A metric M\"obius graph of type (1,2).}
		\label{fig:soccerfield}
	\end{subfigure}	
	\caption{
		Two examples of metric M\"obius graphs.
	}
	\label{fig:MON:exmpl}
\end{figure}
In the top rows of \cref{fig:MON:algorithm}, we list all $12$ possible rootings of $G$, with the root marked in red. In the first four rootings the root lies on the edge of length~$\ell_1$, in the middle four it lies on the edge of length~$\ell_2$, and in the last four it lies on the edge of length~$\ell_3$.
\begin{figure}
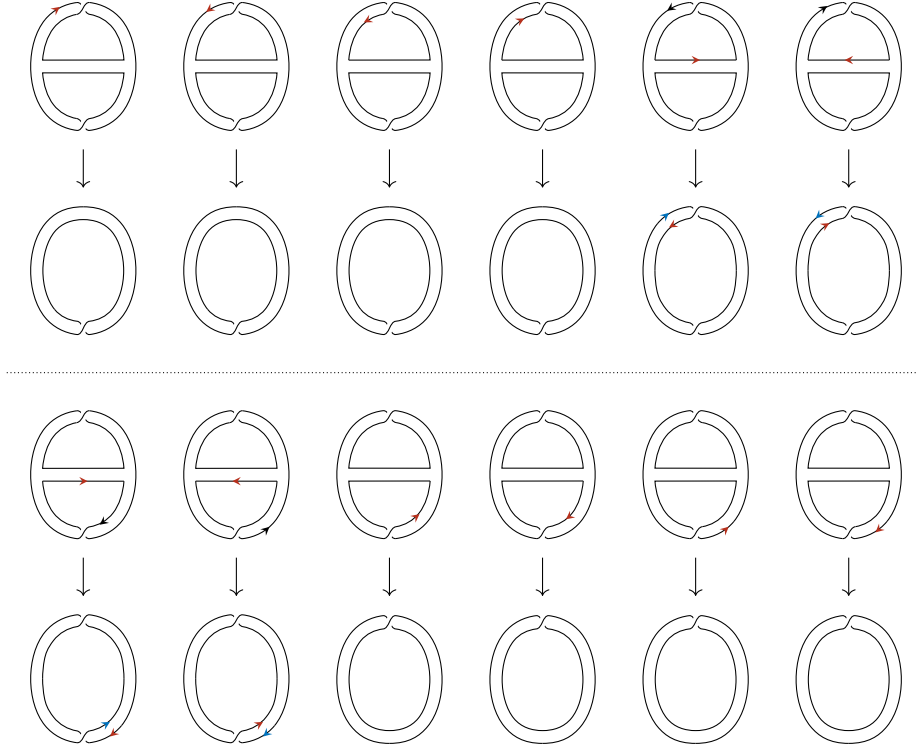

	%
	\caption{MON computation for a graph of type $(1,1)$ drawn on a Klein bottle.}
	\label{fig:MON:algorithm}
\end{figure}
Removing the root from the first or the last four graphs leaves the number of faces unchanged ($\mc{R}$-type); we therefore assign the weight $w=b$. The resulting graph, depicted in the bottom rows, is a M\"obius strip, whose MON is equal to $b$. For the middle four graphs, we indicate in black arrows the orientation induced by the root. In these cases, removing the root increases the number of faces ($\mc{D}^{\rm c}$-type), and we depict the resulting graph in the bottom rows. We mark the new root in red and the orientation induced by this root in blue. Since the orientations of the edge succeeding the root of $G$ do not match, we assign the weight $w=b$. The graph obtained after root removal is orientable, and hence has MON equal to $1$.

Putting these contributions together, we obtain
\begin{equation}
	\rho_G(\ell)
	=
	\frac{1}{4 \ell_1 + 4 \ell_2 + 4 \ell_3} \left(4 \ell_1 b^2 + 4 \ell_2 b + 4 \ell_3 b^2 \right)
	= 
	\frac{1}{\ell_1 + \ell_2 + \ell_3} \left( (\ell_1 + \ell_3) b^2 + \ell_2 b \right).
\end{equation}
The interested reader can also check that the MON for the graph shown in \cref{fig:soccerfield} is
\begin{equation}
	\frac{b^2 \ell_1 - b \ell_1}{\ell_1 + \ell_2 + \ell_3 + \ell_4 + \ell_5 + \ell_6}
	+
	\frac{b \ell_1 - b^2 \ell_1}{\ell_1 + \ell_2 + \ell_3 + \ell_4 + \ell_6}
	+
	\frac{b \ell_1 + b^2 \ell_2 + b^2 \ell_3 + b^2 \ell_4 + b^2 \ell_5}{\ell_1 + \ell_2 + \ell_3 + \ell_4 + \ell_5}.
\end{equation}
In particular, the MON is not a polynomial function of the edge lengths, even if we restrict ourselves to fixed boundary lengths, as the above example shows.

\section{Base topologies}
\label{app:base:top}

In this appendix, we compute the refined lattice point counts for the base-case topologies $(0,3)$, $(\tfrac{1}{2},2)$, and $(1,1)$, corresponding respectively to a pair of pants, a two-holed cross-cap, and the combination of a one-holed torus and a one-holed Klein bottle. 

\subsection*{Pair of pants}
For fixed $(L_1,L_2,L_3) \in \Z_{>0}^{3}$, there exists a single integral metric M\"obius graph with $(L_1,L_2,L_3)$ as perimeters if $L_1+L_2+L_3$ is even, and none otherwise. Since the MON is constantly one and the order of the automorphism group is $2$, we find
\begin{equation}
	N_{0,3}(L_1,L_2,L_3;b) = \frac{1+(-1)^{L_1+L_2+L_3}}{2} \frac{1}{2}.
\end{equation}

\subsection*{Two-holed cross-cap}
Fix $(L_1,L_2) \in \Z_{>0}^{2}$. Again, there are no integral metric M\"obius graphs unless $L_1 + L_2$ is even. Notice that the MON is constantly $b$ in this case. We split the computation into three cases, referring to \cref{ex:chi1} for the contributing graphs.

First, suppose $L_1 = L_2 = L$. In this case, only the last graph in \cref{ex:chi1} contributes, giving $\frac{1}{4} \sum_{\ell_1 + \ell_2 = L} b = \frac{b}{4}(L-1)$.

Next, suppose $L_1 > L_2$. In this case, only the first, third, and fifth graphs contribute. In this order, they add to the refined lattice point count as
\begin{equation}
	\frac{1}{4} \sum_{\substack{\ell_1 + \ell_2 + 2\ell_3 = L_1 \\ \ell_1 + \ell_2 = L_2}} b
	=
	\frac{b}{4} (L_2 - 1),
	\quad
	\frac{1}{2} \sum_{\substack{2\ell_1 + 2\ell_2 + \ell_3 = L_1 \\ \ell_3 = L_2}} b
	=
	\frac{b}{2} \left(\frac{L_1 - L_2}{2} - 1\right),
	\quad
	\frac{1}{2}\sum_{\substack{2\ell_1 + \ell_2 = L_1 \\ \ell_2 = L_2}} b
	=
	\frac{b}{2},
\end{equation}
for a total of $\frac{b}{4}(L_1-1)$. The case $L_2 > L_1$ is symmetric, giving $\frac{b}{4}(L_2-1)$.

Thus, for generic $(L_1,L_2)$, we find
\begin{equation}
	N_{\frac{1}{2},2}(L_1,L_2;b)
	=
	\frac{1 + (-1)^{L_1+L_2}}{2} \,
	b \frac{\max(L_1,L_2) - 1}{4} .
\end{equation}

\subsection*{One-holed torus and Klein bottle}
Fix $L_1 \in \Z_{>0}$, which must be even to yield a non-trivial contribution. The first two graphs in \cref{ex:chi1} are drawn on a torus, giving
\begin{equation}
	N_{\textup{torus}}(L_1)
	=
	\frac{1}{12} \sum_{\ell_1 + \ell_2 + \ell_3 = \frac{L_1}{2}} 1
	+
	\frac{1}{8} \sum_{\ell_1 + \ell_2 = \frac{L_1}{2}} 1
	=
	\frac{L_1^2 -6L_1 + 8}{96}
	+
	\frac{L_1-2}{16}
	=
	\frac{L_1^2 - 4}{96} .
\end{equation}
As for the last four graphs in \cref{ex:chi1}, they are drawn on a Klein bottle (KB). Interestingly, the MON is not constant in this case (cf. \cref{app:MON}), yielding
\begin{equation}
\begin{split}
	N_{\textup{KB}}(L_1)
	&=
	\frac{1}{4}
	\Bigg(
		\sum_{\ell_1 + \ell_2 + \ell_3 = \frac{L_1}{2}} \frac{2(\ell_1 + \ell_2)b^2 + 2\ell_3 b}{L_1}
		+
		\sum_{\ell_1 + \ell_2 + \ell_3 = \frac{L_1}{2}} b^2
		+
		\sum_{\ell_1 + \ell_2 = \frac{L_1}{2}} \frac{2\ell_1 b^2 + 2\ell_2 b}{L_1}
		+
		\sum_{\ell_1 + \ell_2 = \frac{L_1}{2}} b^2
	\Bigg) \\
	&=
	\frac{b(L_1^2 - 4) + b^2(5L_1^2 -12L_1 + 4)}{96}.
\end{split}
\end{equation}
Altogether, taking into account the parity condition, we find
\begin{equation}
	N_{1,1}(L_1;b)
	=
	\frac{1 + (-1)^{L_1}}{2}
	\frac{
		(1+b)(L_1^2 - 4)
		+
		b^2 (5L_1^2 - 12L_1 + 4)
	}{96}.
\end{equation}

\section{Gaussian \texorpdfstring{$\beta$}{β}-ensemble}
\label{app:GbetaE}

This appendix recalls the Gaussian $\beta$-ensemble, its connected correlators, and their $1/N$ expansion. Fix $\beta>0$. The \emph{Gaussian $\beta$-ensemble} (G$\beta$E) is the probability measure $\mu_{N,\beta}$ on $\R^N$ given by
\begin{equation}\label{eq:GbetaE}
	d\mu_{N,\beta}
	\coloneqq
	\frac{1}{Z_{N,\beta}}
	\prod_{1\le r<s\le N}|\lambda_r-\lambda_s|^{2\beta}
	\prod_{r=1}^N e^{-N\beta\,\frac{\lambda_r^2}{2}}\, d\lambda_r \,,
\end{equation}
where $Z_{N,\beta}$ is a normalisation constant such that $\int_{\R^N} d\mu_{N,\beta} = 1$. With this choice of parameters, $\beta=1$ coincides with the eigenvalue distribution of the Gaussian unitary ensemble (Hermitian matrices), while $\beta=\tfrac12$ and $\beta=2$ correspond to the Gaussian orthogonal and symplectic ensembles, respectively.

For polynomial functions $f_1,\ldots,f_n$ on $\R^N$ (observables), the \emph{connected correlator} with respect to $\mu_{N,\beta}$ is
\begin{equation}\label{eq:cumulant}
	\Braket{f_1,\ldots,f_n}^{\mathrm{G}\beta\mathrm{E}}
	\coloneqq
	\sum_{\pi} (-1)^{|\pi|-1} (|\pi|-1)!\,
	\prod_{I \in \pi} \int_{\R^N} \biggl( \prod_{i\in I} f_i \biggr)\,d\mu_{N,\beta},
\end{equation}
where the sum is over set-partitions $\pi$ of $\{1,\ldots,n\}$. We are interested in the connected correlators of the \emph{power-sum observables} $p_k(\lambda)\coloneqq \sum_{r=1}^N \lambda_r^{k}$. Their $1/N$ expansion admits a topological (half-integer genus) expansion \cite{Oko97,BG13}:
\begin{equation}\label{eq:topological:expansion}
	\Braket{p_{k_1},\ldots,p_{k_n}}^{\mathrm{G}\beta\mathrm{E}}
	=
	\beta^{-\frac{n}{2}}
	\sum_{g \in \frac{1}{2}\Z_{\ge 0}} (\beta^{\frac{1}{2}}N)^{2-2g-n}\,
		\Braket{p_{k_1},\ldots,p_{k_n}}^{\mathrm{G}\beta\mathrm{E}}_g,
\end{equation}
and the genus-$g$ correlators are conveniently repackaged into the resolvent-type generating function
\begin{equation}\label{eq:GbetaE:resolvent}
	\omega_{g,n}^{\mathrm{G}\beta\mathrm{E}}(x_1,\ldots,x_n)
	\coloneqq
	\sum_{k_1,\ldots,k_n > 0}
	\Braket{p_{k_1},\ldots,p_{k_n}}^{\mathrm{G}\beta\mathrm{E}}_g
	\left( \prod_{i=1}^n \frac{dx_i}{x_i^{k_i+1}} \right).
\end{equation}
A result of \cite{CDO26} shows that refined topological recursion computes these G$\beta$E resolvent. Let $\omega^{\mathrm{Web}}_{g,n}$ be the refined topological recursion correlators on the Weber curve (cf. \cref{def:Weber:Airy}) with parameter $\mu=-1$.

\begin{theorem}[{G$\beta$E and refined topological recursion \cite[Theorem A.3]{CDO26}}]\label{thm:GbetaE-RTR}
	After the substitution $x_i=x(z_i)$, the genus-$g$ G$\beta$E resolvent differential $\omega^{\mathrm{G}\beta\mathrm{E}}_{g,n}(x(z_1),\ldots,x(z_n))$ extends to a meromorphic multidifferential on $(\P^1)^n$ and is computed by refined topological recursion:
	\begin{equation}\label{eq:GbetaE-RTR}
		\omega^{\mathrm{G}\beta\mathrm{E}}_{g,n}\bigl(x(z_1),\ldots,x(z_n)\bigr)
		=
		\omega^{\mathrm{Web}}_{g,n}(z_1,\ldots,z_n),
	\end{equation}
	under the identification of refinement parameters $\mf{b}=\beta^{1/2}-\beta^{-1/2}$\footnote{The parameter $\beta$ in \cite[Appendix A]{CDO26} is twice the $\beta$ used here.}.
\end{theorem}

We now introduce a family of observables adapted to pruning \cite{GMM}. Let $T_k$ be the Chebyshev polynomials of the first kind, and define the (monic) \emph{Chebyshev observables}
\begin{equation}\label{eq:Cheb-observable}
	t_k(\lambda) \coloneqq 2 \sum_{r=1}^N T_k\left(\frac{\lambda_r}{2}\right),
	\qquad
	k > 0.
\end{equation}
Using the explicit expansion $2\,T_k(\frac{\lambda}{2})=\sum_{m=0}^{\lfloor k/2\rfloor}(-1)^m\frac{k}{k-m}\binom{k-m}{m}\,\lambda^{k-2m}$, we obtain the change of basis
\begin{equation}\label{eq:tk-pk}
	t_k
	=
	\sum_{m=0}^{\lfloor k/2\rfloor}
	(-1)^m\frac{k}{k-m}\binom{k-m}{m}\,p_{k-2m}.
\end{equation}
Since \cref{eq:tk-pk} is triangular in $k$, it can be inverted uniquely. We therefore define the \emph{pruned connected correlators} to be the connected correlators of the Chebyshev observables, namely $\Braket{t_{k_1},\ldots,t_{k_n}}^{\mathrm{G}\beta\mathrm{E}}_g$. With this convention, the relation between usual and pruned correlators reads
\begin{equation}\label{eq:unpruning}
	\Braket{t_{k_1},\ldots,t_{k_n}}^{\mathrm{G}\beta\mathrm{E}}_g
	=
	\sum_{\substack{m_1,\ldots,m_n\ge 0\\ k_i-2m_i>0}}
	\left(\prod_{i=1}^n (-1)^{m_i}\,\frac{k_i}{k_i-m_i}\binom{k_i-m_i}{m_i}\right)
	\Braket{p_{k_1-2m_1},\ldots,p_{k_n-2m_n}}^{\mathrm{G}\beta\mathrm{E}}_g ,
\end{equation}
and, by inversion,
\begin{equation}\label{eq:pruning}
	\Braket{p_{k_1},\ldots,p_{k_n}}^{\mathrm{G}\beta\mathrm{E}}_g
	=
	\sum_{\substack{m_1,\ldots,m_n\ge 0 \\ k_i-2m_i>0}}
	\Braket{t_{k_1-2m_1},\ldots,t_{k_n-2m_n}}^{\mathrm{G}\beta\mathrm{E}}_g
	\prod_{i=1}^n \binom{k_i}{m_i} .
\end{equation}
The interpretation of \cref{eq:pruning} is the usual one in terms of Feynman diagrams: a one-valent vertex (a \emph{petal}) attached to the $i$th boundary component contributes $2$ to its boundary degree, and the factor $\binom{k_i}{m_i}$ counts the choice of $m_i$ attachment sites among the $k_i$ boundary corners, independently for each $i$. With this definition in place, we can state the relation between the refined lattice point count and the pruned G$\beta$E correlators.

\begin{proposition}\label{prop:Ngn-pruned}
	Under the identification of refinement parameters $\beta = \frac{1}{1+b}$, the refined lattice point counts $N_{g,n}$ coincide with the pruned genus-$g$ G$\beta$E correlators, up to an overall combinatorial factor:
	\begin{equation}\label{eq:Ngn-pruned}
		N_{g,n}(L_1,\ldots,L_n)
		=
		\frac{1}{2\beta^g} \,
		\Braket{\frac{t_{L_1}}{L_1},\ldots,\frac{t_{L_n}}{L_n}}^{\mathrm{G}\beta\mathrm{E}}_g .
	\end{equation}
\end{proposition}

\begin{proof}
	By \cref{thm:GbetaE-RTR} we have, after the substitution $x_i=x(z_i)$, the identity
	$\omega^{\mathrm{G}\beta\mathrm{E}}_{g,n}(x(z_1),\ldots,x(z_n))=\omega^{\mathrm{Web}}_{g,n}(z_1,\ldots,z_n)$.
	Expanding both sides at $z_i=0$ and using the defining expansions
	\begin{equation}
		\omega^{\mathrm{Web}}_{g,n}(z_1,\ldots,z_n)
		=
		(-1)^{n} \frac{2}{(1+b)^g}
		\sum_{L_1,\ldots,L_n>0}
		N_{g,n}(L_1,\ldots,L_n)\prod_{i=1}^n L_i\,z_i^{L_i-1}\,dz_i
	\end{equation}
	and, using the expansion
	\begin{equation}\label{eq:dx-over-x-expansion}
		\frac{dx(z)}{x(z)^{k+1}}
		=
		- \sum_{\substack{L \ge k \\ L\equiv k\;(\mathrm{mod}\,2)}}
			(-1)^{\frac{L-k}{2}}\, \frac{2L}{L+k}\, \binom{\frac{L+k}{2}}{\frac{L-k}{2}} \,z^{L-1}\,dz,
	\end{equation}
	together with \cref{eq:unpruning}, we find
	\begin{equation}
		\omega^{\mathrm{G}\beta\mathrm{E}}_{g,n}\bigl(x(z_1),\ldots,x(z_n)\bigr)
		=
		(-1)^n
		\sum_{L_1,\ldots,L_n>0}
		\Braket{\frac{t_{L_1}}{L_1},\ldots,\frac{t_{L_n}}{L_n}}^{\mathrm{G}\beta\mathrm{E}}_g
		\prod_{i=1}^n L_i\,z_i^{L_i-1}\,dz_i.
	\end{equation}
	Extracting the coefficients of $\prod_i L_i z_i^{L_i-1}dz_i$ gives \cref{eq:Ngn-pruned} under the identification $\beta = \frac{1}{1+b}$.
\end{proof}

\printbibliography

\end{document}